\documentclass[11pt]{article}

\usepackage[utf8]{inputenc} 
\usepackage[T1]{fontenc}
\usepackage{amsmath,amssymb,amsthm}
\usepackage{amsfonts}
\usepackage{url}
\usepackage{ifthen}
\usepackage{graphicx}
\usepackage{xcolor}
\usepackage{calc}
\usepackage{bbm}
\usepackage{dsfont}
\usepackage{centernot}
\usepackage{bm}
\usepackage{enumitem}
\usepackage[colorlinks=false]{hyperref}
\usepackage{authblk}
\usepackage{mathtools}
\usepackage{setspace}
\usepackage{subfigure}
\usepackage{authblk}
\usepackage{algorithm} 
\usepackage{algpseudocode} 
\usepackage{comment} 
\usepackage[margin=1in]{geometry}
\usepackage{minitoc}
\usepackage[toc,page,header]{appendix}

\usepackage{times}
\usepackage[bigdelims,slantedGreek,vvarbb,scaled=1.07]{newtxmath}

\usepackage[backend=bibtex,style=authoryear,natbib=true]{biblatex}

\newtheorem{theorem}{Theorem}
\newtheorem*{theorem*}{Theorem}

\newtheorem{lemma}{Lemma}

\numberwithin{claim}{section}

\theoremstyle{remark}
\newtheorem{remark}{Remark}

\newtheorem{definition}{Definition}

\newtheorem*{definition*}{Definition}
\usepackage{xcolor}
\definecolor{p1}{HTML}{97C697}
\definecolor{p2c}{HTML}{FCF6F7}
\definecolor{p2l}{HTML}{EFF5FB}
\definecolor{cred}{HTML}{F2CCD2}
\definecolor{cgreen}{HTML}{EBFCEB}
\definecolor{grey}{HTML}{BDC3C6}

\usepackage{tikz}
\usepackage{array}
\usepackage{tabularx}
\usepackage{multirow} 
\usepackage{booktabs}
\definecolor{applegreen}{rgb}{0.55, 0.71, 0.0}

\DeclareMathOperator*{\argmin}{\arg\!\min}

\newcommand{\m}{\mathcal}

\newcommand{\RR}{\mathbb{R}}

\newcommand{\CC}{\mathbb{C}}

\newcommand{\diag}{\mathop{\mathrm{diag}}}
\newcommand{\rank}{\mathop{\mathrm{rank}}}
\newcommand{\sign}{\mathop{\mathrm{sign}}}
\newcommand{\tr}{\mathop{\mathrm{tr}}}
\newcommand{\Range}{\mathop{\mathrm{Range}}}

\newcommand{\Expt}{\mathbb{E}}
\newcommand{\eps}{\varepsilon}

\newcommand{\T}{\top}
\newcommand{\IIDDist}{ \overset{\mathrm{i.i.d.}}{\sim} }
\newcommand{\Indic}[1]{\mathds{1}_{#1}}
\newcommand{\BigOh}{O}

\newcommand{\MatL}{\begin{bmatrix}}
\newcommand{\MatR}{\end{bmatrix}}
\newcommand{\iu}{\mathsf{i}}

\newcommand{\Stiel}{{\mathsf{s}}}
\newcommand{\StielEmp}{{\hat{\mathsf{s}}}}
\newcommand{\StielInv}{\Psi}

\newcommand{\SGconst}{K}
\newcommand{\EffDim}{\mathsf{d}_H}
\newcommand{\DimEff}[1]{\mathsf{d}_H(#1)}
\newcommand{\Pm}{P_m}
\newcommand{\PmMinK}{P_m^{-k}}

\newcommand{\hl}{\hat{\lambda}}

\newcommand{\hW}{\hat{W}}

\newcommand{\Tell}{T_{\setminus \ell}}
\newcommand{\Tk}{T_{\setminus k}}
\newcommand{\Hell}{H_{\setminus \ell}}

\newcommand{\AssumAsymp}{\textbf{[Asymp($\tau,H,\xi$)]}}
\newcommand{\AssumAsympZero}{\textbf{[Asymp($\tau,H,m_0,\xi$)]}}

\newcommand{\ProcChooseMIID}{{\it Choose-Sketching-Dimension-IID}}
\newcommand{\ProcChooseLambdaIID}{ {\it Estimate-Inverse-Hessian-IID} }
\newcommand{\ProcLineSearch}{ {\it LineSearch} }

\newcommand{\sfH}{\mathsf{H}}
\newcommand{\sfT}{\mathsf{T}}
\newcommand{\sfN}{\mathsf{N}}

\title{Newton Meets Marchenko-Pastur: Massively Parallel Second-Order Optimization with Hessian Sketching and Debiasing}
\author[1]{Elad Romanov \thanks{\texttt{eromanov@stanford.edu}}}
\author[2]{Fangzhao Zhang \thanks{\texttt{zfzhao@stanford.edu}}}
\author[2]{Mert Pilanci \thanks{\texttt{pilanci@stanford.edu}}}

\affil[1]{Department of Statistics, Stanford University}
\affil[2]{Department of Electrical Engineering, Stanford University}

\date{}

\renewcommand \thepart{}
\renewcommand \partname{}

\begin{document}

	\doparttoc 
	\faketableofcontents 

    \maketitle
    
    \begin{abstract}
    	Motivated by recent advances in serverless cloud computing, in particular the ``function as a service'' (FaaS) model, 
we consider the problem of minimizing a convex function in a massively parallel fashion, where communication between workers is limited.
Focusing on the case of a twice-differentiable objective subject to an L2 penalty, we propose a scheme where the central node (server) effectively runs a Newton method, 
offloading its high per-iteration cost---stemming from the need to invert the Hessian---to the workers. 
In our solution, workers produce independently coarse but low-bias estimates of the inverse Hessian, using an adaptive sketching scheme. The server then averages the descent directions produced by the workers, yielding a good approximation for the exact Newton step. The main component of our adaptive sketching scheme is a low-complexity procedure for selecting the sketching dimension, an issue that was left largely unaddressed in the existing literature on Hessian sketching for distributed optimization. Our solution is based on ideas from asymptotic random matrix theory, specifically the Marchenko-Pastur law. For Gaussian sketching matrices, we derive non asymptotic guarantees for our algorithm which are essentially dimension-free. Lastly, when the objective is self-concordant, we provide convergence guarantees for the approximate Newton's method with noisy Hessians, which may be of independent interest beyond the setting considered in this paper. 
    \end{abstract}

	\section{Introduction}
	Consider minimizing a convex, twice-differentiable function $F:\RR^d \to \RR$, 
subject to an L2 regularization penalty:
\begin{equation}\label{eq:Convex-Problem}
	 \min_{\theta \in \RR^d} F(\theta) + \frac{\lambda}{2} \|\theta\|^2 \,,
	\end{equation} 
where $\lambda>0$.
Such problems frequently appear in machine learning and statistics. Common examples include ridge regression, logistic regression with regularization, support vector machines (SVMs) and kernel machines, among others \citep{hastie2009elements}.
	
Consider a scenario where a large number of workers (processors) are available to collaboratively solve \eqref{eq:Convex-Problem}, with the assumptions that each worker (i) has full access to the objective function;
(ii) operates with limited individual computational resources; and (iii) cannot communicate directly with other workers, 
only with a central server, with the latter responsible for orchestrating the work of the workers (a ``star'' network topology). This setting is motivated by recent developments in serverless cloud computing, particularly the "function as a service" (FaaS) model  \citep{jonas2017occupy}. Serverless computing and FaaS are well-suited for computing gradients and Hessians on large datasets, enabling parallelization of these computations across a large number of workers. Additionally, this approach provides resilience to failed or straggler workers, ensuring robust and efficient processing \citep{bartan2019straggler}.
In designing a solution of this setting, we are essentially guided by two key principles:
(a) minimizing the number of communication or worker deployment rounds, with each round utilizing a large number of workers in parallel; and (b) reducing the computational load on individual workers by distributing smaller tasks across many workers concurrently.

This paper proposes and analyzes a 
scheme for solving
\eqref{eq:Convex-Problem} in a massively parallel manner.
In our scheme, the server aims to solve \eqref{eq:Convex-Problem} using Newton's method, a second-order iterative algorithm which is known to converge extremely fast for sufficiently smooth and strongly convex objectives (quadratic convergence rate).
Its fast convergence is attained by incorporating curvature (Hessian) information when searching for a descent direction, and specifically requires an inverse Hessian-gradient product at every iteration. This a priori results in massive per-iteration computational cost: exactly computing the descent direction (assuming the Hessian is accessible) costs practically $O(d^3)$ flops; in high-dimensional problems (large $d$), this may be prohibitively expensive.

We wish to offset the high per-iteration cost of Newton's method, by deploying many workers in parallel to collaboratively compute the Newton direction.
 In the setting we consider, workers cannot easily communicate with one another, and therefore directly inverting the Hessian by is not feasible. 
In our scheme,
each worker independently computes a coarse but \textbf{unbiased} estimate of the exact Newton direction, using randomized sketching. These estimates are then aggregated (averaged) at the server, producing a good approximation of the true Newton step.

\paragraph{Newton's method:}
The exact Newton method for minimizing \eqref{eq:Convex-Problem} has the following form:
\begin{align}
	\theta_{t} = \theta_{t-1} - \alpha_t W_t g_t,\qquad (t=1,\ldots,T)\,, \label{eq:Newton-Exact-0}
\end{align}
where $\alpha_t$ is a step size (typically chosen by line search) and the gradient and inverse Hessian are
\begin{align}
	g_t &:= \nabla F(\theta_{t-1}) + \lambda\theta_{t-1}, \label{eq:Newton-Exact-1}\\
	W_t &:= (H_t + \lambda I)^{-1},\; H_t := \nabla^2 F(\theta_{t-1}) \label{eq:Newton-Exact-2}\;.
\end{align}
For sufficiently smooth convex objectives $F$, Newton's method is known to converge at a quadratic rate: $T=O(\log\log(1/\eps))$ iterations are sufficient to approach the minimum within error $\eps$, as opposed to $T=O(\log(1/\eps))$ for gradient descent and its accelerated variants; see \cite[Chapter 9]{boyd2004convex}. A fast convergence rate is highly desirable in our setting, as the number of iterations corresponds directly to communication/worker deployment rounds.

In our scheme, the server effectively runs the Newton method with an approximate inverse Hessian: 
\begin{equation}
	\theta_t = \theta_{t-1} - \alpha_t \bar{W}_t g_t, 
\end{equation}
where $\bar{W}$ is obtained by averaging $q$ independent local estimates produced by the workers:\footnote{In practice, the workers send the vector $\hat{W}^{(k)}g_t \in \RR^{d}$ to the server, rather the whole matrix $\hat{W}^{(k)} \in \RR^{d\times d}$.}
\begin{equation}
	\bar{W} := \frac{1}{q}\sum_{k=1}^q \hat{W}_t^{(k)} .
\end{equation}

\paragraph{Debiased inverse Hessian sketching:}
We wish to avoid directly inverting the full $d\times d$ Hessian, which practically (when done exactly) costs $O(d^3)$. To this end, we have every worker compute, in parallel, a cheap but low-bias estimate of the inverse Hessian by sketching.
Specifically, at every deployment round, every worker independently samples a random sketching matrix $S\in \RR^{m\times d}$ (where $m<d$), and approximates the inverse Hessian by its sketched version
\begin{equation} \label{eq:feature-sketch-Def}
	\hat{W} = S^\T (SHS^\T + \tilde{\lambda}I)^{-1} S ,
\end{equation} 
where $\tilde{\lambda}>0$ is a \emph{modified} regularization parameter. Sketches of the type in \eqref{eq:feature-sketch-Def} have been studied before in the literature (also in the context of high-dimensional regression and optimization). Variations have appeared under various names: ``dual sketching'' {\citep{zhang2013recovering}}, ``sketch-and-project'' {\citep{gower2015randomized}}, ``right-sketch'' \citep{murray2023randomized}, and ``feature sketching'' {\citep{patil2023asymptotically}}.
Note that the cost of forming $\hat{W}g$ is $O(m^3+M)$, where $M$ is the cost of multiplying $SHS^\T$. 
In this paper, we exclusively consider the case where $S$ is a random dense matrix with i.i.d. entries.\footnote{By ``dense'', we mean that $S$ has $\Omega(md)$ many non-zero entries typically, though it's not necessary that all the entries, or even a majority of them, are non-zero.} In this case $M=d^2 m$, which dominates the overall per-iteration cost. Naturally, we are particularly interested in settings where $m$ can be chosen very small compared to $d$, so that cost, $O(d^2m)$, is substantially smaller than $O(d^3)$.

\begin{remark}
	In this paper we {do not} assume any particular form for $H$ (i.e., that it is readily factorizable $H=X^\T X$, $X\in \RR^{N\times d}$, a setting which is quite common in the sketching and distributed optimization literature; for example {\citep{pilanci2015newton}}). In fact, sketched measurements $SHS^\T$ may be obtainable even in settings when $F$ is not available in analytic form, but only through a program (oracle) that calculates it - for example by automatic differentiation  {\citep{paszke2017automatic}}.
\end{remark}

A key question is how to choose $m$ (and in accordance $\tilde{\lambda}$). Define the $\lambda$- effective dimension of $H$,\footnote{$\EffDim(\lambda)$ is also called the ``effective degrees of freedom'' in ridge regression \citep{hastie2009elements}, when the design matrix $X\in\RR^{N\times d}$ is $H=X^\T X$. }
\begin{equation} 
	\EffDim(\lambda) := \tr(H(H + \lambda I)^{-1}).
\end{equation} 
As we shall see in Section~\ref{sec:RMT-Background}, results from asymptotic random matrix theory imply that, when $S$ is an i.i.d. dense sketching matrix and \textbf{provided} that $m>\EffDim(\lambda)$, the choice
\begin{equation}\label{eq:Oracle-Reg}
	\tilde{\lambda} := \lambda\left( 1 - \frac{1}{m} \EffDim(\lambda) \right)
\end{equation}
results in $\hat{W}$ which is a \emph{low-bias} estimator of the true inverse Hessian $W$. Note that exactly computing the effective dimension is resource-intensive, and essentially requires inverting the Hessian---the very operation we wanted to avoid. A key component of our proposed scheme is a novel low-complexity procedure, based on ideas from asymptotic random matrix theory, for adaptively selecting $m=O(\EffDim(\lambda))$ and $\tilde{\lambda}$, by only observing sketched measurements of $H$. 
\begin{remark}
	When $H$ is effectively low-rank, that is exhibits fast spectral decay, $\EffDim(\lambda)$ may be much smaller than $d$. For example, when the spectrum has power decay $\lambda_k(H)\propto k^{-\alpha}$ and $\lambda=O(1)$, $\EffDim(\lambda)=O(d^{1-\alpha})$ when $\alpha<1$, $\EffDim(\lambda)=O(\log d)$ when $\alpha=1$ and $\EffDim(\lambda)=O(1)$ when $\alpha>1$. Hessian spectral with power decay are common in regression problems involving kernels {\citep{wainwright2019high, yang2017randomized}}.
\end{remark}

\paragraph{Outline and Contributions:} 
\begin{itemize}[leftmargin=*,noitemsep,topsep=0pt]
	\item 
	In Section~\ref{sec:RMT-Background} we review fundamental results from asymptotic random matrix theory (RMT), specifically the Marchenko-Pastur law and its formulation using deterministic equivalents, and demonstrate how they lead naturally to an asymptotically unbiased inverse Hessian estimator, from i.i.d. sketching matrices $S$. The connection to the Marchenko-Pastur law treats in a uniform framework---and recovers almost immediately---earlier results on Hessian sketch debiasing for distributed Newton methods 
	\citep{derezinski2020debiasing,zhang2023optimal}.
	\item 
	In Section~\ref{sec:Procedure} we propose a \emph{novel data-driven method} for selecting $m$ and $\tilde{\lambda}$ using only the sketched Hessian, avoiding the need to compute the effective dimension exactly. The approach is a natural outgrowth of the RMT framework described previously, and is fully adaptive: no tuning parameters are needed.
	Section~\ref{sec:Nonasymptotic} further provides {non-asymptotic} error guarantees for our method, focusing on the case where $S$ is an i.i.d. {Gaussian} matrix. Notably, the error bounds obtained in this setting are essentially dimension-free and independent of the Hessian’s condition number.
	\item 
	In Section~\ref{sec:Quasi-Newton} we derive convergence guarantees for the Newton method with inexact Hessians; to wit, when the exact inverse Hessian $W$ in \eqref{eq:Newton-Exact-0}-\eqref{eq:Newton-Exact-2} is replaced by an $\eta$-accurate estimate $\bar{W}$. By ensuring a sufficiently small $\eta$, the algorithm achieves an arbitrarily fast linear convergence rate. Our results are proved for self-concordant functions $F(\cdot)$, and to our knowledge are novel and may be of independent interest. Combined with the results of the previous section, we obtain a non-asymptotic, end-to-end convergence guarantee for the entire parallel method (with Gaussian sketches), which remarkably does not depend on the condition number of the Hessians.
	 \item 
	  Lastly, Section~\ref{sec:Experiments} is dedicated to experiments, on both synthetic and real-world data.
\end{itemize}

	\subsection{Related Works}
	Distributed second-order optimization methods have been widely studied for large-scale machine learning. Notable examples include DANE \citep{shamir2013communication}, which addresses communication efficiency, AIDE \citep{reddi2016aide}, designed for accelerated convergence, and DiSCO \citep{zhang2015disco}, which focuses on distributed second-order methods for convex optimization problems. The most closely related body of literature to our work centers around the averaging of inexact Newton steps, which has proven effective when combined with randomized linear algebra. GIANT \citep{wang2017giant} uses the average of local Newton steps as the global step and Determinantal Averaging \citep{derezinski2020debiasing} provides an unbiased averaging technique for distributed Newton methods. Surrogate sketching \citep{derezinski2020debiasing} was introduced for distributed Newton's method, revealing improvements via a simple shrinkage strategy.
The work \citep{zhang2023optimal} introduced an optimal debiasing method for distributed Newton's method when the effective dimension of the Hessian is known.
Lastly, there exists a large related corpus of works in optimization involving Hessian subsampling or sketching, for cases where the Hessian admits a good low-rank approximation, though not necessarily in a distributed setting. Examples include \citep{roosta2019sub,
	gower2019rsnrandomizedsubspacenewton,frangella2022sketchysgd}, among others.
	
	\section{Debiasing from I.I.D Sketches and Random Matrix Theory}
	\label{sec:RMT-Background}
	This section concerns inverse sketched Hessian debiasing for sketches $S\in \RR^{m\times d}$ which have i.i.d. entries with mean zero $\Expt[S_{ij}]=0$. We normalize the variance as $\Expt[S_{ij}]^2=1/m$; this normalization is such that $\Expt[S^\T S]=I$, equivalently, $S:\RR^d \to \RR^m$ is an isometry in expectation.

Let $H\succeq 0$ be a positive definite matrix (e.g., the full Hessian at the current iteration). The spectral properties of the matrix $SHS^\T \in \RR^{m\times m}$, which appears in the sketch \eqref{eq:feature-sketch-Def}, are of utmost importance for what follows. Note that this matrix has the same non-zero eigenvalues\footnote{This is a standard consequence of the elementary determinant identity $\det(I+AB)=\det(I+BA)$.} as the matrix $H^{1/2}S^\T S H^{1/2}$: the latter is a \emph{sample covariance} matrix (making the analogy explicit: $m$ is the number of samples, $H$ is the population covariance, and $x_i:=\sqrt{m}H^{1/2}r_i$ are independent ``samples'', $r_1,\ldots,r_m$ being the rows of $S$); accordingly, $SHS^\T$ is sometimes called the companion sample covariance. 
Sample covariance matrices, and their spectral properties, have been a central object of investigation in random matrix theory (RMT); cf. \citep{bai2010spectral}. A foundational result is the {Marchenko-Pastur law} \citep{marvcenko1967distribution}, 
which describes the high-dimensional global behavior of the eigenvalues of $H^{1/2}S^\T S H^{1/2}$ (limiting spectral distribution). 

\paragraph{The Marchenko-Pastur equation:} For\footnote{Here $\CC^+$ is the complex upper halfplane, $\Im(z)>0$.} $z\in \CC^+$ let $\Stiel(z)\in\CC^+$ be the unique solution of
\begin{equation}\label{eq:MP-Eq}
	\frac{1}{\Stiel(z)} = - z + \frac1m \tr\left( H(I+\Stiel(z)H)^{-1} \right) 
\end{equation}
subject to the constraint $\Stiel(z)\in \CC^+$. It is known that (see \citep{bai2010spectral}) $\Stiel(z)$ is the Stieltjes transform\footnote{For a finite measure $\mu$ on $\RR$, its Stieltjes transform is the function $\Stiel_\mu:\CC^+\to \CC^+$ given by $\Stiel_\mu(z)=\int \frac{1}{x-z}d\mu(x)$. Using the Stieltjes inversion formula, the measure $\mu$ can be recovered from $\Stiel_\mu$ in the following sense: for every $\varphi(\cdot)$ bounded and continuous, $\int \varphi(t)d\mu(t) = \lim_{\eta\to 0+}\frac1\pi \int \varphi(x)\Im\Stiel_\mu(x+\iu \eta)dx$. } of a \emph{compactly supported probability measure}, whose support lies in $\RR_{\ge 0}$; this measure is the \textbf{(companion) Marchenko-Pastur law} associated with the population covariance $H$ (and depends only on its eigenvalues) and the number of samples $m$. Moreover: 1) The function $z\mapsto \Stiel(z)$ extends continuously to $\RR$, except possibly at $z=0$; 2) for $x\in \RR$ outside the support, the function $x\mapsto \Stiel(x)$ is increasing; in particular, for negative $x<0$, it is positive and increasing. In this paper, we shall restrict our attention to only negative real valued $z$ in  \eqref{eq:MP-Eq}, understanding that it holds also for $z<0$ real. For more details, cf. \citep{silverstein1995analysis,bai2010spectral}.
Finally, note that the function $z\mapsto \Stiel(z)$ is invertible, the inverse function given by 
\begin{equation}\label{eq:Psi-Def}
	\Psi(\Stiel) = \frac{1}{m}\tr(H(I+\Stiel H)^{-1})- \frac{1}{\Stiel}.
\end{equation}
The Marchenko-Pastur equation \eqref{eq:MP-Eq} may equivalently be written as $z=\Psi(\Stiel(z))$.

In RMT, it is typical to consider the \emph{proportional growth} asymptotic regime: $m,d\to\infty$, $m=\Theta(d)$. 
\begin{definition}
	For (small) constant $\tau>0$, assumption set \AssumAsymp~ holds if: (1) bounded aspect ratio (proportional growth): $\tau<\frac{m}{d}<1/\tau$; (2) bounded operator norm: $\|H\|<1/\tau$; (3) $H$ is invertible; (4) $H$ has non-degenerate spectrum: at least $\tau d$ eigenvalues are larger than $\tau$.(5) $\sqrt{m}S_{ij}\sim \xi$, with their law $\xi$, satisfying: $\Expt[\xi]=0$, $\Expt[\xi^2]=1$, and $\Expt[|\xi|^k]<\infty $ for all $k>0$.
%
%
\label{def:Asymp}
\end{definition}

Denote by $\StielEmp(z)$ the Stieltjes transform of the empirical eigenvalue distribution of $SHS^\T$:
\begin{equation}\label{eq:StielEmp-Def}
	\StielEmp(z) = \frac1m \tr (SHS^\T -z I)^{-1} = \frac1m \sum_{i=1}^m \frac{1}{\lambda_i(SHS^\T)-z} .
\end{equation}
Note that $\StielEmp(z)$ is a random function, since $S$ is random. The following is a (quantitative) statement of the Marchenko-Pastur law \citep{knowles2017anisotropic}; for simplicity, we focus on the case $z<0$.
\begin{theorem}\label{thm:MP-Law}
	Set $\eta,D>0$. Assume \AssumAsymp.  W.p. $1-O(m^{-D})$, simultaneously for all $-1/\tau \le z \le -\tau$,
	\begin{equation}
		|\StielEmp(z)-\Stiel(z)| = O\left( m^{-1/2+\eta} \right).
	\end{equation}
	The constants in the $O(\cdot)$ notation depend on $\tau,\eta,D$.
\end{theorem}
\begin{proof}
	See Theorem 3.7 and Remark 3.10 in \citep{knowles2017anisotropic}.
\end{proof}

\paragraph{Deterministic equivalent for the sample covariance:}
The Marchenko-Pastur law (Theorem~\ref{thm:MP-Law}) describes the eigenvalue distribution of $SHS^\T$, equivalently $H^{1/2}S^\T S H^{1/2}$, and relates it to that of $H$ (the connection made via the Marchenko-Pastur equation \eqref{eq:MP-Eq}). One can show something stronger: the entire (random) resolvent $(H^{1/2}S^\T S H^{1/2}-zI)^{-1}$ is in fact close (entrywise, \emph{not} in operator norm) to the resolvent of $H$ (which is deterministic), with appropriate shifting and shrinkage. 
A result of the following kind is sometimes referred to in the RMT literature as a ``deterministic equivalent'', cf. \citep{couillet2011random,bai2010spectral}. 
\begin{theorem}\label{thm:Determinstic-Cov}
	Set $\eta,D>0$. Assume \AssumAsymp. For any $u,v\in\RR^d$, $\|u\|,\|v\|=1$, the following holds.	
	W.p. $1-O(m^{-D})$, simultaneously for all $-1/\tau^{-1}\le z \le -\tau$,
	\begin{equation}\label{eq:local-law}
		u^\T (H^{1/2} S^\T S H^{1/2} - z I)^{-1}  v = u^\T  (-z\Stiel(z)H - zI)^{-1} v + O(m^{-1/2+\eta}).
	\end{equation}
	Above, the constants in the $O(\cdot)$ notation depend on $\tau,\eta,D$ and are uniform in $u,v$. In particular,
	\begin{equation}
		\left\|\Expt\left[(H^{1/2} S^\T S H^{1/2} - z I)^{-1}\right] - (-z\Stiel(z)H - zI)^{-1}\right\| = O(m^{-1/2+\eta})
	\label{eq:local-law-1}
	\end{equation} 
	\begin{proof}
		See Eq. (3.3) and Theorem 3.7 in \citep{knowles2017anisotropic}.
	\end{proof}
\end{theorem}

\paragraph{Deterministic equivalent for the sketched Hessian:} Recently, a deterministic equivalent for the sketched matrix \eqref{eq:feature-sketch-Def} was obtained by \citep{lejeune2022asymptotics}. We cite a quantitative version of their result \citep[Theorem 4.1]{lejeune2022asymptotics}.
\begin{theorem}\label{thm:Deterministic-Sketch}
	Set $\eta,D>0$. Assume \AssumAsymp, and furthermore  $\lambda_{\rank(H)}(H)\ge \tau$. For any $u,v\in\RR^d$, $\|u\|,\|v\|=1$, w.p. $1-O(m^{-D})$, simultaneously for all $-1/\tau\le z \le -\tau$,
	\begin{equation}\label{eq:lejeune}
		u^\T S^\T (SHS^\T- z I)^{-1}Sv = u^\T  (H - (\Stiel(z))^{-1} I)^{-1} v + O(m^{-1/2+\eta}).
	\end{equation}
	In particular,
	\begin{equation}\label{eq:thm:Deterministic-Sketch-2}
		\left\|\Expt\left[S^\T (SHS^\T- z I)^{-1}S\right] - (H - (\Stiel(z))^{-1} I)^{-1}\right\| = O(m^{-1/2+\eta}).
	\end{equation} 
\end{theorem}
\begin{proof}
	The result is a consequence of Theorem~\ref{thm:Determinstic-Cov}; for completeness, let us repeat the argument of \citep{lejeune2022asymptotics}, assuming $H$ is full rank (see their paper for a generalization). Denote $\tilde{u}:=H^{-1/2}u$, $\tilde{v}:=H^{-1/2}v$, so that $\|\tilde{u}\|,\|\tilde{v}\|\le 1/\sqrt{\tau}=O(1)$. We have w.p. $1-O(m^{-D})$,
	\begin{align*}
		&u^\T S^\T (SHS^\T- z I)^{-1}Sv 
		= \tilde{u}^\T H^{1/2}S^\T (SHS^\T- z I)^{-1}SH^{1/2} \tilde{v} \\
		&\quad= \tilde{u}^\T  (H^{1/2}S^\T S H^{1/2}- z I)^{-1}H^{1/2}S^\T S H^{1/2} \tilde{v} 
		= \tilde{u}^\T \tilde{v} + z \tilde{u}^\T  (H^{1/2}S^\T S H^{1/2}- z I)^{-1} \tilde{v}\\
		&\quad= \tilde{u}^\T \tilde{v} - \tilde{u}^\T (\Stiel(z)H+I)^{-1} \tilde{v} + O(m^{-1/2+\eta}) = u^\T(H+(\Stiel(z))^{-1}I)^{-1} v + O(m^{-1/2+\eta}).
	\end{align*} 
\end{proof}
	
	\section{An Adaptive Sketching and Debiasing Procedure}
	\label{sec:Procedure}
	Theorem~\ref{thm:Deterministic-Sketch} suggests a clear path for obtaining a low-bias estimate of the inverse Hessian from the sketch \eqref{eq:feature-sketch-Def}, namely: one should choose $\tilde{\lambda}>0$ such that $\Stiel(-\tilde{\lambda})=1/\lambda$, \emph{when such a root exists}. 
\begin{lemma}
	A solution $\tilde{\lambda}>0$ satisfying $\Stiel(-\tilde{\lambda})=1/\lambda$ exists if and only if $m>\EffDim(\lambda)$. When this is the case, $\tilde{\lambda}$ is given by \eqref{eq:Oracle-Reg}.
\end{lemma}
\begin{proof}
	By merit of being the Stieltjes transform of a finite measure supported on $\RR_{\ge 0}$, $\Stiel(\cdot)$ maps bijectively $(-\infty,0)$ to $(0,\Stiel(0))$, hence a solution $\tilde{\lambda}>0$ exists if and only if $\Stiel(0)>1/\lambda$. Taking the limit $z\to 0+$ in \eqref{eq:MP-Eq}, we find (after routine algebraic manipulation) that $\Stiel(0)$ solves $\EffDim(1/\Stiel(0))=m$. Since the function $\mu\mapsto \EffDim(\mu)$ is decreasing for positive $\mu$, we deduce that $1/\Stiel(0)<\lambda$ if and only if $m>\EffDim(\lambda)$. When this is the case, $\tilde{\lambda}=\Psi(1/\lambda)$ where $\Psi$ is the explicit inverse \eqref{eq:Psi-Def}.
\end{proof}

\begin{remark}
	Earlier works on distributed Newton's method \citep{derezinski2020debiasing,zhang2023optimal}, considered debiasing for a closely related but different Hessian sketch. Specifically, they assume the Hessian has the form $H=X^\T X$ where $X\in \RR^{N \times d}$; such Hessians naturally appear in machine learning problems,
	for example in the (reguarlized) empirical loss minimization for a generalized linear model (GLM). 
	\citep{zhang2023optimal} considered (among others), the sketch $X^\T S^\T S X$, and calculated the appropriate shrinkage factor for asymptotically debiasing the inverse: $(\gamma X^\T S^\T S X + \lambda I)^{-1} \simeq (H+\lambda I)^{-1}$ where $\gamma:= 1/\left( 1-\frac1m\EffDim(\lambda) \right)$, provided that $m>\EffDim(\lambda)$, under an asymptotic setting as in Theorem~\ref{thm:Determinstic-Cov} (though without making the connection to the Marchenko-Pastur law explicit). In the scheme they considered, the effective dimension was assumed to be known. We remark that Theorem~\ref{thm:Determinstic-Cov} holds verbatim if $H^{1/2}$ is replaced by $X^\T$, and the ratio $\tau<d/N<1/\tau$ is bounded \citep{knowles2017anisotropic}. Rewritting \eqref{eq:local-law-1} as $(H+(\Stiel(z))^{-1}I)^{-1}\simeq \left( (-z\Stiel(z))^{-1}X^\T S^\T S X + (\Stiel(z))^{-1}I \right)^{-1}$, setting $z=-\tilde{\lambda}$ recovers their proposed shrinkage factor. Accordingly, the adaptive procedure we propose below applies to that sketching scheme as well, with the obvious modifications.
\end{remark}

Note that we do not have access to the function $\Stiel(\cdot)$ directly (to calculate it numerically via \eqref{eq:MP-Eq}, one needs to invert the full Hessian)---but we \emph{can} estimate it from the sketched measurements $SHS^\T$, using Theorem~\ref{thm:MP-Law}. Accordingly, we can estimate $\tilde{\lambda}$ by the solution $\hat{\lambda}$ of the (random) equation
\begin{equation}\label{eq:Root-Solving}
	\StielEmp(-\hat{\lambda}) = 1/\lambda ,\quad \hat{\lambda}>0,
\end{equation}
if it exists. To have $\hat{\lambda}\approx \tilde{\lambda}$, we need that $m>\EffDim(\lambda)$, and we do not know this a priori - therefore we also need a way to select large enough $m$ prior to solving \eqref{eq:Root-Solving}. Moreover, to ensure small error in Theorems~\ref{thm:MP-Law}-\ref{thm:Deterministic-Sketch}, one needs to work with $z$ bounded away from $0$. As $\tilde{\lambda}=\lambda(1-\frac1m\EffDim(\lambda))$, to ensure that this is $\Omega(\lambda)$ we need to have $m\ge (1+\Omega(1))\EffDim(\lambda)$. Specifically, we shall aim to choose $m$ between $1.5\EffDim(\lambda)\le m\le 4\EffDim(\lambda)$, the constants chosen somewhat arbitrarily.

\paragraph{Testing for a good $m$.}  
We devise a test that, given $m$, with high probability: (a) rejects if $m<1.5\EffDim(\lambda)$; and (b) accepts if $m\ge 2\EffDim(\lambda)$.
Set $z_0=-5\lambda/12$. The tests accepts if $\StielEmp(z_0)>1/\lambda$, and rejects otherwise. \\
If $m<1.5\EffDim(\lambda)$ then either there is no solution $\tilde{\lambda}>0$ to $\Stiel(-\tilde{\lambda})=1/\lambda$, or there is one, specifically $\tilde{\lambda}=\lambda(1-\frac1m\EffDim(\lambda))\le \lambda/3 < -z_0$. Since $z\mapsto \Stiel(z)$ is positive and increasing for $z<0$, necessarily $\Stiel(z_0)<\Stiel(-\tilde{\lambda})=1/\lambda$. Provided that $m$ is large enough, w.h.p. $\StielEmp(z_0)<1/\lambda$ as well, hence the test rejects. \\
If $m\ge 2\EffDim(\lambda)$ then $\tilde{\lambda}=\lambda(1-\frac1m\EffDim(\lambda))\ge \lambda/2>-z_0$, hence $\Stiel(z_0)>\Stiel(-\tilde{\lambda})=1/\lambda$. Provided that $m$ is large enough, w.h.p. $\StielEmp(z_0)>1/\lambda$ as well, hence the test accepts.

\paragraph{Searching for $m$ by doubling.} Start with $m=m_0$ large enough; proceeds iteratively, doubling $m$ after each iteration $m := 2m$. At every iteration, apply the above test, halting if it accepts and continuing if it rejects. W.h.p., (a) we do not halt on any $m<1.5\EffDim(\lambda)$; (b) we halt on the first $m$ satisfying $m\ge 2\EffDim(\lambda)$, and therefore return $m$ satisfying $m\le 4\EffDim(\lambda)$.

\paragraph{Computational complexity.} W.h.p., the doubling procedure stops within $\log_2(4\EffDim(\lambda)/m_0))$ iterations. Each iteration $t$ costs at most $O(d^2m_t)$, $m_t=2^{t}m_0$, so the total cost is $O(d^2\min\{\EffDim(\lambda),m_0\})$ and dominated by the cost of  multiplying $SHS^\T$ for $S\in \RR^{m_t\times d}$ at the last iteration.

Having found large enough $m$, we can solve the equation \eqref{eq:Root-Solving} by binary search, and form the sketch \eqref{eq:feature-sketch-Def} with $\hat{\lambda}$ in place of $\tilde{\lambda}$.
For the reader's convenience, the overall adaptive sketching procedure---which 
is performed (in parallel) by every worker---is summarized 
in Algorithms~\ref{Alg:1}-\ref{Alg:2}.
\vspace{-0.1cm}
\begin{algorithm}[H]
	\caption{}\label{Alg:1}
	\begin{algorithmic}[1]
			
			\Procedure{\ProcChooseMIID}{$H$, $\lambda$, $m_0$}       
			\State $m=m_0$ 
			\While{$m<d$}
			\State{Sample i.i.d. sketch $S \in \RR^{m\times d}$ and form $S H S^\T$.}
			\State{Compute eigenvalues $\lambda_1(S HS^\T) \ge \ldots \ge \lambda_m(SHS^\T)$.}
			\If{$\StielEmp(-5\lambda/12) > 1/\lambda$} \Comment{$\StielEmp(z)$ defined in \eqref{eq:StielEmp-Def}.}
			\State{\Return{m}}	\Comment{Guarantee: w.h.p., $1.5\EffDim(\lambda) \le m \le 4\EffDim(\lambda)$.}
			\Else
			\State{m := 2m}	
			\EndIf
			\EndWhile
			\State{\Return m}
			\EndProcedure
		\end{algorithmic}
\end{algorithm}
\vspace{-0.5cm}
\begin{algorithm}[H]
	\caption{}\label{Alg:2}
	\begin{algorithmic}[1]
			
			\Procedure{\ProcChooseLambdaIID}{$H$, $\lambda$, $m$}
			\State{Sample sketch $S\in \RR^{m\times d}$ and form $S H S^\T$.}
			\State{Compute eigenvalues $\lambda_1(S HS^\T) \ge \ldots \lambda_m(SHS^\T)$.}
			\If{$\StielEmp_m(0)\le 1/\lambda$}
			\State{ERROR}; set $\hl=5\lambda/12$.
			\Else
			\State{
					Find $5\lambda/12\le \hat{\lambda} \le \lambda$ such that $\StielEmp(-\hat{\lambda})=1/\lambda$.
				}  \Comment{Can be done by binary search.}
			\EndIf
			\State{\Return $(\hat{\lambda}, \hat{W})$ where
		$					\hat{W}=S^\T (S H S^\T + \hat{\lambda}I)^{-1}S 
       $
				}
				\Comment{(In practice, return $\hat{W}g$ where $g\in\RR^d$ is the current gradient.)}
			\EndProcedure
			
		\end{algorithmic}
	
\end{algorithm}
\vspace{-0.5cm}
It is straightforward to establish the asymptotic validity of our adaptive sketching procedure.
\begin{theorem}\label{thm:Asymptotic}
	Set $\eta,D>0$, let $m_0$ be an initial sketch size. Assume \AssumAsympZero~(in particular $m_0=\Omega(d)$), and furthermore that $\lambda_{\rank(H)}(H)\ge \tau$. 
	Let $\hat{m}$ be the output of Algorithm~\ref{Alg:1}, and $(\hat{\lambda},\hat{W})$ be the output of Algorithm~\ref{Alg:2} (the latter given input $m=\hat{m}$).
	Then w.p. $1-O(m_0^{-D})$:
	\begin{itemize}[noitemsep,topsep=-5pt]
		\item $\max\{m_0,1.5\EffDim(\lambda)\} \le \hat{m} \le \max\{m_0,4\EffDim(\lambda)\}$.
		\item $|\hat{\lambda}-\tilde{\lambda}| = O(m_0^{-1/2+\eta})$.
		\item $\|\Expt[\hat{W}]-W\| = O(m_0^{-1/2+\delta})$.
	\end{itemize}
	The constants in the $O(\cdot)$ notation may depend on $\eta,D>0$ and also $\lambda>0$.
\end{theorem}
\begin{proof}
	This is a straighforward consequence of Theorems~\ref{thm:MP-Law}-\ref{thm:Deterministic-Sketch}. We omit the details.	
\end{proof}

\subsection{Non-Asymptotic Guarantees for Gaussian Sketches}
\label{sec:Nonasymptotic}

The guarantees provided by Theorem~\ref{thm:Asymptotic} are asymptotic, in that they do not depend explicitly on crucial problem parameters (such as $\lambda,\|H\|$); in particular, we implicitly assume that $H$ is well-conditioned. When $S\in \RR^{m\times d}$ is an i.i.d. Gaussian matrix, $S_{ij}\sim \m{N}(0,1/m)$, we are able to provide \textbf{non-asymptotic} error guarantees for our adaptive sketching algorithm. Our bounds are dimension-free, and remarkably do not depend explicitly on the condition number of $H$. The theorems below are the main technical part of this paper; their proofs are deferred to the appendix.

Theorems \ref{thm:Alg1-Gauss} and \ref{thm:Alg2-Gauss} respectively provide non-asymptotic guarantees for Algorithms \ref{Alg:1} and \ref{Alg:2}:
\begin{theorem}\label{thm:Alg1-Gauss}
	Assume that $S$ has i.i.d. Gaussian entries.
	Let $\delta\in(0,1)$. Suppose that $m_0\ge  C(1+\log(1/\delta))$ for large enough constant $C$. 
	W.p. $1-\delta$, Algorithm~\ref{Alg:1} outputs $\hat{m}$ such that 
	\begin{equation}\label{eq:thm:Alg1-Gauss}
		\max\{1.5\EffDim(\lambda),m_0\} \le \hat{m} \le \max\{4\EffDim(\lambda),m_0\}\,.
	\end{equation}
\end{theorem}
\begin{theorem}\label{thm:Alg2-Gauss}
	Assume that $S$ has i.i.d. Gaussian entries.
	Let $\delta\in(0,1)$ and $\eps>0$. Suppose 
	Algorithm~\ref{Alg:2} is run with $m\ge 1.5\EffDim(\lambda)$ and that moreover
	$m\ge C\frac{1 + \log(1/\delta)}{\eps^2}$ for large enough constant $C$.
	W.p. $1-\delta$, the output $\hat{\lambda}$ satisfies $|\hat{\lambda}-\tilde{\lambda}| \le \eps \tilde{\lambda}$, where $\tilde{\lambda}=\lambda\left(1-\frac1m\EffDim(\lambda)\right)$.
\end{theorem}

Our final non-asymptotic result concerns the concentration of the aggregate inverse Hessian estimate used by the server. To wit, fix $m\ge 1.5\EffDim(\lambda)$, and let $\hat{W}^{(1)},\ldots,\hat{W}^{(q)}$ be the outputs of $q$ independent runs of Algorithm~\ref{Alg:2}, representing $q$ workers producing estimates of the Newton step in parallel. Denote their average $\bar{W} = \frac{1}{q}\sum_{k=1}^q \hat{W}^{(k)}$, and the error matrix:
\begin{align}
	\bar{\m{E}} := (H+\lambda I)^{1/2}\bar{W}(H+\lambda I )^{1/2}-I.
\end{align}
Later on, we will obtain error bounds for the quasi Newton method in terms of $\|\bar{\m{E}}\|$.

\begin{theorem}\label{thm:Avg-Gauss}
	Assume that $S$ has i.i.d. Gaussian entries.
	Suppose that $m\ge 1.5\EffDim(\lambda)$, and let $q\ge 1$, $\delta\in (0,1)$ satisfy $1/\delta\le \exp(\BigOh(d))$, $q\le \exp(\BigOh(d))$. Let $\exp(-\BigOh(d)) \le \eps\le \BigOh(1)$ be small enough. If moreover
	\begin{align}
		m \ge C\frac{1}{\eps^2}\,,\quad q \ge C \frac{d\log d}{m}\frac{\log(1/\delta)}{\eps^2} 
		\label{eq_26}
	\end{align}
	for large enough $C>0$, then w.p. $1-\delta$, $\| \bar{\m{E}} \| \le \eps$.	
\end{theorem}
\begin{remark}
	Up to the logarithmic factor (which we believe is an artifact), the dimensional dependence of $q$ in \eqref{eq_26} is optimal. To see this, suppose that $H=0,\lambda=1$, so $\hat{W}^{(k)}={S^{(k)}}^\T S^{(k)}$. In particular $\rank(\bar{W})=qm$, therefore $\|\bar{\m{E}}-I\|\ge 1$ when $qm<d$. In fact, $\bar{\m{E}}$ is a Wishart (Gaussian sample covariance) matrix corresponding to $mq$ samples, therefore by standard estimates $\Expt\|\bar{\m{E}}-I\| \asymp \frac{d}{mq}\vee \sqrt{\frac{d}{mq}}$, hence the dependence on $\eps$ is correct as well.
\end{remark}
	
	\section{Convergence Results for the Approximate Newton Method}
	
	\label{sec:Quasi-Newton}

		


In this section we present guarantees for the Newton method with approximate Hessian estimates. These bounds are relevant beyond the parallel scheme considered in this paper, and accordingly are stated in general terms. Let $G:\RR^d\to \bar{\RR}$ be convex and twice-differentiable.\footnote{We denote $\bar{\RR}=\RR\cup\{\infty\}$. With this, the domain of $G$ is $\mathrm{Dom}(G):=\{x\,:\,F(x)<\infty\}$.} We consider generally the Newton-type iteration,
\begin{align}
    \theta_t 
    &= \theta_{t-1} - \alpha_t \bar{W}_t g_t, \label{eq:quasi-Newton-1}\\
    g_t &:= \nabla G(\theta_{t-1})\label{eq:quasi-Newton-2},
\end{align}
where $\bar{W}_t$ approximates the inverse Hessian,
$\sfH_t := \nabla^2 G(\theta_{t-1})$, $\bar{W}_t\approx \sfH^{-1}_t$ and $\alpha_t$ is a step size.
\begin{definition}\label{def:Accurate}
	We say \eqref{eq:quasi-Newton-1}-\eqref{eq:quasi-Newton-2} implement an $\eta$-accurate Newton method if the following holds. Define the error matrix, which measures how well $\bar{W}_t$ approximates the true inverse Hessian $\sfH_t^{-1}$:
	\begin{equation}
		\bar{\m{E}}_t := \sfH^{1/2}_t \bar{W}_t  \sfH^{1/2}_t - I \,.
	\end{equation}
 Then $\|\bar{\m{E}}_t\|\le \eta$ for every iteration $t$.
\end{definition}
For a desired precision $\eps>0$, namely when one desires $\hat{\theta}$ with $G(\hat{\theta})-G(\theta^\star)\le \eps$, $\theta^\star := \argmin G(\theta)$,
let $\sfT(\eps)$ be the smallest $t\ge 1$ (a priori, if it exists) such that $G(\theta_t)-G(\theta^\star)\le \eps$ for all $t\ge \sfT(\eps)$. This notation suppresses the dependence on $G$ and the sequence of step sizes $\alpha_t$. We consider a variation where the step size is chosen by backtracking line search, see Algorithm~\ref{Alg:LineSearch}.
\begin{algorithm}[H]
	\caption{(Backtracking Line Search)}\label{Alg:LineSearch}
	\begin{algorithmic}[1]
			
			\Procedure{\ProcLineSearch}{$\theta_{t-1},g_t,\bar{W}_t$; $G(\cdot)$}
            \State \textbf{Parameters:} $a,b\in (0,1)$ line search parameters.
			\State $\alpha=1$ 
			\While{$G(\theta_{t-1}-\mu \bar{W}g_t) > G(\theta_{t-1})-a g_t^\T (\alpha \bar{W} g_t)$}
            \State $\alpha := b\alpha$
			\EndWhile
            \State\Return $\alpha$ 
			\EndProcedure
		\end{algorithmic}
\end{algorithm}
\vspace{-0.5cm}
This paper focuses on the case when $G$ is self-concordant, which is a standard setting in the literature on second order convex optimization, cf. \citep{boyd2004convex}. We remark that previous works, for example \citep{wang2018giant,dereziński2019distributed}, derived convergence guarantees for the approximate Newton method in a different setting, namely when the Hessian of $G$ is Lipschitz. Due to space constraint, we do not present their results here, though they are applicable in our setting as well. For self-concordant objectives, we are not aware of similar results previously being written down in the literature, hence our results may be of independent interest.
\begin{definition}
	We say a univariate convex function $G:\RR \to \bar{\RR}$ is self-concordant if it is thrice-differentiable and $|G'''(x)|\le 2G''(x)^{3/2}$ for any $x$ in its domain. A multivariate convex function $G:\RR^d \to \bar{\RR}$ is self-concordant if any 1-dimensional restriction $G_{x,y}(t)=G(x+ty)$ is self-concordant.
\end{definition}
\vspace{-0.2cm}
When $f$ is self-concordant, the objective $G(\theta)=f(\theta)+\frac{\lambda}{2}\|\theta\|^2$ in \eqref{eq:Convex-Problem} is self-concordant as well, as the sum of self-concordant function is itself self-concordant, cf. \citep{boyd2004convex}.
\begin{theorem}
    Suppose that (a) $G$ is self-concordant; (b) the Newton method is $\eta$-accurate with $\eta<1/5$; and (c) $\alpha_t$ is chosen by backtracking line search. For large enough numerical $C>0$, let 
    \[
    T_0 = C \frac{G(\theta_0)-G(\theta^\star)}{ab}, \quad T_1(\eps)=C\log\log(1/\eps),\quad T_2(\eps)=C\frac{\log(\eta/\eps)}{\log(1/\eps)}\,,
    \]
    where $a,b$ are the parameters used in backtracking line search (Algorithm~\ref{Alg:LineSearch}).
    Then 
    \begin{equation}
        \sfT(\eps) \le  \begin{cases}
            T_0 + T_1(\eps) \quad&\textrm{if}\quad \eps\ge \eta \\
            T_0 + T_1(\eta) + T_2(\eps) 
            \quad&\textrm{if}\quad \eps<\eta 
        \end{cases} \,.
    \end{equation}
    \label{thm:Convergence}
\end{theorem}
\vspace{-0.5cm}
The proof of Theorem~\ref{thm:Convergence} builds on and extends an argument of \citep{pilanci2015newton}. Interestingly, Theorem~\ref{thm:Convergence} suggests that the approximate Newton method exhibits two stages: (1) While the error is above the noise floor $\eps\ge \eta$, the error decays quadratically fast. (2) Upon reaching the noise floor, $\eps\le \eta$, the convergence rate becomes linear; however, the linear rate is proportional to $\eta$, that is the error decays as $\sim (O(\eta))^t$. When $\eta$ is very small this decay rate is very fast.

Finally, Theorem~\ref{thm:Convergence} may be readily combined with Theorem~\ref{thm:Avg-Gauss} to yield error bounds for the entire end-to-end parallel method: under the setting of Theorem~\ref{thm:Avg-Gauss}, w.h.p. (and up to $\mathrm{polylog}(d)$ factors),
\[
\eta = \tilde{O}\left( \frac{1}{\sqrt{m}} + \sqrt{\frac{d}{mq}} 
 \right) \,.
\]

	
	\section{Numerical Experiments}\label{sec:Experiments}






We demonstrate the utility and validity of our results via numerical experiments. 
Due to space contraints, we defer some details and additional experiments to the appendix.

Our first set of experiments aims to demonstrate our adaptive sketching procedure (Algorithms~\ref{Alg:1}-\ref{Alg:2}) independently of any specific optimization context. For the first experiment, we consider diagonal matrices\footnote{We choose $H$ to be diagonal w.l.o.g. since the Gaussian sketch is invariant to orthogonal transformations.}
$H$ with eigenvalues $(h_1,\ldots,h_d)$ exhibiting polynomial spectral decay $h_k=k^{-\alpha}$, where $\alpha>0$ is a parameter. The dimension $d=10^4$ is fixed and large, and $\lambda=1$.
The sketching matrix is i.i.d. Gaussian (in the appendix, we present results for two more i.i.d. sketching matrices). 
We run Algorithm~\ref{Alg:1}, repeating for $T=20$ Monte-Carlo trials, and report the sketching dimension found, see Table~\ref{table:Exp1}. Remarkably, our algorithm consistently succeeds, namely finds a sketching dimension that is within $1.5\EffDim(\lambda)\le m \le \max\{m_0,4\EffDim(\lambda)\}$ - per Theorems~\ref{thm:Asymptotic} and \ref{thm:Alg1-Gauss}. 
\begin{table}[]
        \centering
        {\small
        \begin{tabular}{ |c|c|c|c|c| } 
         \hline
         $\alpha$ & $\EffDim(\lambda)$ &  $m_0$ & Success rate & Avg. dim.\\
         \hline
         $1$ &$8.8$        &$10$ &$1.0$ &$20$\\ 
         $2/3$ &$59.7$             &$10$ &$1.0$ &$160$\\ 
         $1/2$   &$190.4$              &$10$ &$1.0$ &$640$\\ 
         \hline
        \end{tabular}
        }
        \caption{Success rate of Algorithm~\ref{Alg:1}.}
        \label{table:Exp1}
\end{table}

Next, we demonstrate that Algorithm~\ref{Alg:2} produces low-bias estimates of $W$.
Note that the average Frobenius error $\|\hat{W}-W\|_F^2/d^2$ is a reasonable proxy for the bias of $\hat{W}$, see Theorems~\ref{thm:Determinstic-Cov}-\ref{thm:Deterministic-Sketch}. 
In the next experiment, we consider two random ensembles of matrices $H$, which we call here (L) and (R), and exhibits respectively slow and fast spectral decay; we defer the details to the appendix, due to space constraints. Figure~\ref{fig:Exp2} plots the bias proxy for $m$ increasing, starting from $m=1.5\EffDim(\lambda)$, for the Gaussian sketch. As $m$ increases, the bias decreases (roughly at rate $1/\sqrt{m}$). We also plot the bias proxy for a naive estimator of $W$, without bias correction ($\tilde{\lambda}=\lambda$); clearly, our algorithm has substantially smaller bias throughout.
\begin{figure}[t!]
\centering 
\includegraphics[scale=0.30]{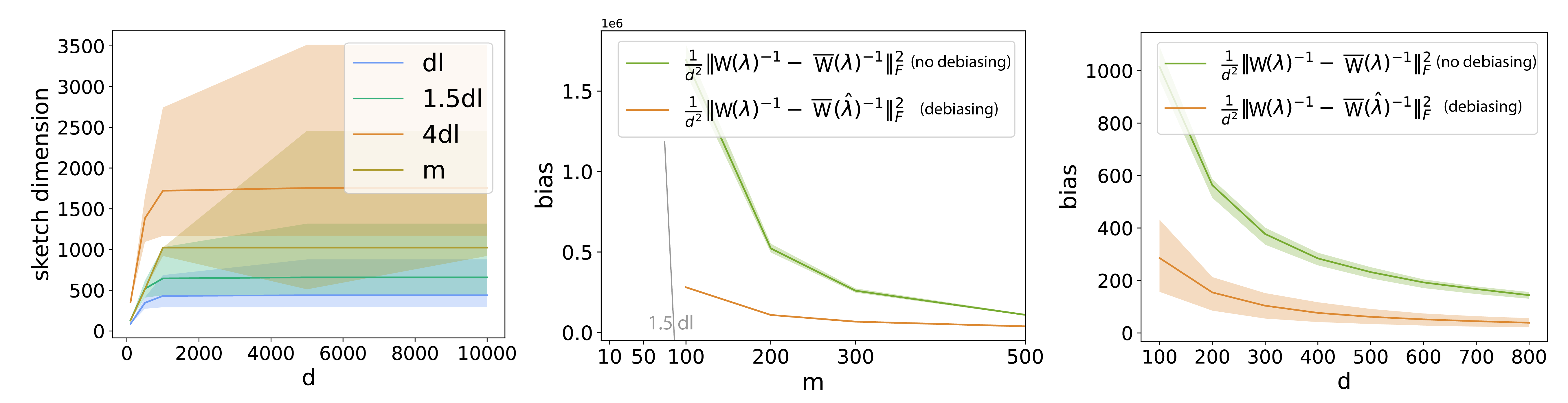}
\caption{
The bias proxy of the bias-corrected inverse Hessian estimator is substantially lower than without bias correction. Rightmost plot: ensemble (R); left and middle plots: ensemble (L). Leftmost plot: the sketching dimension found by Algorithm~\ref{Alg:1} on ensemble (L). Shaded area: 20\%-80\% confidence interval; We take $T=10$  Monte-Carlo trials. 
}
\label{fig:Exp2}
\end{figure}

Our final set of experiments concerns the performance of the proposed optimization strategy in its entirety, focusing on the benefits of bias reduction. 
We present experimental results for real-world optimization tasks---for both ridge and logistic regression---using publicly available UCI datasets \citep{chang2011libsvm}. Due to space constraints, the details appear in the appendix. 
We compare our method with an approximate parallel Newton method where Hessians are sketched ($m$ chosen by Algorithm~\ref{Alg:1}), but no debiasing is done; that is, $\tilde{\lambda}=\lambda$. It is evident from our results, summarized in Figure~\ref{fig:UCI}, that bias correction accelerates convergence, often substantially.

\begin{figure*}[ht!]
\begin{subfigure}{}
\includegraphics[width=3.2cm]{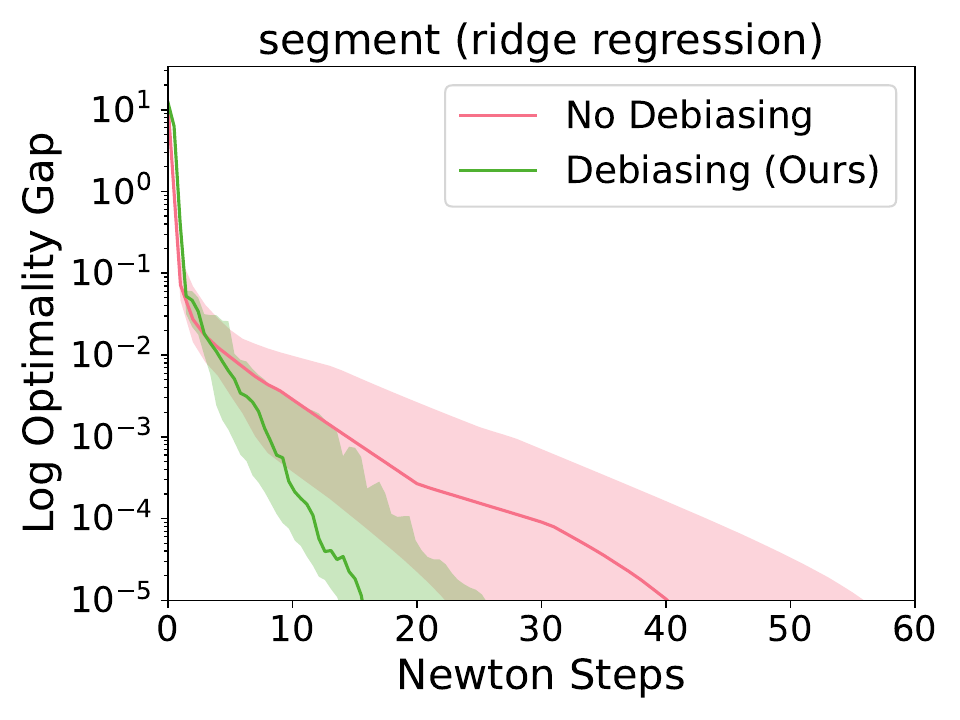}
\end{subfigure}
\hfill
\begin{subfigure}{}
\includegraphics[width=3.2cm]{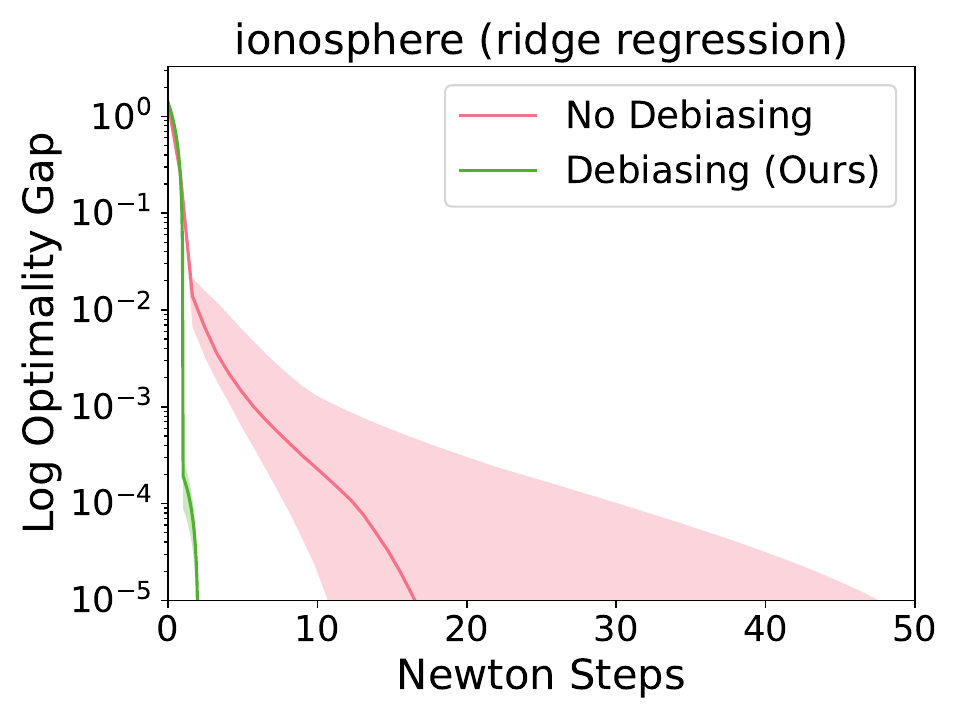}
\end{subfigure}
\hfill
\begin{subfigure}{}
\includegraphics[width=3.2cm]{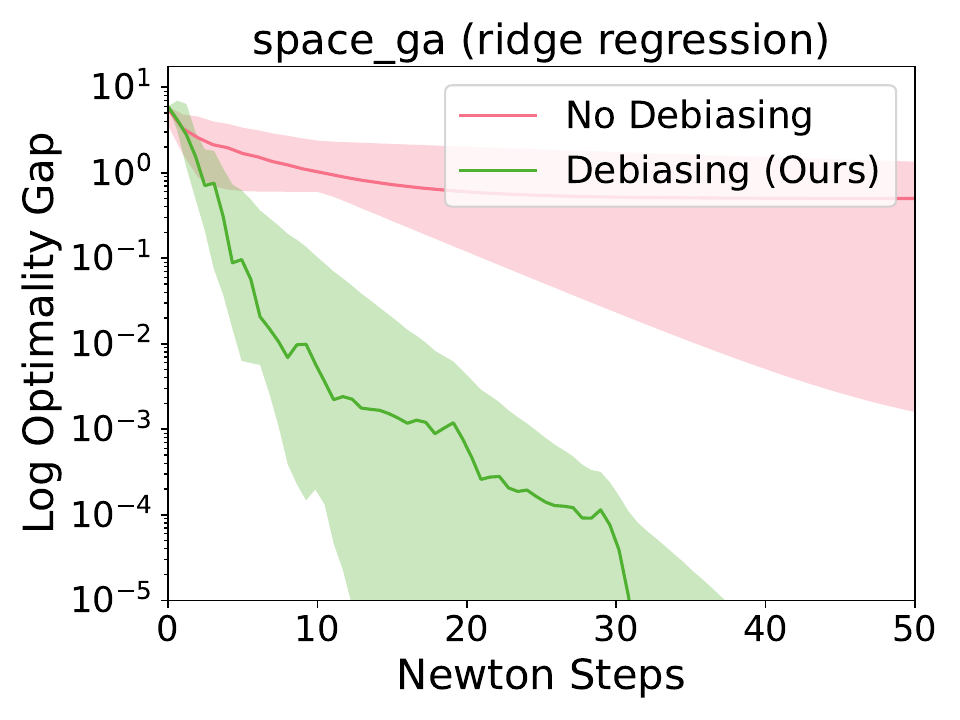}
\end{subfigure}
\hfill
\begin{subfigure}{}
\includegraphics[width=3.2cm]{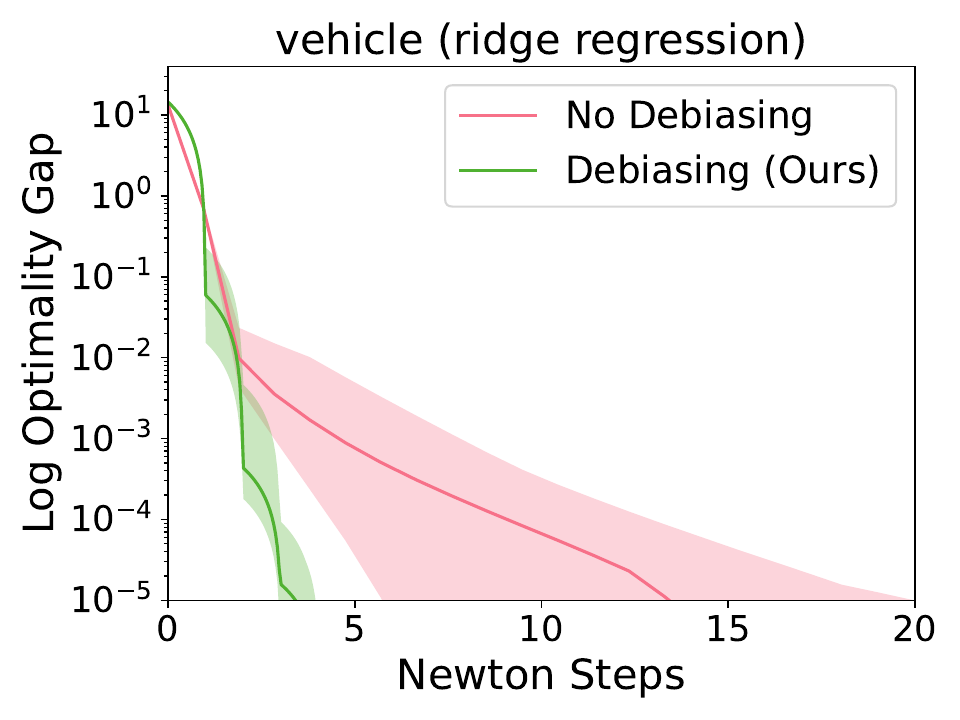}
\end{subfigure}
\\
\begin{subfigure}{}
\includegraphics[width=3.2cm]{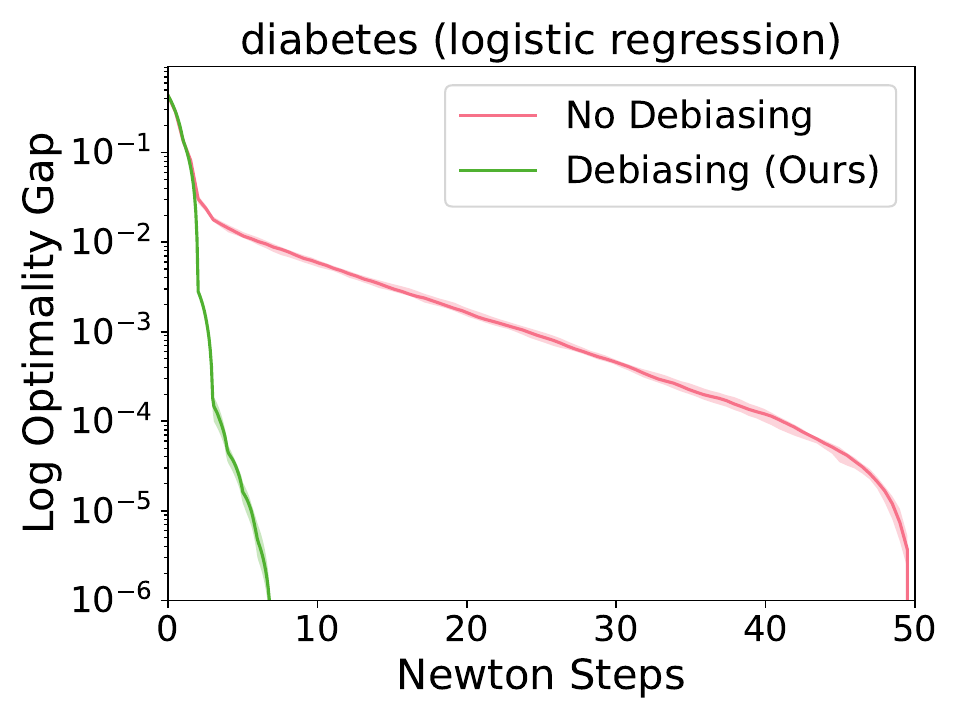}
\end{subfigure}
\hfill
\begin{subfigure}{}
\includegraphics[width=3.2cm]{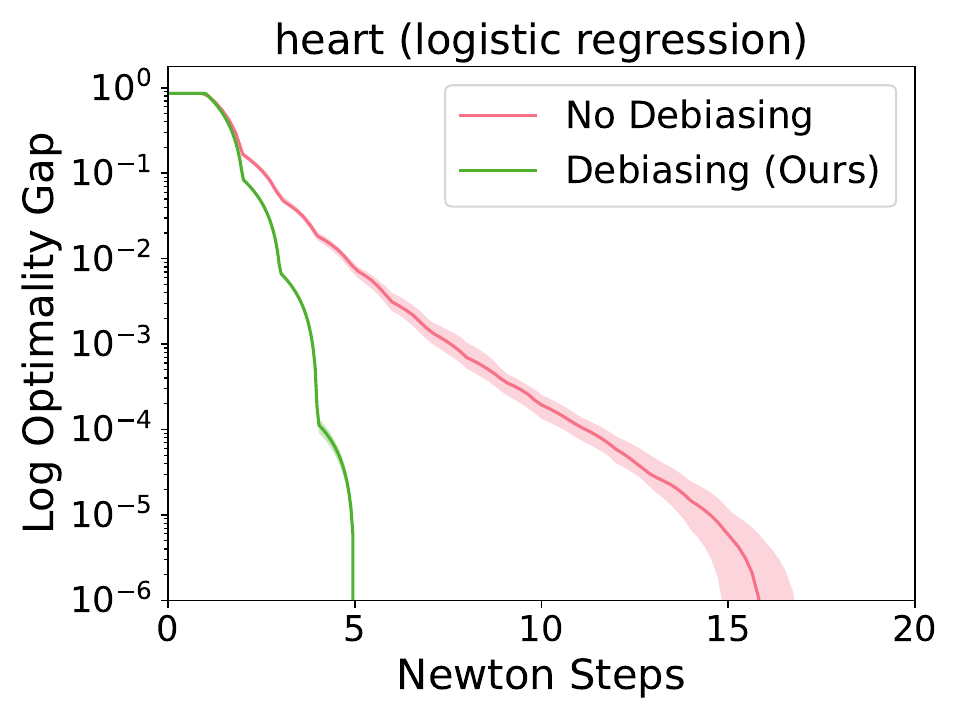}
\end{subfigure}
\hfill
\begin{subfigure}{}
\includegraphics[width=3.2cm]{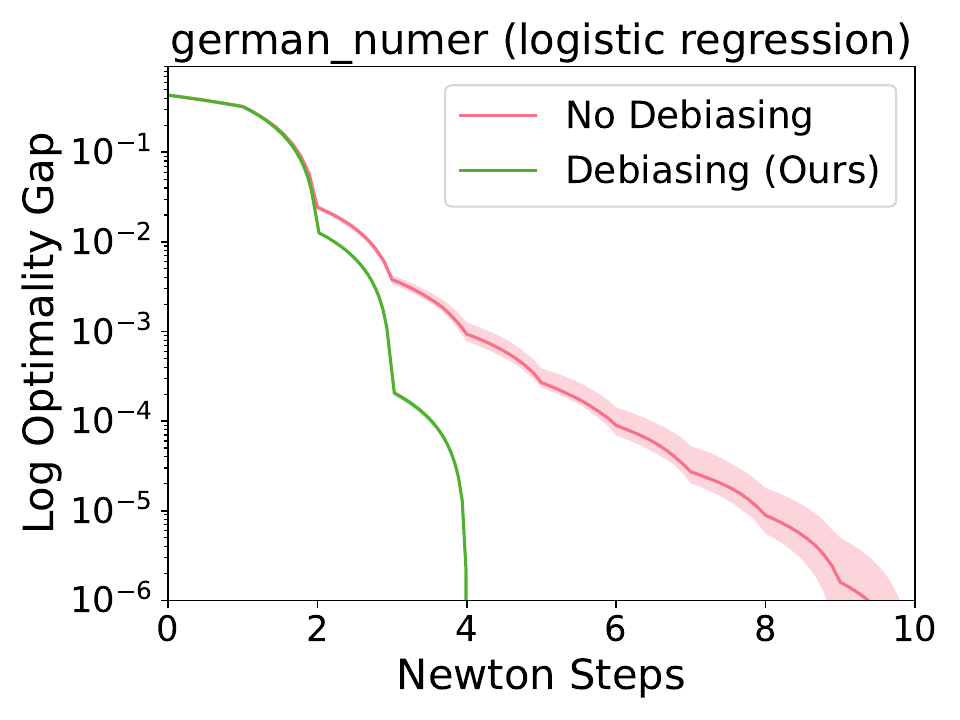}
\end{subfigure}
\hfill
\begin{subfigure}{}
\includegraphics[width=3.2cm]{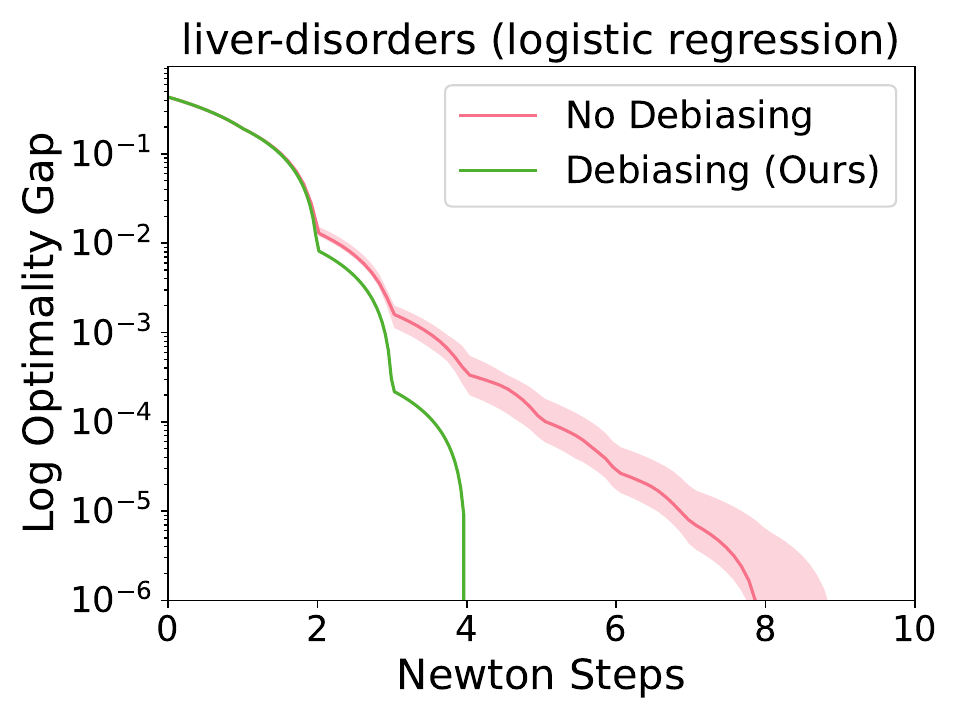}
\end{subfigure}

\caption{
Improved convergence of our parallel sketched Newton method with bias correction.
The title of sub-figure corresponds to the dataset used. Repeating for $T=10$ Monte-Carlo trials, the curve corresponds to the median and the shaded part to a $20\%$-$80\%$ confidence interval.
}\label{fig:UCI}
\end{figure*}

	\section*{Acknowledgments}
	We are grateful to Daniel Lejeune for inspiring conversations about his work \citep{lejeune2022asymptotics}.

Mert Pilanci and Fangzhao Zhang were supported in part by National Science Foundation (NSF) under Grant DMS-2134248; in part by the NSF CAREER Award under Grant CCF-2236829; in part by the U.S. Army Research Office Early Career Award under Grant W911NF-21-1-0242; in part by the Office of Naval Research under Grant N00014-24-1-2164. In addition, Fangzhao Zhang was supported in part by a Stanford Graduate Fellowship.

The work of Elad Romanov was supported in part by the NSF under Grant DMS-1811614 (PI: Donoho), and in part by the personal faculty funds of Prof. David Donoho, for whom he wishes to extend his sincere appreciation.

	\printbibliography

	\newpage
	
	\appendix
	
	\clearpage
\newpage
\appendix

\renewcommand \thepart{} 
\renewcommand \partname{}
\addcontentsline{toc}{section}{Appendix} 
\part{Appendix} 
\parttoc 

\newpage




\newpage

\section{Additional Experiments and Details}\label{exp_supp}

\subsection{Extension of Table~\ref{table:Exp1}}
\label{sec:Exp-Table-Appendix}

We run experiments with random sketching matrices with i.i.d. entries of the following types: (1) Gaussian (G): $S_{i,j}\sim \m{N}(0,1/m)$; Rademacher (R): $S_{i,j}\sim \mathrm{Unif}(\pm 1/\sqrt{m})$; (3) sparse Rademacher (SR): $S_{i,j}=0$ w.p. $1-p$ and $S_{i,j}\sim \mathrm{Unif}(\pm 1/\sqrt{pm})$ w.p. $p$ for $p=1/10$. 
The following is an extension of Table~\ref{table:Exp1}, to include also the Rademacher (R) and sparse Rademacher (SR) sketch types:

\begin{table}[h!]
        \centering
        {\small
        \begin{tabular}{ |c|c|c|c|c|c| } 
         \hline
         $\alpha$ & $\EffDim(\lambda)$ &  $m_0$ & Sketch &Success rate & Avg. dim.\\
         \hline
         $1$ &$8.8$        &$10$ &G &$1.0$ &$20$\\ 
         $2/3$ &$59.7$             &$10$ &G &$1.0$ &$160$\\ 
         $1/2$   &$190.4$              &$10$ &G &$1.0$ &$640$\\ 
         $1$ &$8.8$        &$10$ &R &$1.0$ &$20$\\ 
         $2/3$ &$59.7$             &$10$ &R &$1.0$ &$160$\\ 
         $1/2$   &$190.4$              &$10$ &R &$1.0$ &$640$\\ 
         $1$ &$8.8$        &$10$ &SR &$1.0$ &$20$\\ 
         $2/3$ &$59.7$             &$10$ &SR &$1.0$ &$160$\\ 
         $1/2$   &$190.4$              &$10$ &SR &$1.0$ &$640$\\ 
         \hline
        \end{tabular}
        }
        \caption{Success of Algorithm~\ref{Alg:1}. $T=20$ Monte-Carlo trials.}
        \label{table:Exp1-Appendix}
\end{table}
Our method appears to be remarkably robust, per the finding in Table~\ref{table:Exp1-Appendix}, to the extent that the table may seem somewhat pointless to include. We nonetheless provide it (and code to run this experiment) as a sort of sanity check.

\subsection{Details for Figure~\ref{fig:Exp2}}\label{sec:fig1_detail}



We provide additional details for the experiments in Figure~\ref{fig:Exp2}. In the middle and rightmost plots, the estimator $\bar{W}$ is obtained by averaging $q=500$ local sketched Hessians, both with and without bias correction.

Ensemble (L): We use $L=10^{-3}$. The matrix $H$ is $H=X^\T X$, where $X\in \RR^{n\times d}$ has the singular value decomposition (SVD) $X=UDV^\T$; $U\in \RR^{n\times d},V\in \RR^{d\times d}$  have orthonormal columns and $D\in\RR^{d\times d}$ is diagonal. $U,V$ are sampled uniformly (Haar) from the corresponding Stieffel manifolds (have a rotationally invariant distribution).  
For the leftmost plot $n=10^4$, and $d$ varies. For the middle plot, we fix $d=10^3$. The diagonal entries of $D$ are random, and satisfy
\[
D_{k,k} = (0.9+\eps_k)^{k/2},\quad \eps_k\sim \m{N}(0,10^{-4})\,.
\]

Ensemble (R): We use $\lambda=10^{-5}$. Similarly to ensemble (L), $H=X^\T X$ for $X\in \RR^{n\times d}$, with the singular vectors of $X$ of similarly uniformly random. We use $n=10^4$, $d=10^3$. The singular values of $X$ are  
non-random, with $D_{k,k}=(k/d)^2$.


\subsection{Additional Synthetic Optimization Tasks}\label{sec:exp_syn}

\begin{figure*}[ht!]
\centering 
\includegraphics[scale=0.5]{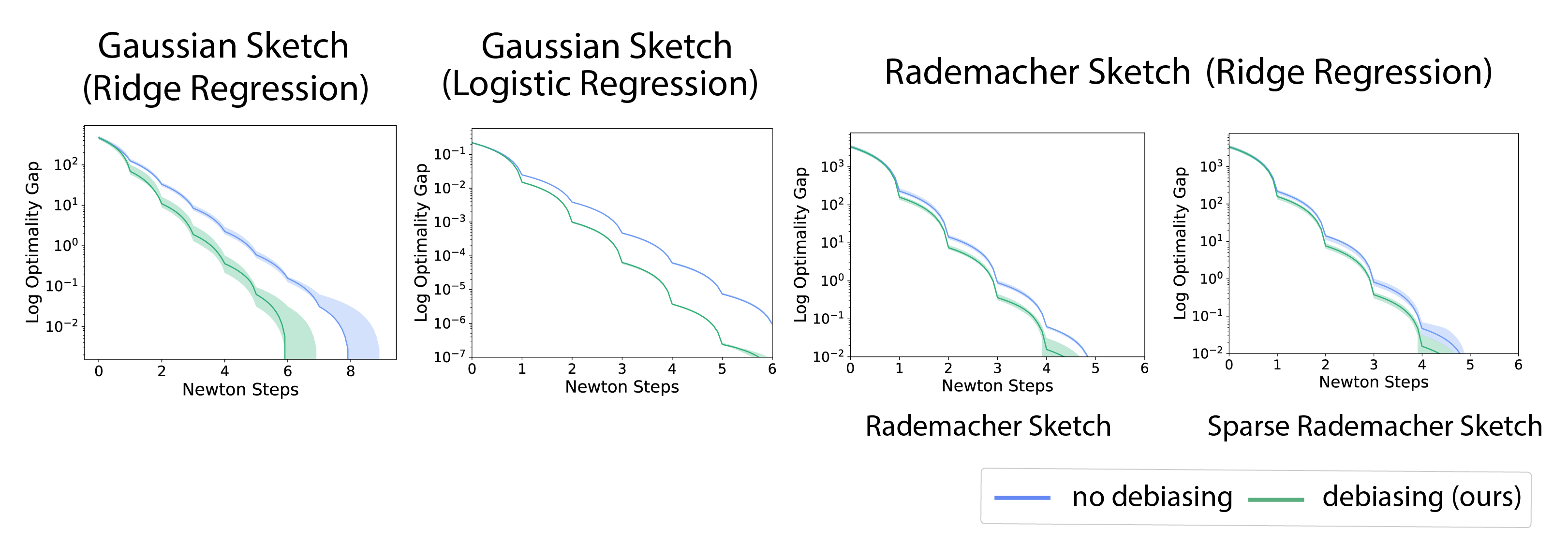}
\caption{
Convergence of the parallel Newton method with Hessian sketching, on synthetic optimization tasks.
}\label{sanity2}
\end{figure*}

In this section we present additional experiments on synthetic optimization tasks, aiming to test the performance of our proposed end-to-end parallel optimization method. 
Specifically, we consider ridge and logistic regression. We test both Gaussian and Rademacher sketches (see Section~\ref{sec:Exp-Table-Appendix}).

\paragraph{Distribution of covariates (design).} 
We generate designs $X\in \RR^{n\times d}$ with $n=10^4$ and $d=500$. Denote the SVD $X=UDV^\T$. The singular vectors $U\in \RR^{n\times d},V\in \RR^{d\times d}$ are uniformly random. The singular values decay exponentially, $D_{k,k}=0.99^{k/2}$.

\paragraph{Parameters.}
In all our experiments, $\lambda=10^{-3}$. For experiments with ridge regression, the number of workers is $q=10$; for logistic regression, it is $q=20$.

\paragraph{Ridge regression.} 
The responses are generated according to $y_i=x_i^\T \theta^\star + \m{N}(0,0.01)$, where the ground-truth regressor $\theta^\star \in \RR^d$ is generated according to $\theta^\star\sim \m{N}(0,I)$.
The unregularized objective $F(\theta)$ is the L2 loss:
\[
F(\theta) = \frac1n \sum_{i=1}^n (x_i^\T \theta - y_i)^2 = \|X\theta - y\|^2,
\]
where $x_1,\ldots,x_n\in \RR^d$ are the rows of $X$. 
Recall that we aim to minimize the regularized objective $G(\theta):= F(\theta) + \frac{\lambda}{2}\|\theta\|^2$.

\paragraph{Logistic regression.}
We generate labels according to 
\[
y_i = (\sign\left[x_i^\T \theta^\star + \m{N}(0,10^4) \right]+ 1)/2 \in \{0,1\} \,,
\]
where $\theta^\star \sim \m{N}(0,I)$. The function $F(\theta)$ is the log loss:
\[
F(\theta) = -\frac1n \sum_{i=1}^n \left[
y_i \log\left(\frac{1}{1+e^{-x_i^\T \theta}}  \right)
+ (1-y_i)\log\left(1 - \frac{1}{1+e^{-x_i^\T \theta}}  \right) \right]\,.
\]

\paragraph*{}
Figure~\ref{sanity2} plots the error $G(\theta_t)-\argmin_{\theta} G(\theta)$ over the iterations $t$ of the algorithm. Each plot is obtained by $T=10$ Monte-Carlo trials; the shaded area corresponds to a 20\%-80\% confidence interval. It is evident that bias correction improves on the the convergence speed of the algorithm over the uncorrected ($\tilde{\lambda}=\lambda$) alternative.

\subsection{Details for Experiments on UCI Data Sets (Figure~\ref{fig:UCI})}\label{sec:uci_detail}

\begin{figure}[th!]
\begin{subfigure}{}
\includegraphics[width=3.2cm]{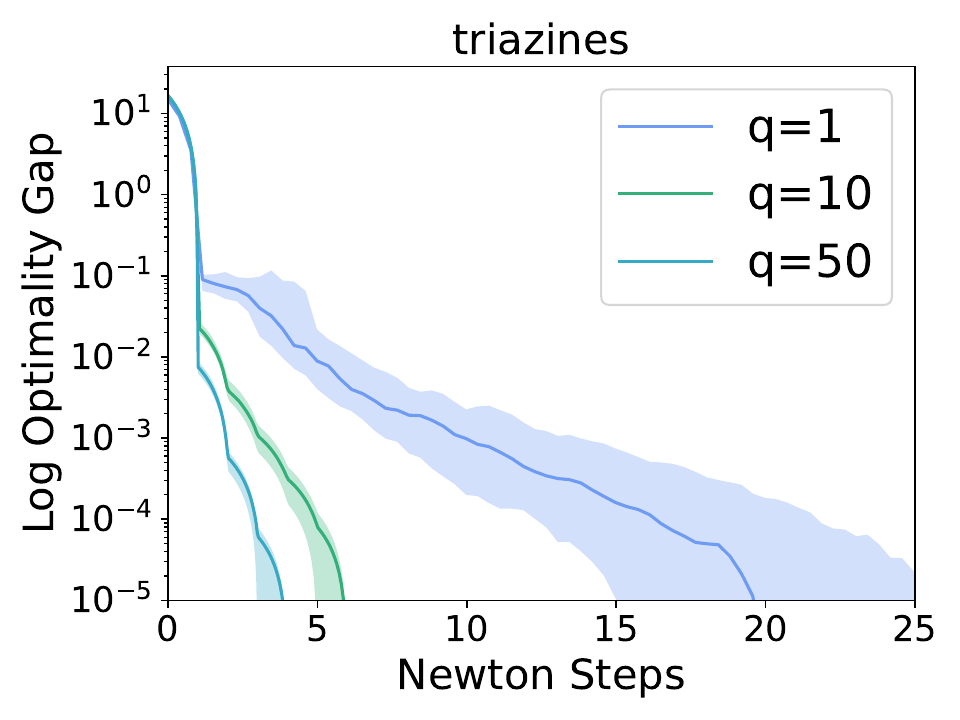}
\end{subfigure}
\hfill
\begin{subfigure}{}
\includegraphics[width=3.2cm]{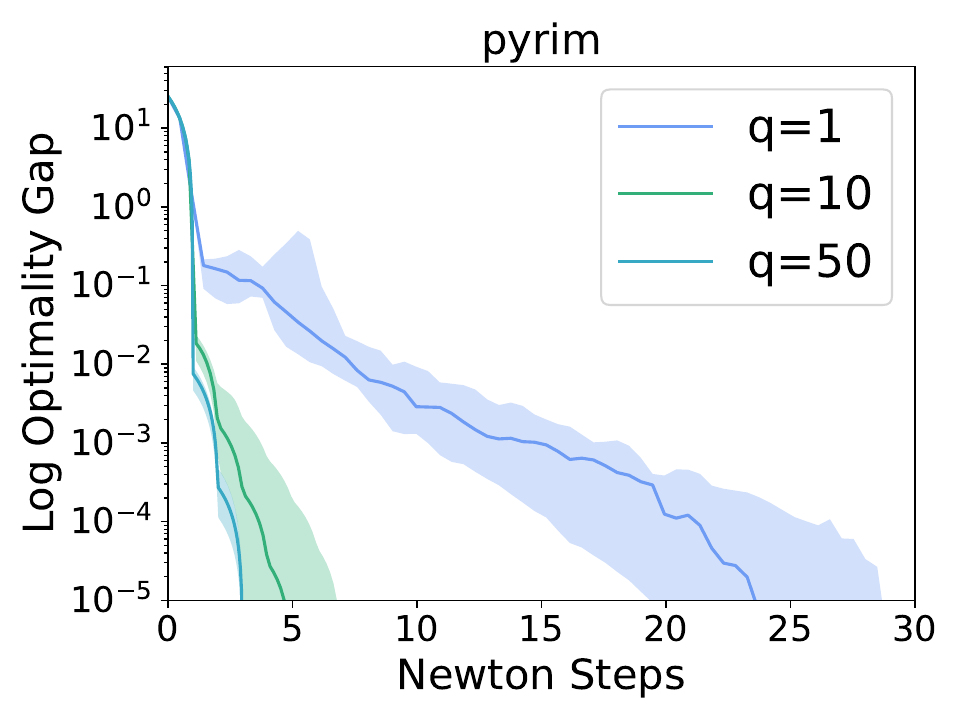}
\end{subfigure}
\hfill
\begin{subfigure}{}
\includegraphics[width=3.2cm]{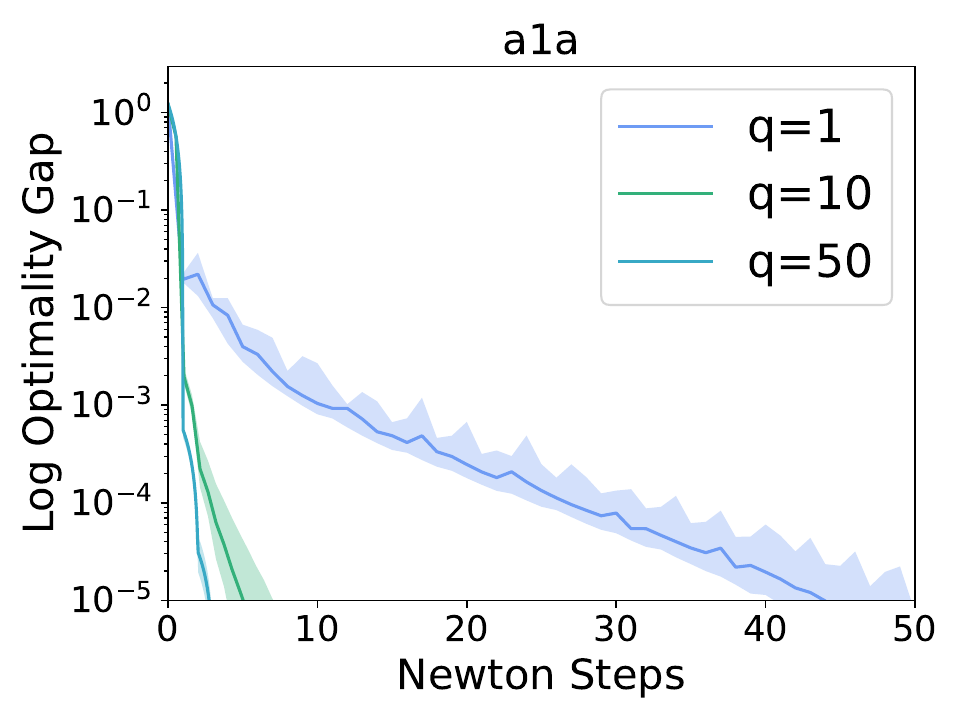}
\end{subfigure}
\hfill
\begin{subfigure}{}
\includegraphics[width=3.2cm]{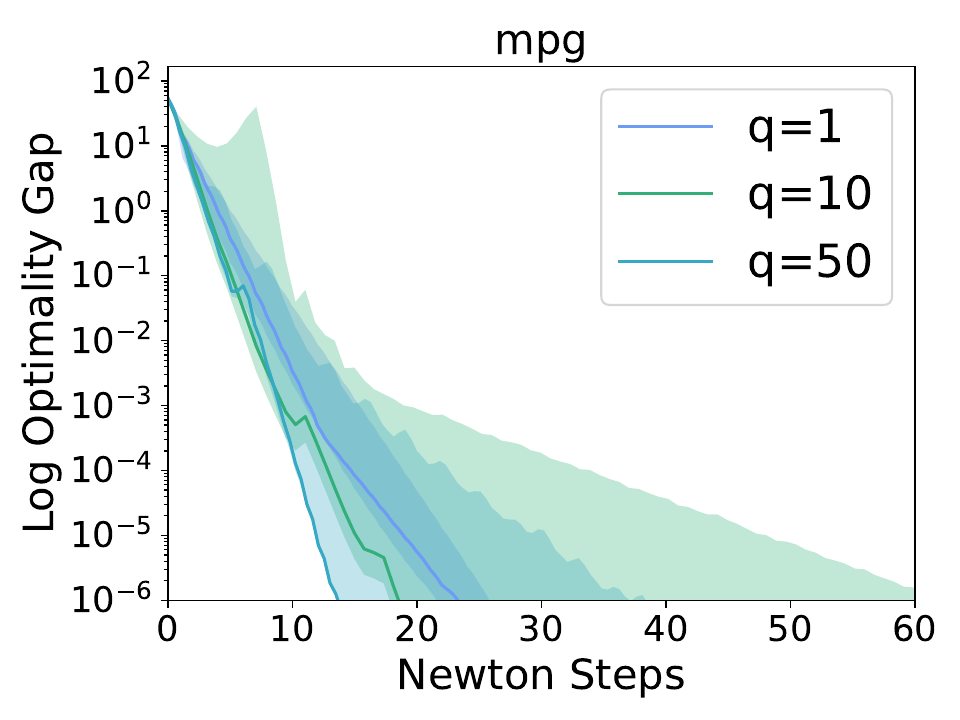}
\end{subfigure}
\\
\begin{subfigure}{}
\includegraphics[width=3.2cm]{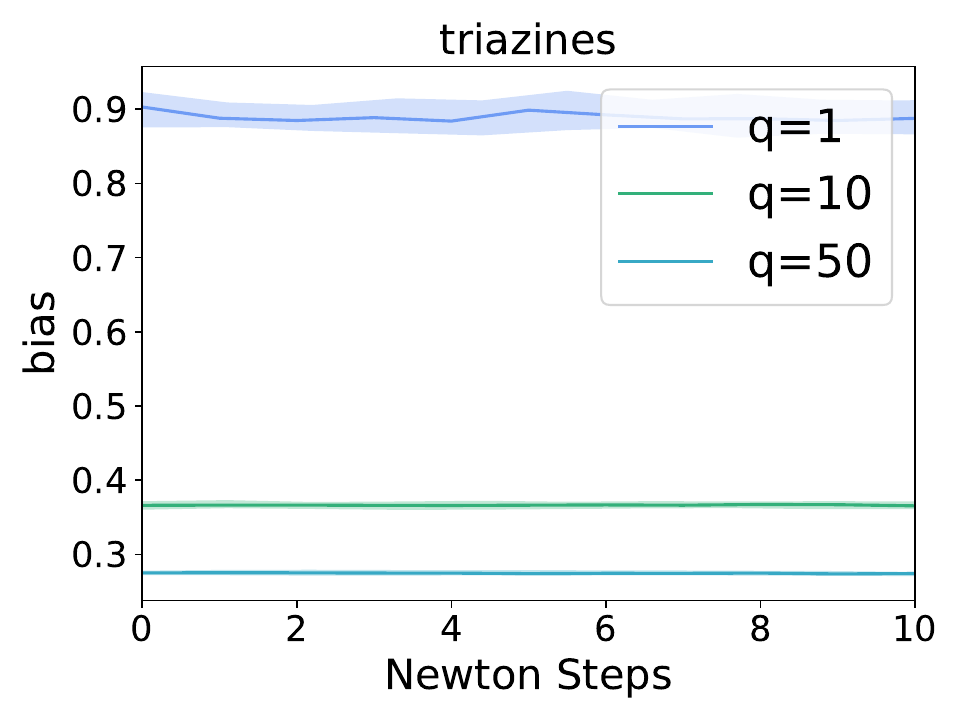}
\end{subfigure}
\hfill
\begin{subfigure}{}
\includegraphics[width=3.2cm]{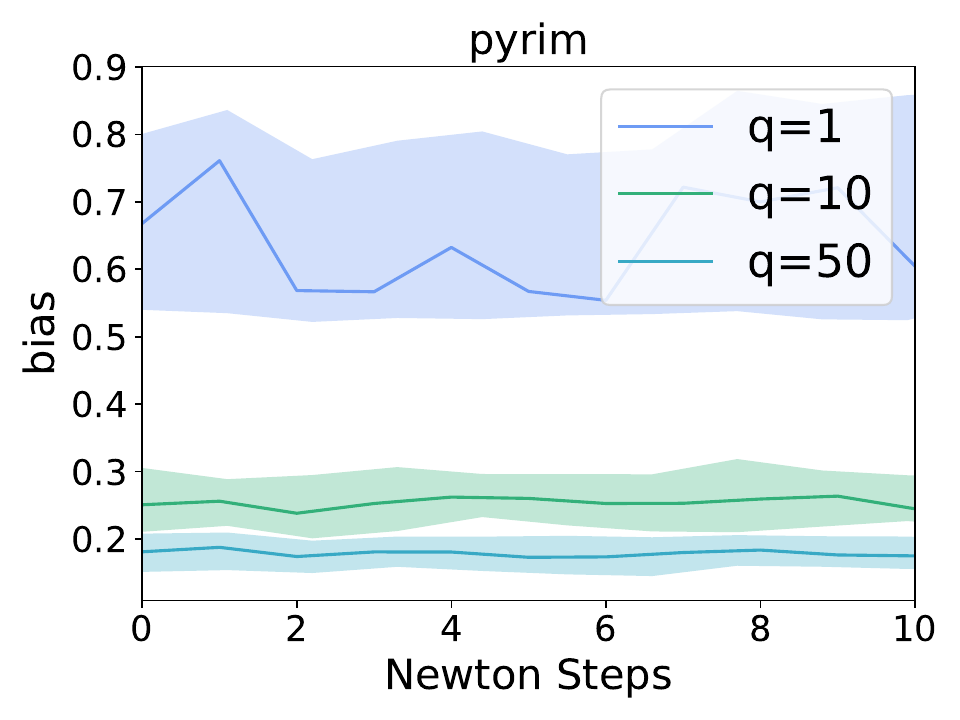}
\end{subfigure}
\hfill
\begin{subfigure}{}
\includegraphics[width=3.2cm]{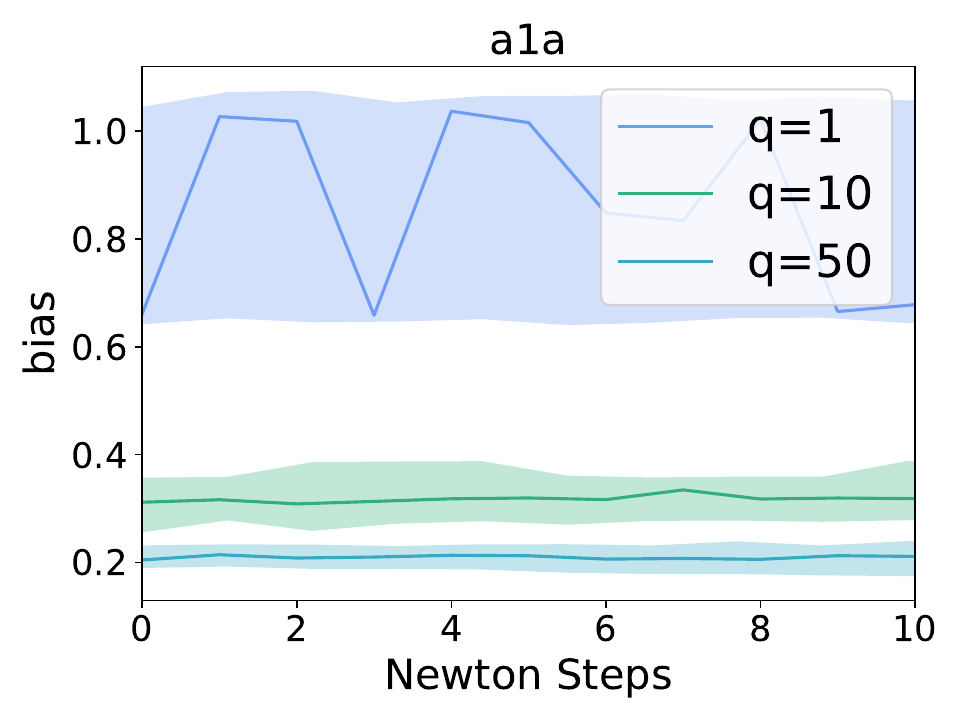}
\end{subfigure}
\hfill
\begin{subfigure}{}
\includegraphics[width=3.2cm]{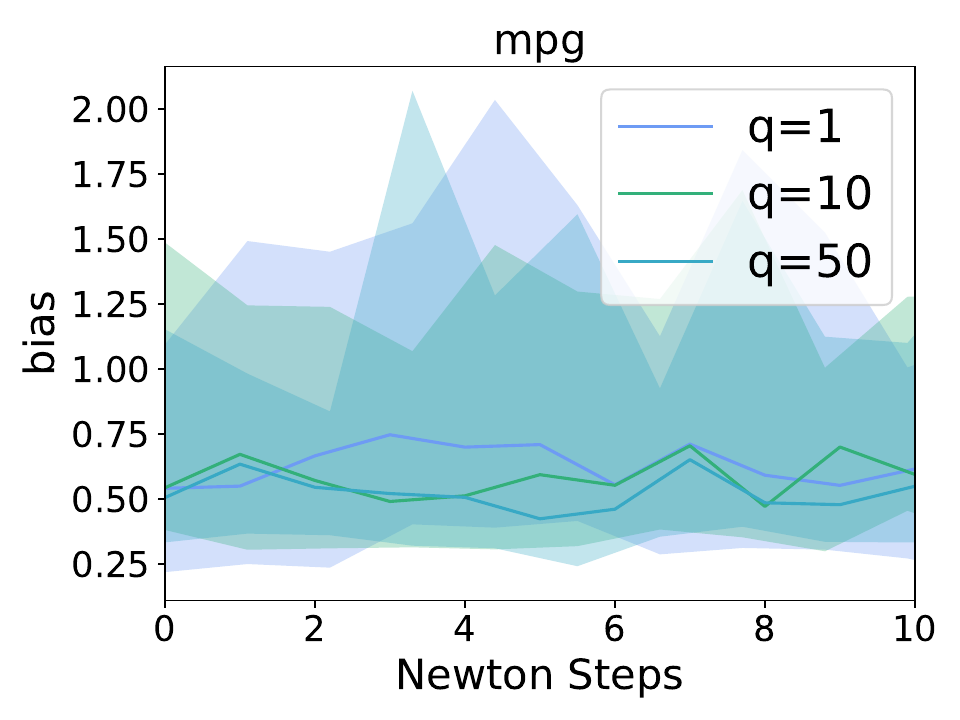}
\end{subfigure}

\caption{Convergence 
rate for optimization tasks on additional UCI data sets. Shaded area corresponds to 20\%-80\% confidence interval. 
}
\label{uci_worker}
\end{figure}

We provide additional details for the experiments reported in Figure~\ref{fig:UCI}, and present results for additional data sets from the UCI data repository \citep{chang2011libsvm}.

We run experiments directly on the available data, without additional preprocessing. The objectives, for ridge and logistic regression, are as described in Section~\ref{sec:exp_syn}, but now the covariates and reponses/labels are given and not generated. 


\paragraph{Parameters.} We use $\lambda=10^{-3}$ for all experiments. We vary the number of workers between $q\in \{1,10,50\}$. For every data set, we run $T=10$ Monte-Carlo trials.

\paragraph{Results.} The experiments are reported in Figure~\ref{uci_worker}. The plots on the top correspond to the optimality gap $G(\theta_t)-\argmin_{\theta} G(\theta)$, $G$ being the regularized objective, as a function of the iteration number $t$. The bottom plots are the inverse Hessian normalized Frobenius error $\|W_t-\bar{W}_t\|_F^2/d^2$. We see that as $q$ grows larger, the algorithm tends to converge faster; though the benefit of further increasing $q$ tends to diminish when $q$ is large.

\newpage

\section{Non-Asymptotic Random Matrix Theory Results}

The proofs of Theorems~\ref{thm:Alg1-Gauss}-\ref{thm:Avg-Gauss}, rely on non-asymptotic results from random matrix theory that that we state in this section and prove later on. 
The results below concern the behavior of the random matrices $S HS^\T,H^{1/2}S^\T S H^{1/2}$ where $H\succeq 0$ and $S\in \RR^{m\times d}$ is an i.i.d. Gaussian matrix $S_{ij}\sim \m{N}(0,1/m)$.

The following result concerns the smallest eigenvalue of the matrix $S H S^\T\in \RR^{m\times m}$. 
\begin{theorem}\label{thm:Smallest-Eigenvalue-Gauss}
	Suppose $S$ has i.i.d. Gaussian entries.
	There are constants $c_1,c_2,c,C>0$ such that the following holds.
    Set any $\eta>0$. 
	If 
	$m \le c_1\EffDim(\eta)$ then w.p. $1-Ce^{-cm}$,
	\[
	\lambda_{\min}(S H S^\T) \ge c_2 \eta  \frac1m\EffDim(\eta) \,.
	\]
\end{theorem}

The following two results concern the closedness of the empirical (companion) Stieltjes transform
\begin{align}
	\StielEmp(z) := \frac1m\tr\left( S H S^\T-zI \right)^{-1},
\end{align}
to the deterministic function $\Stiel(z)$, the solution of the Marchenko-Pastur equation \eqref{eq:MP-Eq}.

The following concentration inequality is well known, cf. \cite{bai2010spectral}. 
\begin{lemma}\label{lem:StielConc}
	Suppose that $S=\frac{1}{\sqrt{m}}[r_1|\cdots|r_m]^\T$ where
	the rows $r_1,\ldots,r_m\in\RR^d$ are independent random variables. For any $z<0$ and $t\ge 0$,
	\[
	\Pr(|\StielEmp(z)-\Expt\StielEmp(z)|\ge t) \le 2e^{-cmz^2t^2} \,.
	\]
\end{lemma}
We provide a self-contained proof for completeness, see Section~\ref{sec:proof-lem:StielConc}. 

Next, we show that $\Expt\StielEmp(z)$ approximately satisfies \eqref{eq:MP-Eq}, with explicit non-asymptotic error bounds. Note that \eqref{eq:MP-Eq} can be be equivalently rewritten (assuming $z\ne 0$) as
\begin{equation}
    \Stiel(z) = - \frac{1}{z} \left( 1 - \frac{1}{m}\EffDim(1/\Stiel(z))
    \right) \,.
\end{equation}
\begin{theorem}\label{thm:Stiel-NonAsymp-Gaussian}
	Suppose that $S$ has i.i.d. Gaussian entries. 
	For any $z<0$, $\Expt\StielEmp(z)$ satisfies:
	\begin{align}\label{eq:thm:Stiel-NonAsymp-Gaussian}
		\Expt\StielEmp(z) 
		&= -\frac{1}{z} \left( 1 - \frac1m \EffDim({1/\Expt\StielEmp(z)}) \right) + e_m(z; H)\,,
	\end{align}
	where the error term $e_m(z;H)$ satisfies
	\begin{align}\label{eq:thm:Stiel-NonAsymp-Gaussian-Error}
		e_m(z;H) 
		&\lesssim 
		\frac{\beta}{\alpha} \frac{1}{|z|\sqrt{m}} \,,
	\end{align}
	with $\alpha=\alpha_m(z;H),\beta=\beta_m(z;H)$ defined as
	\begin{align}
		\alpha &:= |z|\Expt\StielEmp(z)\,,\\
		\beta &:= \frac1m \EffDim({1/\Expt\StielEmp(z)}) \,.
	\end{align}
\end{theorem}
The proof of Theorem~\ref{thm:Stiel-NonAsymp-Gaussian}, given in Section~\ref{sec:proof-thm:Stiel-NonAsymp-Gaussian}, follows along a well-known computation technique in random matrix theory, and explicitly leverages the orthogonal invariance of the Gaussian distribution. Note that the error terms $\alpha,\beta$ depend on $\Expt\StielEmp(z)$ itself. This feature, which somewhat complicates the proofs of Theorems~\ref{thm:Alg1-Gauss}-\ref{thm:Alg2-Gauss}, is generally not a mere artifact of the calculation, see e.g. the bounds of \citep{knowles2017anisotropic} under asymptotic conditions.

\newpage

\section{Proof of Theorem~\ref{thm:Alg1-Gauss}}

Denote for brevity, per the description of Algorithm~\ref{Alg:1},
\begin{equation}
	z_0 := -\frac{5}{12}\lambda \,.
\end{equation}
The proof of Theorem~\ref{thm:Alg1-Gauss} consists of two parts. First, we show that if $m<1.5\EffDim(\lambda)$, then with high probability (w.h.p.) $\StielEmp_m(z_0) < 1/\lambda$, hence such $m$ would not be returned. Then, we show that when $m\ge 2\EffDim(\lambda)$, then w.h.p. $\StielEmp_m(z_0)\ge 1/\lambda$, which would guarantee that the algorithm must return such $m$ (if it has not halted before). Consequently, the output of the algorithm must satisfy $m\le 4\EffDim(\lambda)$. 

The next lemma shows that Algorithm~\ref{Alg:1} does not return $m$ which is \emph{much} smaller than $\EffDim(\lambda)$.
\begin{lemma}\label{lem:Small-n}
	Suppose that $S$ has i.i.d. Gaussian entries. There is a (small) constant $c_0$ such that if $m<c_0 \EffDim(\lambda)$ then
	\begin{align*}
		\Pr(\StielEmp(z_0) >  1/\lambda) \le Ce^{-cm} \,.
	\end{align*} 
\end{lemma}
\begin{proof}
	The lemma is a consequence of Theorem~\ref{thm:Smallest-Eigenvalue-Gauss}, which bounds the smallest eigenvalue of $S HS^\T$. 
	Set $c_0=\min\{c_1,12c_2/7\}$ for $c_1,c_2$ from Theorem~\ref{thm:Smallest-Eigenvalue-Gauss}. Assume that $m<c_0\EffDim(\lambda)$. Since in particular $m<c_1\EffDim(\lambda)$, the following event holds w.p. $1-Ce^{-cm}$:
	\[
	\StielEmp(z_0) \le \frac{1}{\lambda_{\min}(SHS^\T)-z_0} \le  \frac{1}{c_2\frac{1}{m}\EffDim(\lambda) \cdot \lambda  + \frac{5}{12}\lambda}\,.
	\]
	Moreover, since $c_2/c_0\ge 7/12$, 
	\[
	c_2\frac{1}{m}\EffDim(\lambda)>c_2 \frac{1}{c_0}\ge  \frac{7}{12} \,.
	\]
	Hence, under this event, $\StielEmp(z_0)<1/\lambda$.
\end{proof}
Next, we treat the regime of moderate $m$, namely when $c_0\EffDim(\lambda) \le m < 1.5\EffDim(\lambda)$. The main tool here is the non-asymptotic estimate of Theorem~\ref{thm:Stiel-NonAsymp-Gaussian}. The reason we have to assume $m\gtrsim \EffDim(\lambda)$ is that only then we can guarantee the error term is controlled.

The following simple lemma is useful.
\begin{lemma}\label{lem:DimEff-Contract}
	For any $\mu>0$ and $\gamma\ge 1$, 
	\begin{equation}
		\EffDim(\gamma \mu) \le \EffDim(\mu) \le \gamma \EffDim(\gamma \mu) \,.
	\end{equation}
\end{lemma}
\begin{proof}
	The inequality $\EffDim(\gamma\mu)\le \EffDim(\mu)$ is trivial. As for the other one, 
	\begin{align*}
		\DimEff{\gamma\mu} = \tr(H(H+\gamma \mu I)^{-1}) = \frac{1}{\gamma} \tr(H(\frac{1}{\gamma}H+ \mu I)^{-1}) \overset{(\star)}{\ge} \frac{1}{\gamma} \tr(H(H+ \mu I)^{-1}) = \frac1\gamma \DimEff{\mu}\,,
	\end{align*}
	where $(\star)$ follows because $H\succeq 0$, hence $\frac{1}{\gamma}H +  \mu I \preceq H+ \mu I$ and therefore $H(\frac{1}{\gamma}H+ \mu I)^{-1}\succeq H(H+ \mu I)^{-1}$.
\end{proof}
\begin{lemma}\label{lem:Moderate-n}
	
	Suppose that $S$ has i.i.d. Gaussian entries. Moreover, suppose that $c_0\DimEff{\lambda} \le m < 1.5\DimEff{\lambda}$. 
	There is a numerical constant $M$ such that if moreover $m\ge M$ then
	\[
	\Expt\StielEmp(z_0) \le \frac{14}{15} \cdot \frac{1}{\lambda} \,.
	\]
	In particular, $\Pr(\StielEmp(z_0)>1/\lambda) \le Ce^{-cm}$.
\end{lemma}
\begin{proof}
	Note that the second part follows directly from the first part and Lemma~\ref{lem:StielConc}; we therefore focus on bounding $\Expt\StielEmp(z_0)$. 
	Suppose it were the case that $\Expt\StielEmp_m(z_0)\ge \eta \frac{1}{\lambda}$ where $\eta\le 1$ is some constant. 	
	Apply Theorem~\ref{thm:Stiel-NonAsymp-Gaussian} with $z=z_0$. 
	Let us bound the parameters $\alpha,\beta$ used for error control. First,
	\begin{align}
		\alpha := |z_0|\Expt\StielEmp(z_0) \ge \frac{5}{12}\lambda\cdot \eta\frac{1}{\lambda} \ge \frac{5}{12}\eta\,.
	\end{align}
	Next, since $\Expt\StielEmp(z_0)\le \frac{1}{|z_0|}=\frac{12}{5}\frac{1}{\lambda}$ we have $1/\Expt\StielEmp(z_0) \ge \frac{5}{12}\lambda$,
    so
	\begin{align}
		\beta := \frac1m \DimEff{1/\Expt\StielEmp(z_0)} \le \frac1m \DimEff{\frac{5}{12}\lambda} \overset{(\star)}{\le} \frac{12}{5}\frac1m\DimEff{\lambda} \overset{(\star\star)}{\le} \frac{12}{5}\frac{1}{c_0}\,,
	\end{align}
	where $(\star)$ follows from Lemma~\ref{lem:DimEff-Contract} and $(\star\star)$ from the assumption $m\ge c_0\EffDim(\lambda)$. In particular, the error $e_{m}(z_0,H)$ in \eqref{eq:thm:Stiel-NonAsymp-Gaussian-Error} can be bound as $e_m(z_0,H)=\BigOh_{\eta}(\frac{1}{\lambda\sqrt{m}})$. Furthermore, as $\Expt\StielEmp(z_0)\ge \eta \frac{1}{\lambda}$ by assumption,
	\begin{equation}
		\beta \ge \frac1m \DimEff{\frac{1}{\eta}\lambda} \ge \eta \frac1m\DimEff{\lambda}> \frac{2}{3}\eta \,,
	\end{equation}
	where we used Lemma~\ref{lem:DimEff-Contract} and the assumption $m<1.5\DimEff{\lambda}$. With these estimates in mind, consider \eqref{eq:thm:Stiel-NonAsymp-Gaussian},
	\begin{align*}
		\Expt\Stiel(z_0) = -\frac{1}{z_0}(1-\beta-z_0\eps_m(z_0,H)) = \frac{1}{\lambda}\cdot \frac{12}{5}(1-\beta - z_0 e_m(z_0,H)) \le  \frac{1}{\lambda}\cdot \frac{12}{5}(1-\frac23\eta + \BigOh_\eta(\frac{1}{\sqrt{m}})) \,.
	\end{align*}
	As $\Expt\StielEmp_m(z_0)\ge \eta \frac{1}{\lambda}$, we deduce from the above
	\begin{align*}
		\eta \le \frac{12}{5}(1-\frac23\eta + \BigOh_\eta(\frac{1}{\sqrt{m}}))\,,
	\end{align*}
	so that rearranging,
	\begin{align*}
		\eta \le \frac{12}{13}(1 + \BigOh_\eta(\frac{1}{\sqrt{m}}))\,.
	\end{align*}
	By requiring that $m\ge C_\eta $ for large enough $C_\eta$, we can ensure that the error term is as small a constant as we like, so that necessarily (for example) $\eta\le \frac{14}{15}$. 
	
\end{proof}

Next, we show that if the algorithm has reached $m\ge 2\DimEff{\lambda}$, then it will halt in that iteration. We start with two helper lemmas.
\begin{lemma}\label{lem:Stil-Expt-Convexity}
	Suppose that $\Expt[S^\T S]=I$. Then for all $z<0$,
	\begin{equation}\label{eq:lem:Stil-Expt-Convexity}
		\Expt\StielEmp(z) \ge -\frac1z\left(1-\frac1m \DimEff{-z}\right) \,.
	\end{equation}
	(Note that since $\Expt\StielEmp(z)\le -1/z$, we have $\DimEff{-z}\ge\DimEff{1/\Expt\StielEmp(z)}$, so that \eqref{eq:lem:Stil-Expt-Convexity} is perfectly consistent with \eqref{eq:thm:Stiel-NonAsymp-Gaussian}.)
\end{lemma}
\begin{proof}
	Recall that the matrices $SHS^\T\in\RR^{m\times m}$ and $H^{1/2}S^\T S H^{1/2}$ have the same non-zero eigenvalues, and therefore
	\[
	\Expt \StielEmp(z) := \frac1m\Expt\tr(SHS^\T-zI)^{-1} = \frac1m\Expt\tr(H^{1/2}S^\T S H^{1/2}-zI)^{-1} + \frac{1}{m}(d-m)\frac1z \,.
	\]
	Note that since $z<0$, the matrix function $A\mapsto \tr(A-zI)^{-1}$ is convex on the cone of PSD matrices $A\succeq 0$. Thus, by Jensen's inequality,
	\[
	\frac1m\Expt\tr(H^{1/2}S^\T S H^{1/2}-zI)^{-1} \ge \frac1m\tr(\Expt[H^{1/2}S^\T S H^{1/2}]-zI)^{-1} = \frac1m\tr(H-zI)^{-1}\,,
	\]
	so that 
	\[
	\Expt \StielEmp(z) \ge \frac1m\tr(H-zI)^{-1} +\frac1m (d-m)\frac1z = \frac1m\tr[(H-zI)^{-1}+\frac1z I]-\frac1z = -\frac1z\left( 1-\frac1m \DimEff{-z} \right) \,.
	\]
\end{proof}
The next lemma formally establishes that the inverse function of $\Stiel(z)$ is increasing on an appropriate interval. As in \eqref{eq:Psi-Def}, denote $\Psi : (0,\infty) \to \RR$, 
\begin{equation}\label{eq:Stiel-Inv-Def}
	\Psi(s) = -\frac{1}{s}(1-\frac1m \DimEff{1/s}) \,,
\end{equation}
so that for all $z<0$, $\Psi(\Stiel(z))=z$. Note however that in general, $\Range(\Stiel)$ may only be a proper subset of $(0,\infty)$; nonetheless, \eqref{eq:Stiel-Inv-Def} is well-defined. 
\begin{lemma}\label{lem:Stiel-Inv-Monotonicity}
	Let $\mu>0$, and suppose that $m\ge \DimEff{\mu}$. Then $\StielInv$ is increasing on $(0,1/\mu)$. 
\end{lemma}
\begin{proof}
	Write 
	\[
	\StielInv(s) = -\frac1s + \frac1m \tr(H(I+sH)^{-1})\,,
	\]
	so that its derivative is
	\begin{align}
		\StielInv'(s) 
		&= \frac1{s^2} - \frac{1}{m}\tr(H^2(I+sH)^{-2}) \nonumber \\
		&= \frac{1}{s^2}\left( 1 - \frac1m\tr(H^2(H+1/s)^{-2}) \right) \nonumber	\\
		&\ge \frac{1}{s^2}\left( 1 - \frac1m\tr(H(H+1/s)^{-1}) \right) 
		= \frac{1}{s^2}\left( 1 - \frac1m\DimEff{1/s} \right)\,,
		\label{eq:lem:Stiel-Inv-Monotonicity}
	\end{align}
	where the inequality follows since $0\preceq H(H+1/s)^{-1}\preceq I$. When $s<1/\mu$, we have $1/s>\mu$ so that $\DimEff{1/s}<\DimEff{\mu}$, hence $\StielInv'(s)>0$.	
\end{proof}

We now prove:
\begin{lemma}\label{lem:Large-n}
	Suppose that $S$ has i.i.d. Gaussian entries, and moreover that $m\ge 2\DimEff{\lambda}$. 
	There is a numerical $M$ such that whenever $m\ge M$,
	\[
	\Expt\StielEmp(z_0) \ge \frac{24}{23} \cdot \frac{1}{\lambda} \,.
	\]
	In particular, $\Pr(\StielEmp(z_0)\le 1/\lambda) \le Ce^{-cm}$.
\end{lemma}
\begin{proof}
	Suppose that $\Expt\StielEmp(z_0)\le \frac{1}{(1-\eta)}\frac{1}{\lambda}$ for some $\eta\in (0,1)$; we shall show that $\eta$ has to be somewhat large so to not generate a contradiction.
	First, by Lemma~\ref{lem:Stil-Expt-Convexity} (with $z=-\lambda$),
	\begin{equation*}
		\Expt\StielEmp(z_0) \ge \Expt\StielEmp(-\lambda) = \frac{1}{\lambda}(1-\frac1m\DimEff{\lambda}) = \frac12 \cdot \frac{1}{\lambda}\,,
	\end{equation*}
	where the last inequality follows since $m\ge 2\DimEff{\lambda}$. Of course, this lower bound does not suffice for our purposes: we want a bound which is $>\frac1{\lambda}$ (strictly). It does allow us, however, to control the error term in Theorem~\ref{thm:Stiel-NonAsymp-Gaussian} at $z=z_0$ (specifically, to lower bound $\alpha$). We can bound the parameters $\alpha,\beta$ in Theorem~\ref{thm:Stiel-NonAsymp-Gaussian} as
	\begin{align*}
		\alpha &:= |z_0|\Expt\StielEmp(z_0)\ge \frac{5}{12}\lambda\cdot \frac{1}{2}\frac1{\lambda}\ge \frac{5}{24}\,, \\
		\beta &:= \frac1m\DimEff{1/\Expt\StielEmp(z_0)} \le \frac1m\DimEff{-z_0} = \frac1m\DimEff{\frac{5}{12}\lambda} \le \frac{12}{5}\frac1m\DimEff{\lambda} \le \frac{6}{5} \,,
	\end{align*}
	where, in upper bounding $\beta$, we used $\Expt\StielEmp(z_0)\le \frac{1}{-z_0}$ and Lemma~\ref{lem:DimEff-Contract}.
	With these, Theorem~\ref{thm:Stiel-NonAsymp-Gaussian} implies
	\begin{equation*}
		\Expt\StielEmp(z_0) = -\frac{1}{z_0}(1-\frac1m\DimEff{1/\Expt\StielEmp_m(z_0)}) + \BigOh(\frac{1}{\lambda\sqrt{m}}) \,.
	\end{equation*}
	Multiplying by $z_0$ and dividing by $\Expt\StielEmp(z_0)$, 
	\begin{align*}
		z_0 
		&= -\frac{1}{\Expt\StielEmp(z)}\left(1 - \frac1m\DimEff{1/\Expt\StielEmp(z)}\right) + \BigOh(\frac{\lambda}{\sqrt{m}}) \\
		&= \StielInv(\Expt\StielEmp(z))	+ \BigOh(\frac{\lambda}{\sqrt{m}})\,,
	\end{align*}
	where in modifying the error term we used $\Expt\StielEmp(z_0)\gtrsim \frac1\lambda$. 
	
	Note that $m\ge 2\DimEff{\lambda}$ implies $m\ge\DimEff{\frac12\lambda}$ by Lemma~\ref{lem:DimEff-Contract}. By Lemma~\ref{lem:Stiel-Inv-Monotonicity}, this implies that $\StielInv$ is increasing on $(0,2\frac1\lambda)$. Thus, if we further assume that $\eta<1/2$, 
	then $\StielInv$ is increasing on $(0,\frac{1}{1-\eta}\frac1\lambda)$, and so $\StielInv(\Expt\StielEmp(z)) \le \StielInv(\frac1{1-\eta}\frac1\lambda)$. Using this with the previous display, and writing $z_0=-\frac5{12}\lambda$,
	\begin{align*}
		-\frac5{12} \lambda 
		&\le \StielInv(\frac{1}{1-\eta}\frac1\lambda) +  \BigOh(\frac{\lambda}{\sqrt{m}}) \\
		&= -(1-\eta){\lambda} \left( 1 - \frac1m\DimEff{(1-\eta)\lambda} \right) +  \BigOh(\frac{\lambda}{\sqrt{m}})  \\
		&\overset{(\star)}{\le} -(1-\eta){\lambda} \left( 1 - \frac{1}{1-\eta}\frac1m\DimEff{\lambda} \right) +  \BigOh(\frac{\lambda}{\sqrt{m}}) \\
		&\overset{(\star\star)}{\le} -(1-\eta){\lambda} \left( 1 - \frac{1}{1-\eta}\frac12 \right) +  \BigOh(\frac{\lambda}{\sqrt{m}}) \\
		&= -\lambda(\frac12-\eta ) + \BigOh(\frac{\lambda}{\sqrt{m}})\,,
	\end{align*}
	where $(\star)$ uses $\DimEff{(1-\eta)\lambda}\le \frac{1}{1-\eta}\DimEff{\lambda}$ (Lemma~\ref{lem:DimEff-Contract}) and $(\star\star)$ uses the assumption $m\ge 2\DimEff{\lambda}$. Dividing by $(-\lambda)$,
	\begin{align*}
		\frac{5}{12} \ge \frac12-\eta + \BigOh(\frac{1}{\sqrt{m}}) \,.
	\end{align*}
	For large enough $m\ge M$, the $\BigOh(\cdot)$ term is $\le \frac{1}{24}$, therefore $\eta\ge \frac1{24}$. That is, we deduce that necessarily $\Expt\StielEmp(z_0)\ge \frac{1}{1-\frac1{24}}\frac1{\lambda} = \frac{24}{23}\frac{1}{\lambda}$.
\end{proof}

We are ready to conclude the proof of Theorem~\ref{thm:Alg1-Gauss}.

\begin{proof}
(Of Theorem~\ref{thm:Alg1-Gauss}.)
	Let $m_j=2^jm_0$, $j=0,1,\ldots$, be the value of $m$ at iteration $j$ of the algorithm. Set 
	\[
	J^- = \lceil \log_2(1.5\DimEff{\lambda}/m_0)\rceil, \qquad J^+ = \lceil \log_2(2\DimEff{\lambda}/m_0)\rceil\,,
	\]
	chosen so that: 1) 
	$J^-$ is the smallest $j$ such that $m_j\ge 1.5\DimEff{\lambda}$; 2) $2\DimEff{\lambda} \le m_{J^+} < 4\DimEff{\lambda}$. Thus, to establish \eqref{eq:thm:Alg1-Gauss} it suffices to show that the algorithm halts at an iteration $j$ satisfying $J^-\le j\le J^+$. 
	Accordingly, the error probability can be bounded as
	\begin{align*}
		\Pr(\textrm{\eqref{eq:thm:Alg1-Gauss} does not hold} ) \le \sum_{j=0}^{J^- -1}\Pr(\StielEmp_{m_j}(z_0)> 1/\lambda) + \Pr(\StielEmp_{m_{J^+}}(z_0)\le 1/\lambda) \,.
	\end{align*}
	Assume that $m_0\ge M$ is large enough. By Lemmas~\ref{lem:Small-n} and \ref{lem:Moderate-n}, for all $j\le J^--1$, $\Pr(\StielEmp_{m_j}(z_0)> 1/\lambda)\le Ce^{-cm_j}$. By Lemma~\ref{lem:Large-n}, $\Pr(\StielEmp_{m_{J^+}}(z_0)\le 1/\lambda)\le Ce^{-cm_{J^+}}$. Thus,
	\begin{align*}
		\Pr(\textrm{\eqref{eq:thm:Alg1-Gauss} does not hold} ) \le \sum_{j=0}^{J^- -1}Ce^{-c2^j m_0} + Ce^{-c2^{J^+}m_0} \le \sum_{j=0}^{\infty} Ce^{-c2^jm_0} \lesssim e^{-cm_0} \,.
	\end{align*}
	We deduce that if $m_0\gtrsim \log(1/\delta)$, then the error probability is $\le \delta$.
\end{proof}

\newpage

\section{Proof of Theorem~\ref{thm:Alg2-Gauss}}
\label{sec:proof-thm:Alg2}

\begin{proof}[\unskip\nopunct]
	
	Recall the (deterministic) Marchenko-Pastur equation \eqref{eq:MP-Eq}; rearranging yields
	\begin{equation}
		z = -\frac{1}{\Stiel(z)}\left( 1 - \frac1m \DimEff{1/\Stiel(z)} \right)\,,
	\end{equation}
	so that setting $\Stiel(-\tilde{\lambda})=1/\lambda$ yields $\tilde{\lambda}=\lambda(1-\frac1m\DimEff{\lambda})$. Further recall that we do not have access to $\Stiel(z)$, but instead to its random counterpart $\StielEmp(z)$. We know that its expectation $\Expt\StielEmp(z)$ approximately satisfies the Marcheko-Pastur equation with an error term (Theorem~\ref{thm:Stiel-NonAsymp-Gaussian}), and that we have concentration around the expectation (Lemma~\ref{lem:StielConc}). The proof of Theorem~\ref{thm:Alg2-Gauss} down to controlling and propagating these error terms.
	
	\begin{lemma}\label{claim:proof-thm:Alg2-claim1}
		Suppose that $m\ge 1.5\DimEff{\lambda}$ and $m\ge M$ for large enough $M$.
		
		A unique root $z^\star<0$ such that $\Expt\StielEmp(z^\star)=\frac1\lambda$ exists, and $z^\star<-c\lambda$ for small enough $c>0$.
		
		Consider a sufficiently small $\BigOh(\lambda)$-neighborhood of $z^\star$, $I^\star=[z^\star-\eta\lambda,z^\star+\eta\lambda]$ for small enough $\eta>0$. Then,
		\begin{enumerate}
			\item $\Expt\StielEmp(z)\gtrsim \frac{1}{\lambda}$ for all $z\in I^\star$.
			\item $|\Expt\StielEmp(z)-\Expt\StielEmp(z')|\lesssim |z-z'|\frac{1}{\lambda^2}$ for $z,z'\in I^\star$. 
		\end{enumerate}
	\end{lemma}
	\begin{proof}
		The existence of $z^\star$, as well as the bound $\Expt\StielEmp_m(z)\ge \eta \frac1\lambda$, can be shown by an argument similar to Lemma~\ref{lem:Large-n}; we omit the technical details. Specifically, one can show that for small enough $c$,
		$\Expt\StielEmp(-c\lambda)\ge (1+c)\frac1\lambda$. Since $\Expt\Stiel(\cdot)$ is continuous increasing, and $\Expt\StielEmp(-\infty)=0$, a root $\Expt\StielEmp(z^\star)=\frac1\lambda$ exists and $z^\star<-c\lambda$. 
		
		Next, we prove Item 2 above. We have 
		\begin{align*}
			|\Expt\StielEmp_m(z)-\Expt\StielEmp_m(z')| = \Expt\frac1m\tr[(S_mHS_m^\T-zI)^{-1}(S_mHS_m^\T-z'I)^{-1}] |z-z'| \le \frac{|z-z'|}{|zz'|} \lesssim \frac{1}{\lambda^2}|z-z'|\,,
		\end{align*}
		since $|z|,|z'|\gtrsim \lambda$, as $z,z'\in I_\eta$.
		
		Item 1 follows immediately from Item 2.
	\end{proof}

	We need to establish a result of the following kind: if $z<0$ is such that $\Expt\StielEmp(z)\approx \frac{1}{\lambda}$, then necessarily $z\approx z^\star$. This necessitates a {lower bound} (rather than an upper bound, such as given in Lemma~\ref{claim:proof-thm:Alg2-claim1}, Item 2) on the derivative of $\Expt\StielEmp(z)$ around $z^\star$.
	\begin{lemma}\label{lem:InverseLipschitz}
		Suppose that $m\ge 1.5\DimEff{\lambda}$ and $m\ge M$ for large enough $M$. 
		
		For a sufficiently small $\BigOh(\lambda)$-neighborhood of $z^\star$, $z^\star \in I^\star$,
		\begin{equation}
			\label{eq:lem:InverseLipschitz}
			|\Expt\StielEmp(z)-\Expt\StielEmp(z')| \gtrsim \frac{1}{\lambda^2}|z-z'| - \BigOh(\frac{1}{\lambda\sqrt{m}}) \,.
		\end{equation}
	\end{lemma}
	\begin{proof}
		By Theorem~\ref{thm:Stiel-NonAsymp-Gaussian}, having established that $\Expt\StielEmp(z)\gtrsim \frac{1}{\lambda}$ for all $z\in I^\star$, uniformly for all $z\in I^\star$,
		\begin{equation*}\label{eq:lem:InverseLipschitz-0}
		\Expt\StielEmp(z) = -\frac{1}{z}\left( 1-\frac1m\DimEff{1/\Expt\StielEmp(z)}\right) + \BigOh(\frac{1}{\lambda\sqrt{m}})  \,.
		\end{equation*}
		uniformly for all $z\in I^\star$. Dividing by $\Expt\StielEmp(z)$ and multiplying by $z$, 
		\begin{equation}\label{eq:lem:InverseLipschitz-1}
			z = \StielInv(\Expt\StielEmp(z)) + \BigOh(\frac{\lambda}{\sqrt{m}}) \,,
		\end{equation}
		where $\StielInv(s)=-\frac{1}{s}(1-\frac1m\DimEff{1/s})$ is the inverse function of $\Stiel(z)$ (see e.g. \eqref{eq:Stiel-Inv-Def}). 
		Note that for $z,z'\in I^\star$, $|\Expt\StielEmp(z)-\Expt\StielEmp(z')|\lesssim |z-z'|\frac{1}{\lambda^2} = \BigOh(\frac{1}{\lambda})$. Thus, to conclude the proof of the lemma, it suffices to give an upper bound on the derivative (or Lipschitz constant) of $\StielInv$ in a small $\BigOh(1/\lambda)$ neighborhood of $\Expt\StielEmp(z^\star)=\frac1\lambda$.
		
		Indeed, if $s=\frac{1}{\lambda}(1+\eps)$ for small $\eps=\BigOh(1)$, then by Lemma~\ref{lem:DimEff-Contract}, $\frac{1}{1+|\eps|}\DimEff{\lambda} \le \DimEff{1/s}\le \frac{1}{1-|\eps|}\DimEff{\lambda}$. Since $\frac1m\DimEff{\lambda}\le (1.5)^{-1}=\BigOh(1)$, $|\frac{1}{m}\DimEff{1/s}-\frac{1}{m}\DimEff{1/\lambda}|=\BigOh(\eps)$. Thus, for $s_1,s_2$ in a sufficiently small $\BigOh(\frac{1}{\lambda})$-neighborhood of $\frac1\lambda$, $|\frac{1}{m}\DimEff{1/s_1}-\frac{1}{m}\DimEff{1/s_2}|=\BigOh(\lambda|s_1-s_2|)$, and so $|\StielInv(s_1)-\StielInv(s_2)|=\BigOh(\lambda^2|s_1-s_2|)$. 
		Using \eqref{eq:lem:InverseLipschitz-1}, we deduce that for $z_1,z_2$ in a small enough $\BigOh(\lambda)$ neighborhood of $z^\star$, 
		\begin{align*}
			|z_1-z_2| 
			&= \left| \StielInv(\Expt\StielEmp(z_1)) - \StielInv(\Expt\StielEmp(z_2)) + \BigOh(\frac{\lambda}{\sqrt{m}}) \right| \\
			&\lesssim \lambda^2|\Expt\StielEmp(z_1)-\Expt\StielEmp(z_2)| + \BigOh(\frac{\lambda}{\sqrt{m}}) \,.
		\end{align*}
		The desired estimate \eqref{eq:lem:InverseLipschitz} follows dividing by $\lambda^2$.
	\end{proof}
	
	We are ready to conclude the proof of Theorem~\ref{thm:Alg2-Gauss}.
	Let $\eps>0$, $\delta\in (0,1)$ be, respectively, the precision on confidence parameters. 
	Suppose that $\eps=\BigOh(1)$ is small, such that $[z^\star-2\lambda\eps,z^\star+2\lambda\eps]\subseteq I^\star$, therefore setting 
	\begin{align}
		z^\star_\pm = z^\star \pm \lambda\eps 
	\end{align}
	we have $z^\star_\pm \in I^\star$. By 
	Lemma~\ref{lem:InverseLipschitz}, Eq. \eqref{eq:lem:InverseLipschitz}, provided that $m\gtrsim 1/\eps^2$ is large enough,  
	\begin{align*}
		\Expt\StielEmp(z_-^\star) &\le \frac{1}{\lambda} - C_1\eps\frac{1}{\lambda} + \BigOh(\frac{1}{\lambda\sqrt{m}}) \le \frac{1}{\lambda} -  C_2\eps\frac{1}{\lambda}\,, \\ 
		\Expt\StielEmp(z_+^\star) &\ge \frac{1}{\lambda} + C_1\eps\frac{1}{\lambda} - \BigOh(\frac{1}{\lambda\sqrt{m}}) \ge \frac{1}{\lambda} +  C_2\eps\frac{1}{\lambda} \,.
 	\end{align*}
 	By Lemma~\ref{lem:StielConc}, if $m \gtrsim \frac{1}{\eps^2}\log(1/\delta)$ is large enough, then w.p. $1-\delta$, $|\StielEmp(z_-^\star)-\Expt\StielEmp(z_-^\star)|,|\StielEmp(z_+^\star)-\Expt\StielEmp(z_+^\star)|\le C_2\frac{1}{\lambda}\eps/2$. On this event,
	\begin{align*}
	\StielEmp(z_-^\star) \le \frac{1}{\lambda} - \frac12C_2\eps\frac{1}{\lambda} < \frac1\lambda\,, \qquad 
	\StielEmp(z_+^\star) \ge \frac{1}{\lambda} + \frac12C_2\eps\frac{1}{\lambda} > \frac1\lambda\,.
	\end{align*}
	In particular, the output $\hat{\lambda}$ of Algorithm~\ref{Alg:2}, being the solution $\StielEmp(-\hat{\lambda})=1/\lambda$, satisfies $z^\star_- \le -\hat{\lambda} \le z^\star_+$, so $|\hat{\lambda} - (-z^\star)| \le \eps$. Finally, by Eq. \eqref{eq:lem:InverseLipschitz-1},
	\[
	z^\star = \StielInv(1/\lambda) + \BigOh(\frac{\lambda}{\sqrt{m}}) = -\tilde{\lambda} + \BigOh(\frac{\lambda}{\sqrt{m}}) = -\tilde{\lambda} + \BigOh(\lambda\eps)
	\]
	whenever $m\gtrsim 1/\eps^2$ large enough. 
	Since $\tilde{\lambda}\ge \lambda(1-(1.5)^{-1})=\frac13\lambda = \Omega(\lambda)$, this implies
	\[
	\frac{|\hat{\lambda}-\tilde{\lambda}|}{\tilde{\lambda}} = \BigOh(\eps) \,.
	\]
	Thus, Theorem~\ref{thm:Alg2-Gauss} is proved.
	
\end{proof}

In the sequel, the following bound on $\Expt|\hl-\tilde{\lambda}|$ will be useful.
\begin{lemma}\label{lem:thm:Alg-2-Gauss-Expt}
	Suppose that $m\ge 1.5\DimEff{\lambda}$ and $m\ge M$ for large enough $M$. For any $p>0$,
	\begin{align}
		\Expt|\hl-\tilde{\lambda}|^p \le C_p \lambda^p m^{-p/2}\,.
	\end{align}
\end{lemma}
\begin{proof}
	Let $t_0>0$ be a small enough constant, such that $[z^\star-\lambda t_0, z^\star+\lambda t_0]\subseteq I^\star$. By the above, $\Pr(|\hl-\tilde{\lambda}|\ge t_0\lambda ) = \Pr(-\hl\notin [z^\star-\lambda t_0, z^\star+\lambda t_0]) \le Ce^{-cm}$. Set $z^\star_\pm(t) = z^\star \pm \lambda t$, so that by the above, $\Pr(\StielEmp_m(z^\star_+(t) )<1/\lambda),\Pr(\StielEmp_m(z^\star_-(t) )>1/\lambda)\le Ce^{-cmt^2}$ for all $t\le t_0$. We have
	\[
	\Pr(|\hl-\tilde{\lambda}|\ge t\lambda) = \Pr(-\hl\notin [z^\star_-(t),z^\star_+(t)]) 
	= \Pr(\{\StielEmp_m(z^\star_+(t) )<1/\lambda\}\cup \{\StielEmp_m(z^\star_-(t) )>1/\lambda\}) \le 2Ce^{-cmt^2} \,.
	\]
	Thus,
	\begin{align*}
		\Expt|\hl-\tilde{\lambda}|^p 
		= \Expt[|\hl-\tilde{\lambda}|^p \Indic{|\hl-\tilde{\lambda}|\le \lambda t_0}] + \Expt[|\hl-\tilde{\lambda}|^p \Indic{|\hl-\tilde{\lambda}|> \lambda t_0}] \le \Expt[|\hl-\tilde{\lambda}| \Indic{|\hl-\tilde{\lambda}|\le \lambda t_0}] + C\lambda^p e^{-cm}\,,
	\end{align*}
	and
	\begin{align*}
		\Expt[|\hl-\tilde{\lambda}|^p \Indic{|\hl-\tilde{\lambda}|\le \lambda t_0}] = \int_{0}^{t_0} \lambda^{p}pt^{p-1} \Pr(|\hl-\tilde{\lambda}|\ge \lambda t)dt \lesssim p\lambda^p \int_{0}^{t_0} t^{p-1} e^{-cmt^2} dt \le C_p \lambda^p m^{-p/2} \,.
	\end{align*}
	And so, $\Expt|\hl-\tilde{\lambda}|^p\lesssim \lambda\frac{1}{\sqrt{m}} + \lambda e^{-cm} = \BigOh(\lambda\frac{1}{\sqrt{m}})$.
\end{proof}

\newpage

\section{Proof of Theorem~\ref{thm:Smallest-Eigenvalue-Gauss}}
\label{sec:proof-thm:Smallest-Eigenvalue}

\begin{proof}
	[\unskip\nopunct]

By definition, if $\tilde{H}\preceq H$ (in PSD order) then $S\tilde{H}S^\T \preceq SHS^\T$, in particular $\lambda_{\min}(SHS^\T)\ge \lambda_{\min}(S\tilde{H}S^\T)$. Choose $\tilde{H}=H(\mu H + I)^{-1}$, noting that $\|\tilde{H}\|\le 1/\mu$ can be bounded irrespective of $\|H\|$.

We now lower bound $\lambda_{\min}(S\tilde{H}S)$ using a standard net argument, cf. \citep{vershynin2018high}.
For any $\eps>0$, let $\m{N}_\eps$ be an $\eps$-net of $\m{S}^{m-1}$ of minimum size. One can show that $|\m{N}_\eps|\le \left( 1+\frac{2}{\eps} \right)^m$. 
We have 
\begin{align}\label{eq:Smallest-Eigenvalue-aux-1}
    \lambda_{\min}(S\tilde{H}S^\T) = \min_{\|x\|=1}x^\T S\tilde{H}S^\T x \ge \min_{x\in \m{N}_\eps}x^\T S\tilde{H}S^\T x - 2\|S\tilde{H}S^\T\|\eps \,.
\end{align}
Note that for fixed $x\in \m{S}^{m-1}$, $y=S^\T x$ has independent, mean zero, sub-Gaussian entries with $\max_{1\le i \le d}\|y_i\|_{\psi_2}=O(1)$ and variance $1/m$.

We first give a high-probability upper bound on $\|S\tilde{H}S^\T\|$. Let $\m{N}_{1/4}$ a $1/4$-net of minimum size, so that 
\begin{align*}
	\|S\tilde{H}S^\T \| = \max_{\|x\|=1}x^\T S\tilde{H}S x \le \max_{x\in \m{N}_{1/4}} x^\T S\tilde{H}S x + 2 \cdot \|S\tilde{H}S^\T\| \cdot \frac{1}{4}\,,
\end{align*}
hence $\|S\tilde{H}S^\T \| \le 2  \max_{x\in \m{N}_{1/4}} x^\T S_m\tilde{H}S_m x$.
By the Hanson-Wright inequality, Lemma~\ref{eq:HansonWright} (see also \cite[Theorem 6.2.1]{vershynin2018high}),
\begin{align*}
	\Pr(\|S\tilde{H}S^\T \|\ge 2m^{-1}\tr(\tilde{H}) + t) 
	&\le |\m{N}_{1/4}|\max_{x\in\m{N}_{1/4}} \Pr\left( x^\T S\tilde{H}S x \ge m^{-1}\tr(\tilde{H}) + t/2 \right) \\
	&\le 9^m 2\exp\left( -c\min\{ \frac{m^2t^2}{\|\tilde{H}\|_F^2}, \frac{mt}{\|\tilde{H}\|} \} \right) \,.
\end{align*}
Note that 
\begin{align*}
	\tr(\tilde{H})= \frac{1}{\mu}{\DimEff{1/\mu}},\qquad \|\tilde{H}\| \le 1/\mu,\qquad \|\tilde{H}\|_F\le \sqrt{\|\tilde{H}\|\tr(\tilde{H})} \le \frac{1}{\mu} \sqrt{\DimEff{1/\mu}}\,.
\end{align*}
Set $t= C_1\max\{\frac{1}{\sqrt{m}}\|\tilde{H}\|_F,\|\tilde{H}\|\}$ large enough. For such choice, w.p. $1-Ce^{-cm}$, 
\begin{align}
	\|S\tilde{H}S\| 
	&\le \frac{2}{m}\tr(\tilde{H}) + C_1\max\{ \frac{1}{\sqrt{m}}\|\tilde{H}\|_F, \|\tilde{H}\| \} \nonumber \\
	&\lesssim \frac{1}{\mu}\frac{{\DimEff{1/\mu}}}{m} + \frac{1}{\mu}\max\{1,\sqrt{\frac{d_{1/\mu}}{m}}\} \label{eq:Smallest-Eigenvalue-aux-2}
\end{align}

Next, we give a high-probability lower bound on $\min_{x\in\m{N}_\eps}x^\T S\tilde{H}S^\T x$ in \eqref{eq:Smallest-Eigenvalue-aux-1}. Again by Hanson-Wright,
\begin{align*}
	\Pr(\min_{x\in\m{N}_\eps}x^\T S\tilde{H}S^\T x \le  m^{-1}\tr(\tilde{H})-t) \le \left(1+\frac{2}{\eps}\right)^m  2\exp\left( -c\min\{ \frac{m^2t^2}{\|\tilde{H}\|_F^2}, \frac{mt}{\|\tilde{H}\|} \} \right)\,,
\end{align*}
so that w.p. $1-Ce^{-cm}$, for small $\eps>0$,
\begin{align}
	\min_{x\in\m{N}_\eps}x^\T S\tilde{H}S^\T x \ge \frac{1}{\mu}\frac{{\DimEff{1/\mu}}(H)}{m} - C_2 \log(1+1/\eps)\frac{1}{\mu}\max\{1, \sqrt{\frac{{\DimEff{1/\mu}}(H)}{m}}\} \label{eq:Smallest-Eigenvalue-aux-3}
\end{align}

Inserting \eqref{eq:Smallest-Eigenvalue-aux-2}-\eqref{eq:Smallest-Eigenvalue-aux-3} into \eqref{eq:Smallest-Eigenvalue-aux-1} implies that w.p. $1-Ce^{-cm}$,
\begin{equation}
    \lambda_{\min}(SHS^\T) \ge (1-O(\eps))\frac{1}{\mu}\frac{\DimEff{1/\mu}}{m} - O(\eps+\log(1+1/\eps))\frac{1}{\mu}\max\{1, \sqrt{\frac{{\DimEff{1/\mu}}(H)}{m}}\} .
\end{equation}
By choosing $\eps>0$ a sufficiently small numerical constant, we can deduce the following: there are numerical constants $c,C,c_1,c_2>0$ such that if $\DimEff{1/\mu}/m>1/c_1$, then  
w.p. $1-Ce^{-cm}$, $\lambda_{\min}(\tilde{H}) \ge c_2 \frac{1}{\mu}\frac{\DimEff{1/\mu}}{m}$. To recover the claimed result, use $\mu=1/\eta$.

\end{proof}

\newpage

\section{Proof of Lemma~\ref{lem:StielConc}}
\label{sec:proof-lem:StielConc}

The following argument is standard (cf. \citep{bai2010spectral}) and presented for completeness.

\begin{proof}[\unskip\nopunct]
	Let $r_1,\ldots,r_m\in \RR^d$ be the rows of $S$; that is, $S^\T=[r_1|\ldots|r_m]$. Let
	\begin{equation}
		\hat{\Sigma}_m = H^{1/2}S^\T S H^{1/2} =  \frac1m\sum_{i=1}^n H^{1/2}r_ir_i^\T H^{1/2} \,.
	\end{equation}
	This is a sample covariance matrix, corresponding to $m$ samples with population covariance $\Expt[\hat{\Sigma}_m]=H$. Recall that $\hat{\Sigma}_m\in\RR^{d\times d}$ and $S_mHS_m^\T\in\RR^{m\times m}$ have the same non-zero eigenvalues. Denote
	\begin{align}
		\Pm(z) = (\hat{\Sigma}_m+zI)^{-1}\,,
	\end{align}
	so that 
	\begin{align}\label{eq:Stiel-To-Sigma}
		\StielEmp(z)=m^{-1}\tr(\Pm(z)) + m^{-1}(d-m)\frac{1}{z}\,.
	\end{align}
	
	Central to the proof is the following leave-one-out decomposition:
	\begin{align}
		\Pm(z) &= \PmMinK(z) - \frac{\frac{1}{m}\PmMinK(z) H^{1/2} r_k r_k^\T H^{1/2}\PmMinK(z) }{1 + \frac{1}{m}r_k^\T H^{1/2} \PmMinK(z) H^{1/2} r_k}, \label{eq:PmMinK}\\
		\PmMinK(z) &= (\hat{\Sigma}_m - \frac1m H^{1/2}r_kr_k^\T H^{1/2}+zI)^{-1}, \qquad k=1,\ldots,m \,.
	\end{align}
	The above can be readily shown by the Sherman-Morrison lemma (Lemma~\ref{lem:Sherman-Morrison}).
	
	We shall now prove Lemma~\ref{lem:StielConc} by a standard martingale concentration argument. Let $\m{F}_k$ be the $\sigma$-algebra generated by $r_1,\ldots,r_k$, $k=0,1,\ldots, m$, and $\Expt_{\le k}[\cdot]:=\Expt[\cdot|\m{F}_k]$. Decompose into a sum of martingale differences,
	\begin{align}
		\StielEmp(z)-\Expt[\StielEmp(z)] 
		&= m^{-1}\sum_{k=1}^m (\Expt_{\le k}-\Expt_{\le k-1})[\tr\Pm(z)] \nonumber \\
		&= m^{-1}\sum_{k=1}^m (\Expt_{\le k}-\Expt_{\le k-1})[D_{m,k}(z)]\,,\label{eq:proof-lem:StielConc-aux1}
	\end{align}
	where, using \eqref{eq:PmMinK},
	\begin{align}\label{eq:Qmk-Def}
		D_{m,k}(z) = \frac{\frac{1}{m}r_k^\T H^{1/2}(\PmMinK(z))^2H^{1/2}r_k}{1 + \frac{1}{m}r_k^\T H^{1/2} \PmMinK(z) H^{1/2} r_k}\,.
	\end{align}
	Note that $0\preceq \PmMinK(z)\preceq \frac{1}{-z}I$, hence
	\begin{align*}
		|D_{m,k}(z)|  \le \frac{1}{-z}\cdot \frac{\frac{1}{m}r_k^\T H^{1/2}(\PmMinK(z))H^{1/2}r_k}{1 + \frac{1}{m}r_k^\T H^{1/2} \PmMinK(z) H^{1/2} r_k} \le \frac{1}{-z}\,.
	\end{align*}
	Thus, \eqref{eq:proof-lem:StielConc-aux1} is a sum of bounded martingale differences $|(\Expt_{\le k}-\Expt_{\le k-1})[D_{m,k}(z)]|\le \frac{2}{-z}$. To conclude, use the Azuma-Hoeffding inequality (Lemma~\ref{lem:Hoeffding}).

	Remark: The above calculation (regarding the boundedness of $D_{m,k}(z)$) is, essentially, the proof of the well-known low-rank perturbation bound for resolvents, Lemma~\ref{lem:LowRankResolvent}, which we shall use later.

\end{proof}

\newpage

\section{Proof of Theorem~\ref{thm:Stiel-NonAsymp-Gaussian}}
\label{sec:proof-thm:Stiel-NonAsymp-Gaussian}

\begin{proof}
	[\unskip\nopunct]

We implement a well-known computation technique in random matrix theory (cf. \cite{bai2010spectral}), while carefully keeping track of the error terms. 

We rely explicitly on properties of the Gaussian distribution. Denote the eigendecomposition $H=\sum_{\ell=1}^d \tau_\ell v_\ell v_\ell^\T$ and 
\begin{align*}
	H_{\setminus \ell} = H - \tau_\ell v_\ell v_\ell^\T\,,\qquad s_\ell = S v_\ell \,.
\end{align*}
for $\ell=1,\ldots,d$. Note that $s_1,\ldots,s_d\sim \m{N}(0,m^{-1}I_m)$  are independent (since $v_1,\ldots,v_d$ are orthogonal).

Recall Eq. \eqref{eq:Stiel-To-Sigma}, so that 
\begin{align}
	\StielEmp(z) 
	&= -\frac1m(d-m)\frac{1}{z} + \frac1m \tr(H^{1/2}S^\T S H^{1/2} - z I)^{-1} \nonumber \\
	&= -\frac1m(d-m)\frac{1}{z} + \frac1m \sum_{\ell=1}^d v_\ell^\T (H^{1/2}S^\T S H^{1/2} - z I)^{-1} v_\ell \,.
	\label{eq:proof-thm:Stiel-NonAsymp-Gaussian-0}
\end{align}
We now compute an expression for $v_\ell^\T (H^{1/2}S_m^\T S_m H^{1/2} - z I)^{-1} v_\ell$, the $\ell$-th diagonal element of the resolvent $(H^{1/2}S_m^\T S_m H^{1/2} - z I)^{-1} $ written in the population covariance eigenbasis $V=[v_1|\dots|v_d]^\T$. 

Upon a coordinate permutation (where $\ell$ becomes the first), we have
\begin{align*}
	V^\T (H^{1/2}S^\T S H^{1/2}-z I)V = 
	\MatL 
		\tau_\ell \|s_\ell\|^2 - z 
		&\sqrt{\tau_\ell}s_\ell S H^{1/2}_{\setminus \ell} 
		\\
		\sqrt{\tau_\ell}H_{\setminus \ell}^{1/2}S^\T s_\ell 
		&H_{\setminus\ell}^{1/2}S^\T S H_{\setminus\ell}^{1/2}-z I
	\MatR \,.
\end{align*}
By the block matrix inverse formula (Lemma~\ref{lem:BlockInverse}),
\begin{align}
	v_\ell^\T (H^{1/2}S^\T S H^{1/2}-z I)^{-1} v_\ell 
	&= (V^\T (H^{1/2}S^\T S H^{1/2}-z I)V)^{-1}_{\ell,\ell} \nonumber \\
	&= \left(
		\tau_\ell \|s_\ell\|^2 - z 
		-
		{\tau_\ell}s_\ell S H^{1/2}_{\setminus \ell} 
		(H_{\setminus\ell}^{1/2}S^\T S H_{\setminus\ell}^{1/2}-z I)^{-1}
		H_{\setminus \ell}^{1/2}S^\T s_\ell
	\right)^{-1}
	\nonumber \\
	&\overset{(\star)}{=} \left(
		\tau_\ell \|s_\ell\|^2 - z 
		-
		{\tau_\ell}s_\ell 
		(S H_{\setminus\ell} S^\T -z I)^{-1}
		S H_{\setminus\ell} S^\T s_\ell
	\right)^{-1} 
	\nonumber \\
	&= -\frac1z
	\left(
		1 + \tau_\ell s_\ell 
		(S H_{\setminus\ell} S^\T -z I)^{-1}
		s_\ell
	\right)^{-1}\,,
	\label{eq:proof-thm:Stiel-NonAsymp-Gaussian-1}
\end{align}
where in $(\star)$ we used the identity $X^\T f(XX^\T) X = f(X^\T X)X^\T X$ which holds for any matrix $X$ and analytic function $f(\cdot)$.

We now wish to estimate the expectation of \eqref{eq:proof-thm:Stiel-NonAsymp-Gaussian-1}. Denote 
\begin{align*}
	D_\ell 
	&= s_\ell 
	(S H_{\setminus\ell} S^\T -z I)^{-1}
	s_\ell - \Expt\StielEmp(z) \\
	&= D_{\ell,1} + D_{\ell,2} + D_{\ell,3}
\end{align*}
where
\begin{align*}
	D_{\ell,1} 
	&= s_\ell 
	(S H_{\setminus\ell} S^\T -z I)^{-1}
	s_\ell 
	- 
	\frac1m
	\tr(S H_{\setminus\ell} S^\T -z I)^{-1} \,,\\
	D_{\ell,2}
	&= 
	\frac1m
	\tr(S H_{\setminus\ell} S^\T -z I)^{-1} 
	- 
	\frac1m
	\tr(S H S^\T -z I)^{-1} \,,\\
	D_{\ell,3}
	&= 
	\frac1m
	\tr(S H S^\T -z I)^{-1} - 
	\frac1m\Expt
	\tr(S H S^\T -z I)^{-1} \,.
\end{align*}
We have 
\begin{align*}
	 v_\ell^\T (H^{1/2}S^\T S H^{1/2}-z I)^{-1} v_\ell  =  -\frac1z
	\left(
		1 + \tau_\ell \Expt\StielEmp(z)
	\right)^{-1}
	+ 
	\frac1z \frac{\tau_\ell D_\ell}{(1 + \tau_\ell s_\ell 
	(S H_{\setminus\ell} S^\T -z I)^{-1}
	s_\ell)(1 + \tau_\ell \Expt\Stiel(z))}\,,
\end{align*}
so that using \eqref{eq:proof-thm:Stiel-NonAsymp-Gaussian-0},
\begin{align}
	\Expt\Stiel(z)
	&= \frac1m(d-m)\frac{1}{z} + \frac1m \sum_{\ell=1}^d \Expt v_\ell^\T (H^{1/2}S^\T S H^{1/2} - z I)^{-1} v_\ell \nonumber \\
	&= \frac1m(d-m)\frac{1}{z}  -\frac1m \sum_{\ell=1}^d \frac1z
	\left(
		1 + \tau_\ell \Expt\StielEmp(z)
	\right)^{-1}
	+ e_m(z) \nonumber \\
	&= -\frac1z + \frac1z \frac1m \tr\left[ I - (I+\Expt\StielEmp(z)H)^{-1} \right] + e_m(z) \nonumber \\
	&= -\frac1z \left(  1- \frac1m\DimEff{1/\Expt\StielEmp(z)} \right) + e_m(z) \,,
\end{align}
where
\begin{align}
	e_m(z)
	&= \frac{1}{zm}\sum_{\ell=1}^m 
	\Expt\left[ \frac{\tau_\ell D_\ell}{(1 + \tau_\ell s_\ell 
	(S H_{\setminus\ell} S^\T -z I)^{-1}
	s_\ell)(1 + \tau_\ell \Expt\Stiel(z))} \right]\,. \nonumber 
\end{align}
We may now bound
\begin{align}
	|e_m(z)| 
	&\le \frac{1}{m|z|}\sum_{\ell=1}^d \frac{\tau_\ell}{1+\tau_\ell\Expt\StielEmp(z)}\max_{1\le \ell \le d} \Expt|D_\ell| \nonumber \\ 
	&=
	\frac{1}{|z|\Expt\StielEmp(z)} \frac1m\DimEff{1/\Expt\StielEmp(z)} \max_{1\le \ell \le d}\Expt|D_\ell| \,.
\end{align}
We now decompose $\Expt|D_\ell| \le \Expt|D_{\ell,1}| + \Expt|D_{\ell,2}| + \Expt|D_{\ell,3}|$. By Lemma~\ref{lem:HW-Properties} (Item 1),
\begin{align*}
	\Expt|D_{\ell,1}| \le \sqrt{\Expt|D_{1,\ell}|^2} \lesssim \frac1m \Expt \|(S_mH_{\setminus\ell} S_m - zI)^{-1}\|_F \le \frac{1}{|z|\sqrt{m}}\,. 
\end{align*}
By the low rank resolvent perturbation lemma, Lemma~\ref{lem:LowRankResolvent}, almost surely
\begin{align*}
	|D_{\ell,2}| \lesssim \frac{1}{|z|m} \,.
\end{align*}
Finally, by Lemma~\ref{lem:StielConc},
\begin{align*}
	\Expt|D_{\ell,3}|=\Expt|\StielEmp(z)-\Expt\StielEmp_m(z)| \lesssim \frac{1}{|z|\sqrt{m}} \,.
\end{align*}
Thus,
\begin{equation}
	|e_m(z)|=\BigOh\left( \frac{1}{|z|\Expt\StielEmp(z)} \frac1m\DimEff{1/\Expt\StielEmp(z)} \frac{1}{|z|\sqrt{m}} \right)\,,
\end{equation}
and so the proof is concluded.

\end{proof}

\newpage

\section{Proof of Theorem~\ref{thm:Avg-Gauss}}

	Let $\hl$ the output of Algorithm~\ref{Alg:2}, and denote 
	\begin{align}
		\m{E} = (H+\lambda I)^{1/2}\hat{W} (H+\lambda I)^{1/2} - I,\qquad \hW := S^\T (S H S^\T + \hl I)^{-1} S\,.
	\end{align}
	We have $\bar{\m{E}}=\frac1q \sum_{\ell=1}^q \m{E}^{(\ell)}$, where $\m{E}^{(1)},\ldots,\m{E}^{(q)}\IIDDist \m{E}$.

	The proof proceeds in two parts. First we show that the expectation $\Expt[\m{E}]$ is small; then, we show that $\bar{\m{E}}$ concentrates around $\Expt[\bar{\m{E}}]=\Expt[\m{E}]$ using matrix concentration inequalities.

	The following lemma, proven in Section~\ref{sec:proof-lem:Expt-Matrix-Gauss}, bounds $\Expt[\m{E}]$.
	\begin{lemma}
		\label{lem:Expt-Matrix-Gauss}
		Suppose that $S$ has i.i.d. Gaussian entries. Assume that $m\ge 1.5\DimEff{\lambda}$. We have 
		\[
		\|\Expt[\m{E}]\| = \BigOh(\frac{1}{\sqrt{m}}) \,.
		\]
	\end{lemma}

	To establish concentration, we use the matrix Bernstein inequality \cite[Theorem 6.6.1]{tropp2015introduction}. We cite it here.
	\begin{lemma}\label{lem:MatrixBernstein}
		Let $X_1,\ldots,X_n\in \RR^{d\times d}$ be independent Hermitian random matrices. 
		Assume that $\Expt[X_n]=0$, $\|X_n\|\le L$, and denote
		\begin{align}\label{eq:lem:MatrixBernstein-Variance}
			\nu^2 = \left\| \sum_{n=1}^N \Expt[X_n^2] \right\| \,.
		\end{align}
		Then for all $t\ge 0$,
		\begin{align}
			\Pr\left( \left\| \sum_{n=1}^N X_n \right\| \ge t \right) \le d\exp\left(
   \frac{-t^2/2}{\nu^2 + Lt/3}
   \right) \,.
		\end{align}
	\end{lemma}

	We shall apply Lemma~\ref{lem:MatrixBernstein} with a truncated version of $\m{E}$, $\tilde{\m{E}}_L=\m{E}\Indic{\|\m{E}\|\le L}-\Expt[\m{E}\Indic{\|\m{E}\|\le L}]$ for a suitably chosen $L$. 
	To this end, we give the following high-probability bound on $\|\Phi\|$.

	\begin{lemma}\label{lem:Avg-OpNormBound}
		Suppose that $S$ has i.i.d. Gaussian entries. For $\delta\in (0,1)$ and $q\ge 1$, set
		\begin{align}
			L_{\delta,q} = C \frac{d}{m} + C\frac{\log(q/\delta)}{m}
		\end{align}
		for large enough $C>0$. Then,
		\begin{enumerate}
			\item W.p. $1-\delta/q$, $\|\m{E}\|\le L_{\delta,q}$.
			\item $\Expt[\|\m{E}\|\Indic{\|\m{E}\|\ge L_{\delta,q}}] = \BigOh(\frac1m e^{-d})$.
			\item $\Expt[\|\m{E}\|^2\Indic{\|\m{E}\|\ge L_{\delta,q}}] = \BigOh(\frac{d}{m^2} e^{-d})$.
		\end{enumerate}
	\end{lemma}
	We prove Lemma~\ref{lem:Avg-OpNormBound} in Section~\ref{sec:proof-lem:Avg-OpNormBound}.

	The last component needed to apply matrix Bernstein is the variance proxy \eqref{eq:lem:MatrixBernstein-Variance}. Note that we always have $\nu^2 \le L^2N$, so that 
	\begin{align}
		\Pr\left( \left\| \frac1N \sum_{n=1}^N X_n \right\| \ge \eps \right) \le d\exp(-cN \min\left\{\left(\frac{\eps}{L}\right)^2,\frac{\eps}{L}\right\}) \,.
	\end{align}
	However, if the variance $\Expt[X^2] \ll L^2 $, one may obtain substantially sharper bounds for small $\eps\ll L$. 
	This is the case in our setting. The following is proved in Section~\ref{sec:proof-lem:Expt-Matrix-Covariance}.
	\begin{lemma}\label{lem:Expt-Matrix-Covariance}
		Suppose that $S$ has i.i.d. Gaussian entries. Assume that $m\ge 1.5\DimEff{\lambda}$. Then 
		\[
		\Expt[\m{E}^2] = \BigOh(d/m) \,.	
		\]
	\end{lemma}

	With the above, we are ready to conclude the proof of Theorem~\ref{thm:Avg-Gauss}.
	\begin{proof}
		(Of Theorem~\ref{thm:Avg-Gauss}.)
  
	Consider the truncated matrices $\tilde{\m{E}}^{(1)}_{L_{\delta,q}},\ldots,\tilde{\m{E}}^{(q)}_{L_{\delta,q}}$ with $L_{\delta,q}$ as in Lemma~\ref{lem:Avg-OpNormBound}. By Item 1 of  Lemma~\ref{lem:Avg-OpNormBound} (and union bound over $\ell=1,\ldots,q$), with probability $1-\delta$, we have 
	\begin{align}
		\bar{\m{E}} = \frac{1}{q} \sum_{\ell=1}^q \tilde{\m{E}}^{(\ell)}_{L_{\delta,q}} + \Expt\left[ {\m{E}} \Indic{\|\m{E}\|\le L_{\delta,q}} \right] \,.
	\end{align}
	
	Note that 
	\begin{align}
		\left\|
			\Expt\left[ {\m{E}} \Indic{\|\m{E}\|\le L_{\delta,q}} \right]
		\right\|
		=
		\left\|
			\Expt[\m{E}] +	
			\Expt\left[ {\m{E}} \Indic{\|\m{E}\|>  L_{\delta,q}} \right]
		\right\|
		\le 
		\|\Expt[\m{E}]\| + \Expt\left[ \|\m{E}\|\Indic{\|\m{E}\|>  L_{\delta,q}} \right] = \BigOh(1/\sqrt{m})\,,
	\end{align}
	where the first inequality uses Jensen's inequality, and the second inequality follows from Lemma~\ref{lem:Avg-OpNormBound} Item 2 and Lemma~\ref{lem:Expt-Matrix-Gauss}. Similarly,
	\begin{align}
		\left\| \Expt\left[ \tilde{\m{E}}_{L_{\delta,q}}^2 \right] \right\|
		\lesssim
		\left\| \Expt\left[ {\m{E}}^2 \right] \right\| +  \Expt\left[ \|\m{E}\|^2 \Indic{\|\m{E}\|>  L_{\delta,q}} \right] = \BigOh(d/m)\,, 
	\end{align}
	where we used Lemma~\ref{lem:Expt-Matrix-Covariance}. 

	Now, by the matrix Bernstein inequality (Lemma~\ref{lem:MatrixBernstein}), the bound 
	\begin{align}
		\|\bar{\m{E}}\| \le \eps + \BigOh(1/\sqrt{m})
	\end{align}
	holds with probability at least
	\begin{align}
		1-d\exp\left( 
  -cq\min\{\frac{\eps^2}{d/m}, \frac{\eps}{L_{p,q}}\} 
  \right) \,.
	\end{align}
	When $q\le \exp(\BigOh(d))$ and $1/\delta\le \exp(\BigOh(d))$ we have $L_{\delta,q}=\BigOh(d/m)$, so for $\eps=\BigOh(1)$ the smaller term is $\frac{\eps^2}{d/m}$. And so, $q\gtrsim \frac{d\log d}{m}\frac{\log(1/\delta)}{\eps^2}$ ensures that the probability above is $\ge 1-\delta$. 

\end{proof}

\section{Proof of Lemma~\ref{lem:Expt-Matrix-Gauss}}
\label{sec:proof-lem:Expt-Matrix-Gauss}

\begin{proof}
	[\unskip\nopunct]

We use the machinery from Section~\ref{sec:proof-thm:Stiel-NonAsymp-Gaussian}. As before, we denote the eigendecomposition of $H$,
\begin{align*}
	H = V\diag(\tau_1,\ldots,\tau_d)V^\T,\qquad V=[v_1|\cdots|v_d]
\end{align*}
 and 
\begin{align*}
	H_{\setminus \ell} = H - \tau_\ell v_\ell v_\ell^\T\,,\qquad s_\ell = S v_\ell,\qquad s_1,\ldots,s_d \IIDDist \m{N}(0,m^{-1}I_m)\,.
\end{align*}
We shall compute the expectation of $\hat{W}=S^\T(S H S^\T +  I)^{-1}S$ in the eigenbasis $V$, that is, $\Expt[V^\T \hW V]$.

Denote, for brevity,
\begin{align*}
	T(\hl) = (SHS^\T + \hl I)^{-1},\qquad \Tell(\hl) = (S\Hell S^\T + \hl I)^{-1} \,,
\end{align*}
so that by the Sherman-Morrison formula (Lemma~\ref{lem:Sherman-Morrison}), writing $SHS^\T = S\Hell S^\T + \tau_\ell s_\ell s_\ell^\T$,
\begin{align}\label{eq:sec:proof-lem:Expt-Matrix-Gauss-aux-1}
	T(\hl) = \Tell(\hl) - \frac{\tau_\ell \Tell(\hl)s_\ell s_\ell \Tell(\hl)}{1 + \tau_\ell s_\ell \Tell(\hl) s_\ell } \,.
\end{align}
Let $j\ne \ell$. Note that $\hl,\Tell(\hl)$ do not change under a sign flip $s_j\mapsto -s_j$. Since $s_j$ has a symmetric distribution ($s_j\overset{d}{=}-s_j$), we deduce that 
\begin{align}
	\Expt[v_\ell^\T \hW v_j] = \Expt[s_\ell^\T T(\hl) s_j] = 0 \,,
\end{align} 
that is, $\Expt[V^\T \hW V]$ is diagonal. Consequently, $\Expt[V^\T \m{E} V] = \Expt[V^\T(H+\lambda I)^{1/2}\hW (H+\lambda I)^{1/2} V]$ is diagonal as well (since $V$ is an eigenbasis of $H$).

Let us calculate the diagonal elements. By \eqref{eq:sec:proof-lem:Expt-Matrix-Gauss-aux-1},
\begin{align}
	v_\ell^\T \hW v_\ell = s_\ell^\T T(\hl) s_\ell 
	= \frac{s_\ell^\T \Tell(\hl) s_\ell}{1 + \tau_\ell s_\ell^\T \Tell(\hl) s_\ell} =
	\frac{1}{(s_\ell^\T \Tell(\hl) s_\ell)^{-1} + \tau_\ell}\,,
\end{align}
so 
\begin{align}
	v_\ell^\T \Phi v_\ell 
	&= \frac{\lambda + \tau_\ell}{(s_\ell^\T \Tell(\hl) s_\ell)^{-1} + \tau_\ell} - 1 \nonumber\\
	&= \frac{\lambda - (s_\ell^\T \Tell(\hl) s_\ell)^{-1}}{ (s_\ell^\T \Tell(\hl) s_\ell)^{-1} + \tau_\ell } \nonumber \\
	&= \BigOh(|\lambda s_\ell^\T \Tell(\hl) s_\ell - 1|) \,.
	\label{eq:sec:proof-lem:Expt-Matrix-Gauss-aux-2}
\end{align}

Let $\tilde{\lambda}$ be given by \eqref{eq:Oracle-Reg}, so that by Theorem~\ref{thm:Alg2-Gauss}, $\hl$ is close to $\tilde{\lambda}$ with high probability. Let
\begin{align*}
	\Delta_1 &= \lambda s_\ell^\T \Tell(\hl)s_\ell - \lambda s_\ell^\T \Tell(\tilde{\lambda})s_\ell\,,\\
	\Delta_2 &= \lambda s_\ell^\T \Tell(\tilde{\lambda})s_\ell - \lambda \frac1m \tr \Tell(\tilde{\lambda})\,, \\
	\Delta_3 &= \lambda \frac1m \tr \Tell(\tilde{\lambda}) - \lambda \frac1m \tr T(\tilde{\lambda}) \,,\\
	\Delta_4 &= \lambda \frac1m \tr T(\tilde{\lambda}) - \lambda \frac1m \tr T(\hl) = \lambda \frac1m \tr T(\tilde{\lambda}) - 1 \,,
\end{align*}
so that 
\[
	\lambda s_\ell^\T \Tell(\hl) s_\ell - 1 = \Delta_1 + \Delta_2 + \Delta_3 + \Delta_4 \,.
\]

We have 
\begin{align*}
	\Expt|\Delta_1| 
	&= \Expt[ \lambda |s_\ell^\T \Tell(\hl)\Tell(\tilde{\lambda}) s_\ell | |\hl-\tilde{\lambda}| ] 
	= \Expt\left[ \frac{\lambda}{\hl \tilde{\lambda}} \|s_\ell\|^2 |\hl-\tilde{\lambda}| \right] 
	\lesssim \frac{1}{\lambda}(\Expt\|s_\ell\|^4\|)^{1/2}(\Expt|\hl-\tilde{\lambda}|^2)^{1/2} 
	\lesssim \frac{1}{\sqrt{m}} \,, \\
	\Expt|\Delta_3| 
	&= \lambda \frac1m \Expt[\tr(\Tell(\hl)\Tell(\tilde{\lambda}))|\hl-\tilde{\lambda}|] 
	\le \Expt\left[ \frac{\lambda }{\hl \tilde{\lambda}} |\hl-\tilde{\lambda}| \right] \lesssim \frac{1}{\sqrt{m}} \,,\\
	\Expt|\Delta_4| &\lesssim \frac{1}{\sqrt{m}}
\end{align*}
where we used Lemma~\ref{lem:thm:Alg-2-Gauss-Expt} and that $\hl=\Omega(\lambda)$ w.p. $1$. As for $\Delta_2$, by Lemma~\ref{lem:HW-Properties} Item 1, 
\begin{align*}
	\Expt|\Delta_2| \lesssim \lambda \frac1m  (\Expt \|\Tell(\tilde{\lambda})\|^2_F)^{1/2} \lesssim \frac{1}{\sqrt{m}} \,.
\end{align*}
Thus, we conclude that $\|\Expt\m{E}\|=\|\diag(\Expt V^\T \m{E} V)\| = \BigOh(\frac{1}{\sqrt{m}})$.

\end{proof}

\section{Proof of Lemma~\ref{lem:Avg-OpNormBound}}
\label{sec:proof-lem:Avg-OpNormBound}

\begin{proof}
	[\unskip\nopunct]

	Clearly, $\|\m{E}\|\le 1 + \|(H+\lambda I)^{1/2}\hW (H+\lambda I)^{1/2}\|$. By the inequality $\sqrt{a+b}\le \sqrt{a}+\sqrt{b}$ (for $a,b\ge 0$), we have $\|(H+\lambda I)^{1/2}(H^{1/2} + \lambda^{1/2}I)^{-1}\| \le 1$, so
		\begin{align*}
			\|\Phi\| \le \|(H^{1/2}+\lambda^{1/2})\hW (H^{1/2}+\lambda^{1/2}) \| \le 2\|H^{1/2}\hW H^{1/2}\| + 2\lambda \|\hW\| \,,
		\end{align*}
		where the second inequality follows from the fact that $\hW$ is PSD, and therefore for any $u,v$, $(u+v)\hW (u+v) = \|u+v\|_{\hW}^2 \le 2\|u\|_{\hW}^2 + 2\|v\|_{\hW}^2$ where $\|u\|_{\hW}^2:= u^\T\hW u$ is a norm. 
		We have 
		\begin{align*}
			H^{1/2}\hW H^{1/2}=H^{1/2}S^\T(SHS^\T + \hl I)^{-1}S H^{1/2} 
   &= (H^{1/2}S^\T S H^{1/2} + \hl I)^{-1}H^{1/2}S^\T S H^{1/2} \\
   &= I - \hl (H^{1/2}S^\T S H^{1/2} + \hl)^{-1}\,,
		\end{align*}
		so that $\|H^{1/2}\hW H^{1/2}\| \le 2$. 
		Moreover, $\|\hW\| = \|S^\T (SHS^\T +\hl I)^{-1} S\| \le \frac{1}{\hl}\|S^\T S\|$. 
		Recalling that $\hl \ge 5\lambda/12$ (by construction), 
		$\lambda \|\hW\|\le \frac{12}{5}\|S^\T S\|$. Combining all the above, $\|\m{E}\|\le 5 + \frac{24}{5}\|S^\T S\|$. 
		
		By \cite[Theorem 4.6.1]{vershynin2018high}, for every $t\ge 0$, w.p. $1-2e^{-t}$,
		\begin{align*}
			\|S^\T S\| \le 1 + C_1 \max\{\mu,\mu^2\}\quad\textrm{where}\quad \mu = \sqrt{\frac{d}{m}} + \sqrt{\frac{t}{{m}}} \,.
		\end{align*}
		for $C_1$ large enough. 
		Since $m\le d$, note that $\max\{\mu,\mu^2\}=\mu^2 \le 2\frac{d}{m} + 2\frac{t}{m}$. Item 1 follows by setting $t\sim \log(q/\delta)$.

		As for Item 2, 
		\begin{align*}
			\Expt[\|S^\T S\|\Indic{\|S^\T S\|>1+4C_1\frac{d}{m}}] 
			&= \int_{1+4C_1\frac{d}{m}} ^\infty \Pr(\|S^\T S\|\ge s)ds \\
			&= 2C_1 \frac{1}{m}\int_{d}^\infty \Pr(\|S^\T S\|\ge 1 + 2C_1\frac{d}{m} + 2C_1 \frac{t}{m})dt \\
			&\lesssim \frac{1}{m}\int_{d}^{\infty}e^{-t}dt = \frac{1}{m}e^{-d} \,.
		\end{align*}
		Deducing Item 2 of the lemma is now easy from the above estimate.

		Finally, for Item 3, 
		\begin{align*}
			&\Expt[\|S^\T S\|^2\Indic{\|S^\T S\|>1+4C_1\frac{d}{m}}] 
			= \int_{1+4C_1\frac{d}{m}} ^\infty 2s \Pr(\|S^\T S\|^2\ge s^2)ds \\
			&\qquad= 2C_1 \frac{1}{m}\int_{d}^\infty \left( 1 + 2C_1\frac{d}{m} + 2C_1 \frac{t}{m} \right)\Pr(\|S^\T S\|\ge 1 + 2C_1\frac{d}{m} + 2C_1 \frac{t}{m})dt \\
			&\qquad\lesssim \frac{1}{m}\int_{d}^{\infty} \left( 1 + 2C_1\frac{d}{m} + 2C_1 \frac{t}{m} \right)e^{-t}dt = \BigOh(\frac{d}{m^2}e^{-d}) \,.
		\end{align*}
\end{proof}

\section{Proof of Lemma~\ref{lem:Expt-Matrix-Covariance}}
\label{sec:proof-lem:Expt-Matrix-Covariance}

\begin{proof}
	[\unskip\nopunct]

	We continue from Section~\ref{sec:proof-lem:Expt-Matrix-Gauss}, expressing $\m{E}$ in the eigenbasis of $H$. 

	Recall that $\m{E}=(H+\lambda)^{1/2}S^\T T(\hl) S (H+\lambda)^{1/2} - I$.  
	Then starting from \eqref{eq:sec:proof-lem:Expt-Matrix-Gauss-aux-1},
	we have for any $\ell,j=1,\ldots,d$,
	\begin{align}
		v_\ell^\T \m{E} v_j 
		=
		(\tau_\ell+\lambda)^{1/2}(\tau_j+\lambda)^{1/2} \frac{ s_j \Tell(\hl)s_\ell}{1+\tau_\ell s_\ell \Tell(\hl)s_\ell} - \Indic{\ell=j}\,.
		\label{eq:proof-lem:Expt-Matrix-Covariance-axu-1}
	\end{align}
	Consequently, if $j\ne k$ then note that $\Expt[v_j^\T \m{E} v_\ell v_k^\T \m{E} v_\ell ] = 0$, since $s_j,s_k$ have a symmetric distribution.

	Recall, we are interested in bounding $\|\Expt[\m{E}^2]\|=\|\Expt[(V^\T \m{E} V)^2]\|$. We have 
	\begin{align*}
		\Expt[(V^\T \m{E} V)^2_{j,k}] = \sum_{\ell=1}^d \Expt[ (V^\T \m{E} V)_{j,\ell}(V^\T \m{E} V)_{k,\ell} ]\,,
	\end{align*}
	which is zero when $j\ne k$; that is, $\Expt[(V^\T \m{E} V)^2]$ is diagonal. Let us compute the diagonal elements, $j=k$. We have 
	\begin{align}
		\Expt[(V^\T \m{E} V)^2_{k,k}]
		&= \Expt[ (V^\T \m{E} V)_{k,k}(V^\T \m{E} V)_{k,k} ] + \sum_{\ell\ne k} \Expt[ (V^\T \m{E} V)_{k,\ell}(V^\T \m{E} V)_{k,\ell} ] \,.
	\end{align}
	The first term is
	\begin{align}
		\Expt[ (V^\T \Phi V)_{k,k}^2 ] 
		&= \Expt\left[ \left( \frac{\lambda - (s_\ell^\T \Tell(\hl) s_\ell)^{-1}}{ (s_\ell^\T \Tell(\hl) s_\ell)^{-1} + \tau_\ell } \right)^2 \right] \nonumber \\
		&= \BigOh(|\lambda s_\ell^\T \Tell(\hl) s_\ell - 1|^2)\,,
	\end{align}
	where we used \eqref{eq:sec:proof-lem:Expt-Matrix-Gauss-aux-2}. Similar to Section~\ref{sec:proof-lem:Expt-Matrix-Gauss}, one can show this is $\BigOh(1/m)$. For the second term, using  \eqref{eq:proof-lem:Expt-Matrix-Covariance-axu-1},
	\begin{align}
		&\sum_{\ell=1,\ell\ne k}^d  (V^\T \m{E} V)_{k,\ell}(V^\T \m{E} V)_{k,\ell} 
		= 
		\sum_{\ell=1,\ell\ne k}^d 
		(\tau_k+\lambda)(\tau_\ell+\lambda)
		\frac{(s_k \Tk(\hl)s_\ell)^2}{(1+\tau_k s_k^\T \Tk(\hl)s_k)^2} \nonumber \\
		&\qquad=
		(\tau_k+\lambda)
		\frac{1 }{(1+\tau_k s_k \Tk(\hl)s_k)^2}
		s_k^\T \Tk(\hl)
		\left[
		\sum_{\ell=1,\ell\ne k}^d 
		(\tau_\ell+\lambda)
		s_\ell s_\ell^\T 
		\right]
		\Tk(\hl) s_k  \,.
		\label{eq:proof-lem:Expt-Matrix-Covariance-axu-2}
	\end{align}

	For small $c'>0$, let $\Omega$ be the event that $ s_k^\T \Tk(\hl) s_k \ge c'/\lambda$;  by similar calculations as in Section~\ref{sec:proof-lem:Expt-Matrix-Gauss}, for some $c,C>0$, $\Pr(\Omega^c)\le Ce^{-cm}$. Then
	\begin{align*}
		\sum_{\ell=1,\ell\ne k}^d  (V^\T \m{E} V)_{k,\ell}(V^\T \m{E} V)_{k,\ell} \Indic{\Omega}
		&\lesssim 
		\frac{\lambda^2}{\tau_k + \lambda} 
		s_k^\T \Tk(\hl)
		\left[
		\sum_{\ell=1,\ell\ne k}^d 
		(\tau_\ell+\lambda)
		s_\ell s_\ell^\T 
		\right]
		\Tk(\hl) s_k
		\Indic{\Omega} \\ 
		&\le 
		\frac{\lambda^2}{\tau_k + \lambda} 
		s_k^\T \Tk(\hl)
		\left[
		\sum_{\ell=1,\ell\ne k}^d 
		(\tau_\ell+\lambda)
		s_\ell s_\ell^\T 
		\right]
		\Tk(\hl) s_k\,,
	\end{align*}
	so taking the expectation,
	\begin{align*}
		\Expt\left[ \sum_{\ell=1,\ell\ne k}^d  (V^\T \m{E} V)_{k,\ell}(V^\T \m{E} V)_{k,\ell} \Indic{\Omega} \right]
		&\lesssim
		\frac{\lambda^2}{\tau_k + \lambda} 
		\frac1m \Expt \tr \left(
			\Tk(\hl)
			\left[
			\sum_{\ell=1,\ell\ne k}^d 
			(\tau_\ell+\lambda)
			s_\ell s_\ell^\T 
			\right]
			\Tk(\hl)
		\right) \\
		&\lesssim
		\frac{\lambda}{\tau_k + \lambda} 
		\frac1m \Expt \tr \left(
			\left[
			\sum_{\ell=1,\ell\ne k}^d 
			(\tau_\ell+\lambda)
			s_\ell s_\ell^\T 
			\right]
			\Tk(\hl)
		\right) \\
		&\le 
		\frac1m 
			\sum_{\ell=1,\ell\ne k}^d 
			(\tau_\ell+\lambda)
			\Expt  s_\ell^\T \Tk(\hl) s_\ell \,.
	\end{align*}
	Using \eqref{eq:sec:proof-lem:Expt-Matrix-Gauss-aux-1},
	\begin{align}
		\Expt  s_\ell^\T \Tk(\hl) s_\ell = 
             \Expt
			 \left[ 
				\frac{s_\ell^\T T_{\setminus k,\ell}(\hl) s_\ell}{1+\tau_\ell s_\ell^\T T_{\setminus k,\ell}(\hl) s_\ell} 
				\right]
			 \le 
			 \frac{\Expt[s_\ell^\T T_{\setminus k,\ell}(\hl) s_\ell]}{1+\tau_\ell \Expt[s_\ell^\T T_{\setminus k,\ell}(\hl) s_\ell]} \,,
			 \label{eq:proof-lem:Expt-Matrix-Covariance-axu-4}
	\end{align}
	where we used Jensen's inequality with the concave function $x\mapsto \frac{x}{1+\tau_\ell x}$, $x>0$. By a calculation similar to Section~\ref{sec:proof-lem:Expt-Matrix-Gauss}, $\Expt[s_\ell^\T T_{\setminus k,\ell}(\hl) s_\ell]= \frac{1}{\lambda} + \BigOh(\frac{1}{\lambda\sqrt{m}})$, hence 
	\begin{align*}
		\Expt  s_\ell^\T \Tk(\hl) s_\ell \le  \frac{\Expt[s_\ell^\T T_{\setminus k,\ell}(\hl) s_\ell]}{1+\tau_\ell \Expt[s_\ell^\T T_{\setminus k,\ell}(\hl) s_\ell]} \lesssim \frac{1}{\tau_\ell + \lambda } \,.
	\end{align*}
	Consequently,
	\begin{align}
		\Expt\left[ \sum_{\ell=1,\ell\ne k}^d  (V^\T \m{E} V)_{k,\ell}(V^\T \m{E} V)_{k,\ell} \Indic{\Omega} \right]
		&\lesssim
  \frac1m 
			\sum_{\ell=1,\ell\ne k}^d 
			(\tau_\ell+\lambda)
			\Expt  s_\ell^\T \Tk(\hl) s_\ell
  \lesssim \frac{d}{m} \,.
	\end{align}
 
	Finally, towards bounding the expectation under the complement event ${\Omega^c}$ note we can write
	\begin{align}
		\sum_{\ell=1,\ell\ne k}^d  (V^\T \m{E} V)_{k,\ell}(V^\T \m{E} V)_{k,\ell} 
		&=
		(\tau_k+\lambda)
		\frac{1 }{(1+\tau_k s_k \Tk(\hl)s_k)^2}
		\sum_{\ell=1,\ell\ne k}^d 
		(\tau_\ell+\lambda)
		(s_\ell^\T \Tk(\hl) s_k)^2		\nonumber \\
		&\le 
		(\tau_k+\lambda)
		\frac{s_k^\T  \Tk(\hl)s_k }{(1+\tau_k s_k \Tk(\hl)s_k)^2}
		\sum_{\ell=1,\ell\ne k}^d 
		(\tau_\ell+\lambda)
		s_\ell^\T \Tk(\hl) s_\ell	\nonumber \\
		&\le 
		(\tau_k+\lambda)
		\frac{s_k^\T  \Tk(\hl)s_k }{1+\tau_k s_k \Tk(\hl)s_k}
		\sum_{\ell=1,\ell\ne k}^d 
		(\tau_\ell+\lambda)
		s_\ell^\T \Tk(\hl) s_\ell
		\,,	
		\label{eq:proof-lem:Expt-Matrix-Covariance-axu-3}
	\end{align}
	where the first inequality, $(s_\ell^\T \Tk(\hl) s_k)^2 \le (s_\ell^\T \Tk(\hl) s_\ell)(s_k^\T \Tk(\hl) s_k)$, follows by Cauchy-Schwartz. We have
	\begin{align*}
		(\tau_k+\lambda)
		\frac{s_k^\T  \Tk(\hl)s_k }{1+\tau_k s_k \Tk(\hl)s_k} 
		&=
		\frac{\tau_k + \lambda}{\tau_k + (s_k^\T  \Tk(\hl)s_k)^{-1}} \\
		&= 1 + \frac{\lambda - (s_k^\T  \Tk(\hl)s_k)^{-1}}{\tau_k + (s_k^\T  \Tk(\hl)s_k)^{-1}}\\
		&\le 1 + |\lambda (s_k^\T  \Tk(\hl)s_k)-1| \,.
	\end{align*}
	Recall $\Omega$, the event that $\lambda s_k^\T \Tk(\hl) s_k \ge c'$;  we now need to bound $\Expt[\sum_{\ell=1,\ell\ne k}^d  (V^\T \m{E} V)_{k,\ell}(V^\T \m{E} V)_{k,\ell}\Indic{\Omega^c}]$. Under $\Omega^c$, $0\le \lambda s_k^\T \Tk(\hl) <  c'$, and so $(\tau_k+\lambda)
	\frac{s_k^\T  \Tk(\hl)s_k }{1+\tau_k s_k \Tk(\hl)s_k}\Indic{\Omega^c}=\BigOh(1)$. Plugging this into \eqref{eq:proof-lem:Expt-Matrix-Covariance-axu-3},
	\begin{align}
		\Expt[\sum_{\ell=1,\ell\ne k}^d  (V^\T \m{E} V)_{k,\ell}(V^\T \m{E} V)_{k,\ell} \Indic{\Omega^c}]
		\lesssim 
		\sum_{\ell=1,\ell\ne k}^d 
		(\tau_\ell+\lambda)
		\Expt[s_\ell^\T \Tk(\hl) s_\ell \Indic{\Omega^c} ] \,.
	\end{align}
	Using \eqref{eq:proof-lem:Expt-Matrix-Covariance-axu-4},
	\begin{align*}
		(\tau_\ell+\lambda)
		\Expt[s_\ell^\T \Tk(\hl) s_\ell \Indic{\Omega^c} ]
		&=
		\Expt
			 \left[ 
				\frac{\tau_\ell + \lambda}{\tau_\ell + (s_\ell^\T T_{\setminus k,\ell}(\hl) s_\ell)^{-1} } 
				\Indic{\Omega^c} 
				\right] \\
		&\le 
		\Expt\left[\left(1 + |\lambda(s_\ell^\T T_{\setminus k,\ell}(\hl) s_\ell)-1|\right)\Indic{\Omega^c} \right] \\
		&\lesssim \Expt[(1+\|s_\ell\|^2)\Indic{\Omega^c}] \lesssim \Pr(\Omega^c) + \sqrt{\Expt\|s_\ell\|^4 \Pr(\Omega^c) } \lesssim e^{-cm} \,,
	\end{align*}
	as $\Pr(\m{E}^c)\lesssim e^{-cm}$. Thus, $\Expt[\sum_{\ell=1,\ell\ne k}^d  (V^\T \m{E} V)_{k,\ell}(V^\T \m{E} V)_{k,\ell} \Indic{\Omega^c}] \le de^{-cm}$. And so, we finally conclude
	\begin{align}
		\Expt[\sum_{\ell=1,\ell\ne k}^d  (V^\T \Phi V)_{k,\ell}(V^\T \Phi V)_{k,\ell} ] 
		&= \Expt\sum_{\ell=1,\ell\ne k}^d  (V^\T \m{E} V)_{k,\ell}(V^\T \m{E} V)_{k,\ell} \Indic{\Omega} + \Expt\sum_{\ell=1,\ell\ne k}^d  (V^\T \m{E} V)_{k,\ell}(V^\T \m{E} V)_{k,\ell} \Indic{\Omega^c} \nonumber \\
		&\lesssim \frac{d}{m} + de^{-cm} = \BigOh(d/m) \,.
	\end{align}

\end{proof}

\newpage

\section{Proof of Theorem~\ref{thm:Convergence}}

Define the \textbf{Newton decrement} at a point $\theta\in\RR^d$:
\begin{equation}
    \sfN(\theta) = \left( (\nabla G(\theta))^\T (\nabla^2 G(\theta))^{-1}(\nabla G(\theta)) \right)^{1/2} .
\end{equation}
Denote the approximate Newton decrement, where $\bar{W}(\theta)$ is an approximation of the true inverse Hessian:
\begin{equation}
    \tilde{\sfN}(\theta) = \left( (\nabla G(\theta))^\T \bar{W}(\theta) (\nabla G(\theta)) 
    \right)^{1/2}.
\end{equation}
Note that if $\bar{W}$ is $\eta$-accurate in the sense of Definition~\ref{def:Accurate}, then 
\begin{equation}\label{eq:eta-Accurate}
    \sqrt{1-\eta}\sfN(\theta)\le \tilde{\sfN}(\theta) \le \sqrt{1+\eta}\sfN(\theta) \,.
\end{equation}

For self-concordant functions, it is known that the Newton decrement yields an upper bound on the suboptimality gap $G(\theta)-G(\theta^\star)$:
\begin{lemma}
    If $\sfN(\theta)<\sfN_0$ for some numerical constant $\sfN_0\le 0.68$, then
    \begin{equation*}
        G(\theta)-G(\theta^\star) \le \sfN^2(\theta)\,.
    \end{equation*}
\end{lemma}
\begin{proof}
    See \citep[Section 9.6.3]{boyd2004convex}, specifically Eq. (95).
\end{proof}
Clearly, by \eqref{eq:eta-Accurate}, the approximate Newton decrement $\tilde{\sfN}(\theta)$ provides an upper bound on the optimality gap provided that the inverse Hessian is $\eta$-accurate, albeit a looser one.

The key tool in our analysis are the following estimates, due to \citep{pilanci2015newton}.
\begin{lemma}\label{lem:Convergence-Const}
Operate under the conditions of Theorem~\ref{thm:Convergence}. There exist numerical constants $\Lambda,\nu>0$ with $\Lambda<1/16$ such that the following holds. Assume that $0\le \eta<1/2$.
\begin{itemize}
    \item If at an iteration $t$, one has $\sfN(\theta_{t-1}) \ge \Lambda$, then at the next (approximate) Newton step: $G(\theta_t)-G(\theta^\star) \le -ab\nu$, where $a,b>0$ are the parameters for backtracking line search (see Algorithm~\ref{Alg:LineSearch}).
    \item If $\sfN(\theta_{t-1}) \le \Lambda$ then at the next iteration $\sfN(\theta_t)\le \sfN(\theta_{t-1})$.
\end{itemize}
\end{lemma}
\begin{proof}
    See \citep[Lemma 6.4]{pilanci2015newton}.
\end{proof}
Note that Lemma~\ref{lem:Convergence-Const} implies that within at most
\begin{equation}
    T_0 = (G(\theta_0)-G(\theta^\star))/(ab\nu)
\end{equation}
iterations, we are guaranteed to reach $\theta_t$ such that $\sfN(\theta_t)\le \Lambda$; furthermore, once we have achieved that, $\sfN(\theta_{t'})\le \Lambda$ holds for every subsequent iteration $t'\ge t$. 

The following is the main estimate for the remainder of the analysis:
\begin{lemma}\label{lem:ConvProof-Descent}
Operate under the conditions of Theorem~\ref{thm:Convergence}, and assume that $\eta<1/2$. Then
\begin{equation}\label{eq:lem:ConvProof-Descent}
    \sfN(\theta_{t}) \le \frac{(1+\eta)\sfN^2(\theta_{t-1}) + \eta \sfN(\theta_{t-1})}{\left( 1 - (1+\eta)\sfN(\theta_{t-1}) \right)^2} \,.
\end{equation}
\end{lemma}
\begin{proof}
    This result is \citep[Lemma 6.6]{pilanci2015newton}.
\end{proof}
One can further simplify \eqref{eq:lem:ConvProof-Descent} assuming that $\sfN(\theta_{t-1})<1/16$. Coarsely lower bounding the denominator and upper bounding the numerator: 
\begin{equation}\label{eq:proof-Conv}
    \sfN(\theta_t) \le 2\sfN^2(\theta_{t-1}) + 2 \eta \sfN(\theta_{t-1}) 
    \le 
    \begin{cases}
        4\sfN^2(\theta_{t-1})\quad&\textrm{if}\quad \sfN(\theta_{t-1})\ge \eta, \\
        4\eta \sfN(\theta_{t-1}) \quad&\textrm{if}\quad \sfN(\theta_{t-1})\le \eta
    \end{cases} \,,
\end{equation}
provided that $\sfN(\theta_{t-1})<1/16$. 

With the above estimates in hand, we are ready to prove Theorem~\ref{thm:Convergence}.

\begin{proof}
    (Of Theorem~\ref{thm:Convergence}.)
Recall that we assume that $\eta<1/5$, and in particular $\eta<1/2$. By Lemma~\ref{lem:Convergence-Const},
within $T_0$ iterations, the approximate Newton method reaches $\theta_t$ such that $\sfN(\theta_{t'})<1/16$ for all $t'\ge t$. In particular, also $G(\theta)-G(\theta^\star)\le (1+\eta)1/16^2 \le 3/500$.
Let $\eps>0$ be the desired precision; suppose that $\eps<3/500$. 
We consider two cases.

First, assume that $\eta \le \eps <3/500$. Let $T_1(\eps)$ be the smallest integer $t$ such that the dynamic
\begin{equation}
    A_t = 4A_{t-1}^2,\quad A_0=1/16
\end{equation}
satisfies $A_t\le \eps$. It is easy to verify that $A_t= 1/4(1/4)^{2^t}$ (e.g. by induction), hence $T_1(\eps)=O(\log\log(1/\eps))$. We conclude that when $\eps\ge \eta$, $\sfT(\eps)\le T_0 + T_1(\eps)$.

The second case is when $\eps<\eta$, therefore we are in the regime of \eqref{eq:proof-Conv} where decay is linear rather than quadratic. 
Let $T_2(\eps)$ be the smallest integer $t$ such that the dynamic 
\begin{equation}
    B_t = 4\eta B_{t-1},\quad B_0=\eta
\end{equation}
satisfies $B_t\le \eps$. Clearly $B_t=(4\eta)^t B_0$,  hence $T_2(\eps)=O(\frac{\log(\eta/\eps)}{\log(1/\eta)})$ (here we used that $4\eta<1$, by assumption). We conclude that when $\eps<\eta$, $\sfT(\eps)\le T_0+T_1(\eta)+T_2(\eps)$.
 
\end{proof}

\newpage

\section{Auxilliary Technical Results}

\begin{lemma}
    [Sherman-Morrison]
    Let $A\in \RR^{n\times n}$ be invertible, and $u,v\in\RR^n$.
     
    The matrix $A+uv^\T$ is invertible if and only if $1+v^\T A^{-1}u\ne 0$. In that case,
    \begin{equation}
        (A+uv^\T)^{-1} 
        =
        A^{-1} 
        - \frac{A^{-1}uv^\T A^{-1}}{1+v^\T A^{-1}u } \,.
    \end{equation}
    \label{lem:Sherman-Morrison}
\end{lemma}

\begin{lemma}
    [Block matrix inversion]
    Let $A,B,C,D$ be matrices of conforming dimensions. Provided that $A,D$ are both invertible,
    \begin{equation}
    	\MatL
    	A &B \\
    	C &D
    	\MatR^{-1}
    	=
    	\MatL
    	(A-BD^{-1}C)^{-1} &0 \\
    	0 &(D-CA^{-1}B)^{-1}
    	\MatR
    	\MatL
    	I &-BD^{-1} \\
    	-CA^{-1} &I
    	\MatR \,.
    \end{equation}
    \label{lem:BlockInverse}
\end{lemma}

\begin{lemma}\label{lem:LowRankResolvent}
	(Resolvent low-rank perturbation.) 
	Let $A,B,C\succeq 0$ be PSD matrices, and $z>0$. Then
	\[
	\left| \tr C(zI+A)^{-1} - \tr C(zI+B)^{-1} \right| \le \rank(A-B)\frac{\|C\|}{z}\,.
	\]
\end{lemma}
For reference, see for example \citep{bai2010spectral}.

\paragraph*{}
Let $\{\emptyset,\Omega\}=:\m{F}_0\le \m{F}_1\le \m{F}_2 \le \ldots $ be a filtration of $\sigma$-algebras over some given probability space $\Omega$. Recall that a random process $\{X_n\}_{n=1}^\infty$ is a \emph{martingale difference} (adapted to this filtration) if 1) for every $n$, $X_n$ is $\m{F}_n$-measurable; and 2) $\m{F}_{n-1}$-almost surely, $\Expt[X_n|\m{F}_{n-1}]=0$. 
\begin{lemma}
	[Azuma-Hoeffding]
	Suppose that $(X_n,\m{F}_{n})_{n=1}^{N}$ is a martingale difference, and $\{B_n\}_{n=1}^N$ constants. Suppose that for all $1\le n\le N$, almost surely, $|X_n|\le B_n$. Then for all $t\ge 0$,
	\[
	 \Pr(|\sum_{n=1}^N X_n| \ge t) \le 2\exp\left(\frac{-\frac{1}{2}t^2}{\sum_{n=1}^N B_n^2}\right) \,.
	\] 
	\label{lem:Hoeffding}
\end{lemma}

The sub-Gaussian and sub-Exponential norms of a (scalar) random variable are, respectively 
\begin{align}
	\|X\|_{\psi_2} = \inf \left\{\sigma>0 \;:\; \Expt e^{X^2/\sigma^2} \le 2 \right\} \,,\qquad 
	\|X\|_{\psi_1} = \inf \left\{\sigma>0 \;:\; \Expt e^{|X|/\sigma} \le 2 \right\} \,.
	\label{eq:Orlicz-Norms}
\end{align}
We call $X$ sub-Gaussian, resp. sub-Exponential, if $\|X\|_{\psi_2}<\infty$, resp. $\|X\|_{\psi_1}<\infty$. 

For a random  vector $r$, we define $\|r\|_{\psi_i}=\sup_{\|v\|=1}\|v^\T r\|_{\psi_i}$. That is, it is the largest $\psi_i$-norm of a one-dimensional projection of $r$.

\begin{lemma}
	The following holds.
	\begin{enumerate}
		\item $\|\cdot\|_{\psi_i}$ is a norm on the subspace of random variables $X$ such that $\|X\|_{\psi_i}<\infty$.
		\item If $X,Y$ are sub-Gaussian (possible statistically dependent), then $XY$ is sub-Exponential, with $\|XY\|_{\psi_1} \le \|X\|_{\psi_2}\|Y\|_{\psi_2}$.
		\item If $|Y|\le B$ almost surely for $B\ge 0$ constant, $\|XY\|_{\psi_i} \le B\|X\|_{\psi_i}$.
		\item (Centralization.) For any $X$, $\|X-\Expt[X]\|_{\psi_i} \le \|X\|_{\psi_i}$.
	\end{enumerate}
\end{lemma}
We refer to the book \citep{vershynin2018high} for reference.

\begin{lemma}
	[Bernstein]
	Suppose that $(X_n,\m{F}_{n})_{n=1}^{N}$ is a martingale difference. Then for all $t\ge 0$,
	\[
	\Pr(|\sum_{n=1}^N X_n| \ge t) \le 2\exp\left(-c\min\{ \frac{t}{\max_{n=1,\ldots,N}\|X_n\|_{\psi_1}}, \frac{t^2}{\sum_{n=1}^N \|X_n\|_{\psi_1}^2} \}\right) \,.
	\] 
	\label{lem:Bernstein}
\end{lemma}

\paragraph*{}
Next, we cite the following version of the Hanson-Wright inequality, see \citep{vershynin2018high}.
\begin{lemma}
	[Hanson-Wright]
	Suppose that $r$ has independent sub-Gaussian entries, with $\Expt[r_i]=0,\Expt[r_i^2]=1$, $\max_{i}\|r_i\|_{\psi_2} \le \SGconst$. 
	For any matrix $A$ and $t\ge 0$,
	\begin{align}\label{eq:HansonWright}
		\Pr(|r^\T A r - \tr(A)|\ge t) \le \exp\left( -c\min\{\frac{t}{\SGconst^2\|A\|}, \frac{t^2}{\SGconst^4\|A\|_F^2}\} \right)\,,
	\end{align}
	where $c>0$ is a universal constant.
	\label{lem:HansonWright}
\end{lemma}
Lastly, 
the following lemma collects, for convenience, some  properties satisfied by any isotropic random vector $r$ that satisfies a concentration inequality of the form \eqref{eq:HansonWright}.
\begin{lemma}\label{lem:HW-Properties}
	Suppose $r$ is istropic $\Expt[r]=0,\Expt[rr^\T]=I$, and satisfies \eqref{eq:HansonWright} for some $\SGconst>0$.
	\begin{enumerate}
		\item For any matrix $A$ (independent of $r$), $\Expt[|r^\T A r - \tr(A)|^2] \lesssim \SGconst^4 \|A\|_F^2$.
		\item For any vector $v$ (independent of $r$), $\|v^\T r\|_{\psi_2}\lesssim \SGconst\|v\|$. That is, $r$ is sub-Gaussian, with $\|r\|_{\psi_2}\lesssim \SGconst$.
	\end{enumerate}
\end{lemma}
\begin{proof}
	Start with the first item.
	Write $\Expt[|r^\T A r - \tr(A)|^2] =\int_0^\infty\Pr(|r^\T A r - \tr(A)|^2>s)ds=2\int_0^\infty t \Pr(|r^\T A r - \tr(A)|^2>t)dt$. Using \eqref{eq:HansonWright}, we crudely bound
	\begin{align*}
		\int_0^\infty t \Pr(|r^\T A r - \tr(A)|^2>t)dt
		&\le \int_0^\infty t \exp(-c\frac{t^2}{\SGconst^4\|A\|_F^2})dt + \int_0^\infty t \exp(-c\frac{t}{\SGconst^2\|A\|})dt 
		\lesssim \SGconst^4\|A\|_F^2 \,.
	\end{align*}
	As for the second item, note that $\|vv^T\|=\|vv^T\|_F=\|v\|^2$. Hence,
	\[
	\Pr(|v^\T r|\ge t)=\Pr(r^\T vv^\T r \ge t^2) \lesssim \exp(-\min\{\frac{t^2}{\SGconst^2\|v\|^2},\frac{t^4}{\SGconst^4\|v\|^4}\}) \,.
	\]
	That is, $\|v^\T r\|$ has a Gaussian tail with variance proxy $\SGconst^2\|v\|^2$.
\end{proof}

\end{document}